\documentclass[11pt]{article}
\usepackage{amsfonts,amsmath,geometry,xcolor}
\usepackage{amsthm}

\usepackage{txfonts}
\usepackage{amssymb}
\usepackage{hyperref}

\hypersetup{
colorlinks=true, linkcolor=blue, anchorcolor=blue, citecolor=blue}
\allowdisplaybreaks[4]

\newtheorem{theorem}{Theorem}[section]
\newtheorem{lemma}[theorem]{Lemma}

\newtheorem{corollary}[theorem]{Corollary}
\newtheorem{definition}[theorem]{Definition}
\newtheorem{remark}[theorem]{Remark}

\begin{document}

\thispagestyle{empty}

\title{
Partial regularity for local minimizers of functionals with double-phase Orlicz growth}
\author{
Wenrui Chang
\footnote{
School of Mathematics and Statistics,
Beijing Jiaotong University,
Beijing 100044, China.
Email:wenruichang@bjtu.edu.cn}
\quad
Shenzhou Zheng
\footnote{
School of Mathematics and Statistics,
Beijing Jiaotong University,
Beijing 100044, China.
Email:shzhzheng@bjtu.edu.cn}}
\maketitle

\begin{abstract}
In this paper, we consider the local minimizers to a class of non-autonomous functional with double-phase Orlicz growth:
\begin{equation*}
u \in W^{1,1}(\Omega; \mathbb{R}^N) \mapsto  \int_\Omega \Big( {\varphi_1}\big(|Du|_{\mathbb{A}^u}\big) + a(x)\,{\varphi_2} \big(|Du|_{\mathbb{A}^u}\big) \Big)\, dx,
\end{equation*}
where $|Du|_{\mathbb{A}^u} := \big\langle \mathbb{A}(x,u)\,Du, Du \big\rangle^{\frac{1}{2}}$ with $\mathbb{A}(x,u) = \big\{A^{\alpha \beta}_{ij}(x, u)\big\}_{i,j = 1\cdots N}^{\alpha,\beta = 1\cdots n}$ as the uniformly elliptic bounded symmetric tensor field, and $\varphi_k(\cdot)$ for $k=1,2$ are two different $N$-functions.
We prove the partial regularity of their local minimizers if the continuity modulus of $a(\cdot)$ and $\varphi_k(\cdot)$ are fit for the gap conditions, and $\mathbb{A}(x,u)$ meets an optimal regularity in $(x,u)$.
This is an extension from the single Orlicz growth functional to the double-phase one based on an essential improvement of several important inequalities.

\vskip 0.2cm
\noindent{\bf MR(2020) Subject Classification:}
{Primary: 35B65; 35J15; Secondary: 49N60}
\end{abstract}

\section{Introduction}

Let $u: \Omega \to \mathbb{R}^N$  be a vector-valued function with bounded open subset $\Omega \subset \mathbb{R}^n$. We devote this paper to the partial regularity of local minimizers to a class of non-autonomous functional with double-phase Orlicz growth as follows:
\begin{equation}\label{main_funl}
\mathcal{F}_{\mathbb{A},{\varphi_1},{\varphi_2}}(u ; \Omega): = \int_\Omega \Big( {\varphi_1}\big(|Du|_{\mathbb{A}^u}\big) + a(x)\,{\varphi_2} \big(|Du|_{\mathbb{A}^u}\big) \Big)\, dx,
\end{equation}
where $a: \Omega \to [0,+\infty)$ is of bounded function, $|Du|_{\mathbb{A}^u} := \big\langle \mathbb{A}(x,u)\,Du, Du \big\rangle^{\frac{1}{2}}$ which the tensor field $\mathbb{A}(x,u) = \big\{A_{i j}^{\alpha \beta}(x, u)\big\}_{i,j = 1\cdots N}^{\alpha,\beta = 1\cdots n}$ satisfies the symmetry $A_{i j}^{\alpha \beta}(x, u)=A_{j i}^{\beta \alpha}(x, u)$ and there exists a  constant $\Lambda \geq 1$ such that
\begin{equation}\label{cond_A}
\Lambda^{-1}|\xi|^2
\leq	\sum_{i,j =1}^{N} \sum_{\alpha,\beta  =1}^{n}A_{i j}^{\alpha \beta}(x, u)\, \xi_\alpha^i \xi_\beta^j
\leq	\Lambda|\xi|^2 \qquad  \forall (x,u,\xi) \in \Omega\times\mathbb{R}^{N}\times\mathbb{R}^{N \times n}.
\end{equation}
The functions ${\varphi_k}\in C^1[0,+\infty) \cap C^2(0,+\infty)$ for $k=1,2$ are two $N$-functions
satisfying
\begin{equation}\label{cond_p_q}
0 < p - 1
\leq	\frac{t\varphi^{\prime\prime}_k(t)}{\varphi^{\prime}_k(t)}
\leq	q - 1 < \infty \qquad  \forall t > 0,
\end{equation}
where $p,q$ are two constants. We also assume that ${\varphi_2} \circ \varphi_1^{-1}(\cdot)$ still is $N$-function, and denote this as ${\varphi_1} \prec {\varphi_2}$.

As is well known, the Morrey and Campanato regularity estimates for higher-order gradient to local vector-valued minimizers of the quadratic functional
$\int_\Omega \big\langle \mathbb{A}(x)\,Du, Du \big\rangle\, dx $
can be derived while the coefficient tensor $\mathbb{A}(x)$ is sufficiently smooth with uniform ellipticity and boundedness, see \cite[Chapter III]{Gi83}. De Giorgi in 1968 presented a counterexample to show that there exists an unbounded vector-valued minimizer for the quadratic uniform elliptic variational integral with its coefficient $A(x)\in L^{\infty}$, it yields that we could not generally expect \emph{everywhere regularity} for the vector-valued minimizers of such functionals. As a modification of De Giorgi's example, Hildebrandt-Widman \cite{HiW75} constructed a counterexample to show that there exists a discontinuous vector-valued minimizer for  the functional
$ u \mapsto \int_\Omega a\big(|u|\big)\, |Du|^2\, dx $
with smooth functions $a(\cdot)$.
To address this issue, Morrey \cite{Mo67} introduced the concept of the  so-called \emph{partial regularity} for minimizers/solutions which indicates their regularity results of minimizers/solutions being outside a negligible set. From then on, the partial regularity for the vectorial minimizers of functionals or the solutions of partial differential systems has been popularly studied. Generally speaking, the partial regularity asserts  pointwise regularity for minimizers/solutions, in an open subset whose complement is negligible.
Regarding our topic, Fusco-Hutchinson \cite{FuH89} explored the partial $C_\mathrm{loc}^{1,\gamma}$-regularity of minimizers for functionals of the form
\begin{equation}\label{A(x,u)-p}
\mathcal{F}_{\mathbb{A},p}(u ; \Omega) := \int_\Omega |Du|_{\mathbb{A}^u}^p\, dx
\end{equation}
for $p \geq 2$, while  Acerbi-Fusco \cite{AcF89} extended this result to that for $1 < p < 2$. Here, we would like to remark that Uhlenbeck \cite{Uh77} proved the full $C_\mathrm{loc}^{1,\gamma}$-regularity for vector-valued minimizers of functionals of the $p$-growth form
\begin{equation*}
\mathcal{F}_{p}(u ; \Omega): =\int_\Omega |Du|^p\, dx
\end{equation*}
for $p \geq 2$, and subsequently Tolksdorf \cite{To83} extended the Uhlenbeck result to that for $1<p<2$. Later, Di~Gironimo-Esposito-Sgambati in \cite{DiES04} proved the partial Morrey $L^{2,\lambda}$-regularity of gradients to the local minimizers of the following quadratic functional
\begin{equation*}
\mathcal{F}_{\mathbb{A},2}(u ; \Omega) :=\int_{\Omega} |Du|_{\mathbb{A}^u}^2\,dx,
\end{equation*}
while $\mathbb{A}(x, u)$ is only assumed to be vanishing mean oscillation (VMO) in the $x$-variable and uniformly continuous in the $u$-variable.

Nowadays, several classes of integral functionals or nonlinear elliptic equations concerning non-standard growth have been popular issues, such as $(p,q)$-growth, $p(x)$-growth, Orlicz growth, multi-phase growth, etc. These problems have attracted considerable attention due to their relevance in nonlinear elasticity, electrorheological fluids, and image processing. In particular, the $(p,q)$-growth problem explains the mathematical model for composite materials composed of two different power hardening exponents. Let us now review some progresses regarding elliptic problems with double-phase growths. Marcellini was the first to get the Lipschitz regularity of local minimizers of the double-phase energy functional in \cite{Ma89}, and also proved the corresponding regularity for the double-phase elliptic equations in \cite{Ma91}.
Since then, the double-phase elliptic problem has been an active field that has attracted the attention of researchers over the past decade. Esposito-Leonetti-Mingione \cite{EsLM04} proved the higher gradient integrability of local minimizers to the energy functionals $\int_\Omega f(x,Du) dx$ with $(p,q)$-growth while the integrand function $f(x, Du)$ is $\alpha$-H\"older continuous in $x$ under the gap condition $1\le \frac qp < \frac{n+\alpha}n$. With the $(p, q)$-growth functionals modeled by
\begin{equation}\label{double}
\mathcal{F}_{p,q}(u ; \Omega):  =\int_\Omega\Big(|Du|^{p} + a(x)\,|Du|^{q}\Big)\, dx,	
\end{equation}
Colombo-Mingione \cite{CoM15} showed the $C_\mathrm{loc}^{1,\gamma}$-regularity to the local minimizers $u\in W^{1,p}(\Omega)$ of functional $\mathcal{F}_{p,q}(u ; \Omega)$, under assumptions of
\begin{equation}\label{double_gap}
a(x) \in C^{0, \alpha}(\Omega)\quad \text{for} \ \alpha \in (0,1], \quad \text{and} \ \ \frac{q}{p} < 1 + \frac{\alpha}{n};
\end{equation}
and they \cite{CoM15-1} further got the $C_\mathrm{loc}^{1,\gamma}$-regularity for the bounded minimizers of  functional \eqref{double} by weakening the gap assumption to $q\le p+\alpha$ while $a(x)\in C^{\alpha}$. It would be also remarked that Eleuteri-Marcellini-Mascolo \cite{ElMM16} gave the Lipschitz regularity of local
 minimizers to the functional $\int_\Omega g(x,|Du|) dx$ with $(p, q)$-growth just for $|Du|\rightarrow\infty $ without further structure conditions imposed on the integrand.
Recently, De~Filippis-Mingione \cite{DeM23} established the local Schauder theory of distributional solutions $u\in W^{1,1}(\Omega)$ to the nonuniform elliptic problems under the different gap conditions; while  they \cite{DeM25} further established the Schauder estimates for the nonuniform elliptic problems under the optimal assumptions on growth of elliptic ratio.

On the other hand, there are recently a lot of research of partial regularity for the vector-valued minimizers of double-phase elliptic functionals or nonlinear elliptic systems with double-phase growths. For instances, Coscia-Mingione \cite{CoM99} proved $C^{1,\gamma}$-regularity of the local vector-valued minimizers to the $p(x)$-growth functionals
$\mathcal{F}_{p(x)}(u ; \Omega) := \int_\Omega |Du|^{p(x)}\, dx$ while $p(x) > 1$ is H\"older continuous. Bildhauer-Fuchs \cite{BiF05} considered the $(p,q)$-growth functionals with non-autonomous form,
and gave the partial regularity of the local minimizers of non-autonomous functionals based on the so-called blow-up approach. We would particularly like to mention that Ragusa-Tachikawa-Takabayashi \cite{RaTT13} explored the partial H\"older continuity of gradients to the local minimizers of the $p(x)$-growth functional as follows:
\begin{equation*}
\mathcal{F}_{\mathbb{A},p(x)}(u ; \Omega) := \int_\Omega |Du|_{\mathbb{A}^u}^{p(x)}\, dx,
\end{equation*}
while $p(x)\geq 2$ is sufficiently smooth and $\mathbb{A}(x,u)$ is H\"older continuous in $x$ and smooth in $u$. For more literature regarding the $p(x)$-growth, we can refer to \cite{Ta14, GoRS20, GoS22} and the references therein.
We would particularly like to point out that Diening-Stroffolini-Verde \cite{DiSV09} proved everywhere $C_\mathrm{loc}^{1,\gamma}$-regularity of the local minimizers to the Orlicz-growth functional $\mathcal{F}_{\varphi}(u ; \Omega) := \int_\Omega {\varphi}\big(|Du|\big)\, dx$  with the \emph{Uhlenbeck structure} while ${\varphi}(\cdot)$ is an $N$-function. Giannetti-Passarelli di Napoli-Tachikawa \cite{GiPT17} gave the partial regularity of minimizers to the following non-autonomous integral functional
\begin{equation*}
\mathcal{F}_{\mathbb{A},{\varphi}}(u ; \Omega) := \int_\Omega {\varphi}\big(|Du|_{\mathbb{A}^u}\big)\, dx,
\end{equation*}
 under suitable regularity on the boundary datum provided that $\mathbb{A}(x,u)$ is a uniformly elliptic, bounded and continuous function. For more functionals with Orlicz-growths, we can refer to \cite{ByO20, GoSS22, Ir25}. In addition, Ok \cite{Ok18p} proved the partial H\"older regularity to such elliptic systems with $ (p,q) $-growth or Orlicz growth based on the \emph{$ \mathcal{A} $-harmonic approximation}, while Tachikawa \cite{Ta24} proposed the partial regularity to the double-phase functionals with variable exponents.
 We refer the readers to \cite{De22, GmK24, OkSS25} for more $(p,q)$-growth functionals. As borderline double-phase functionals with the model as follows:
\begin{equation}\label{p-plog}
\mathcal{F}_{\log}(u ; \Omega) := \int_\Omega \Big( |Du|^p + a(x) |Du|^p \log \big(e + |Du|\big) \Big) \, dx,
\end{equation}
one can weaken the modulus of continuity $\omega_a(\cdot)$ on $a(\cdot)$ rather than its H\"older continuity, since $a(\cdot)$ is used to calibrate the size of the transition between $|Du|^p$ and $|Du|^p \log \big(e + |Du|\big)$. Baroni-Colombo-Mingione \cite{BaCM16} proved that every local minimizer of the functional \eqref{p-plog} achieves local H\"older continuity under assumption of
\begin{equation}\label{p-plog-gap}
\limsup_{r \to 0^+}  \omega_a(r)\,\log\Big(\frac{1}{r}\Big) < \infty,
\end{equation}
and its gradient is locally H\"older continuous when $\omega_a(r) = r^{\,\beta}$ with $\beta > 0$. Also, Ragusa-Tachikawa \cite{RaT24} gave partial regularity for borderline double-phase functionals with variable exponent growths.

Motivated by the aforementioned progresses, this article is committed to the partial regularity of such functional \eqref{main_funl} with double-phase Orlicz growth. In fact, we are in part inspired by the following recent developments. Byun-Oh in \cite{ByO20} showed the H\"older continuity of minimizers, based on  the De Giorgi iteration, to the scalar functionals with the model as
\begin{equation}\label{2_orlicz}
\mathcal{F}_{{\varphi_1},{\varphi_2}}(u ; \Omega):  =\int_\Omega\Big({\varphi_1}\big(|Du|\big) + a(x)\,{\varphi_2} \big(|Du|\big)\Big)\, dx
\end{equation}
under the gap condition
\begin{equation}\label{2_orlicz_gap}
\limsup_{r \to 0^+}\omega_a(r)\, \frac{({\varphi_2} \circ \varphi_1^{-1})(r^{-n})}{r^{-n}} < \infty,
\end{equation}
while Baasandorj-Byun \cite{BaB25}  provided  a unified method to show comprehensive regularity results of the functionals involving Orlicz multi-phase under some optimal conditions. On the other hand, Harjulehto-H\"ast\"o-Toivanen \cite{HaHT17} gave the Harnack inequality and local H\"older continuity of quasiminimizers of the functionals under general growth conditions as
\begin{equation}\label{g_orlicz}
\mathcal{F}_{\varPhi}(u ; \Omega) := \int_\Omega \varPhi(x, |Du|)\, dx,
\end{equation}
which includes the special settings of standard, variable exponent and double phase growth by unifying the gap condition instead of \eqref{double_gap}, \eqref{p-plog-gap} and \eqref{2_orlicz_gap}. Moreover, H\"ast\"o-Ok \cite{HaO22a} proved the local $C^{1,\alpha}$-regularity theory of local minimizers to the above functional in a universal way, while in \cite{HaO22b} they
derived the maximal local regularity results to the non-autonomous partial differential problems without recourse to special function
structure and without assuming Uhlenbeck structure. Later, they in \cite{HaO23} further developed the regularity theory of such non-autonomous problem under the a priori assumptions.

Our problem under consideration also includes both the functionals $\mathcal{F}_{p,q}(u ; \Omega)$ in \eqref{double} and $\mathcal{F}_{\log}(u ; \Omega)$ in \eqref{p-plog} as special settings. We also notice that Tachikawa \cite{Ta24} proved the partial H\"older continuity of gradient to the local minimizers of double-phase functional with variable exponent growth
\begin{equation*}
\mathcal{F}_{\mathbb{A},(p(x),q(x))}(u ; \Omega) := \int_\Omega \Big( |Du|_{\mathbb{A}^u}^{p(x)} + a(x)\, |Du|_{\mathbb{A}^u}^{q(x)} \Big)\, dx,
\end{equation*}
under assumptions that $ a(\cdot), p(\cdot), q(\cdot) \in C^{0,\,\beta}(\Omega)$, $\mathbb{A}(\cdot, \cdot)\in C^{0,\alpha}(\Omega, \mathbb{R}^{Nn \times Nn})$ and the gap condition
\begin{equation*}
1 \leq \sup_{x \in \Omega}\frac{q(x)}{p(x)} < 1 + \frac{\alpha}{n}.
\end{equation*}
Unfortunately, the model under our consideration does not include such functional like $\mathcal{F}_{\mathbb{A},(p(x),q(x))}$, since the $N$-function $\varphi(|\cdot|)$ can not describe variable exponent growth of the norm with $|\cdot|^{p(x)}$.
We now recall the definition of the local minimizer to functional $\mathcal{F}_{\mathbb{A},{\varphi_1},{\varphi_2}}(u ; \Omega)$ in the weak sense.

\begin{definition}\label{def-min}
The function $u \in W^{1,1}(\Omega ;{\mathbb{R}^N})$ is said to be a local minimizer of the functional \eqref{main_funl}, if $\mathcal{F}_{\mathbb{A},{\varphi_1},{\varphi_2}}(u ; \Omega) < \infty$ and for any $\Omega' \Subset \Omega$ there holds
\begin{equation*}
\mathcal{F}_{\mathbb{A},{\varphi_1},{\varphi_2}}(u ; \Omega')
\leq	\mathcal{F}_{\mathbb{A},{\varphi_1},{\varphi_2}}(u + \phi; \Omega')\qquad  \forall \phi \in W_0^{1,1}(\Omega' ;{\mathbb{R}^N}).
\end{equation*}
\end{definition}

The goal of this paper is to investigate the partial regularity of such functional \eqref{main_funl} with double-phase Orlicz growth  in the vectorial setting. It should be noted that the term $|Du|_{\mathbb{A}^u}$ in the functional \eqref{main_funl} depends explicitly on $(x,u)$ as well as on $Du$. To this end, it is necessary to impose regularity assumptions on $\mathbb{A}(\cdot,\cdot)$ and $a(\cdot)$: namely, that there exists a positive constant $L \geq 1$ and two non-decreasing concave moduli of continuity $0\le \omega_{\mathbb{A}}(\cdot), \omega_a(\cdot)\le 1$ with $\omega_{\mathbb{A}}(0)=\omega_a(0)=0$ such that for $\forall x,y \in \Omega; u,v \in \mathbb{R}^N$ one has
\begin{eqnarray}\label{cond_continuous}
  \big| \mathbb{A}(x,u) - \mathbb{A}(y,v) \big| \leq L\, \omega_{\mathbb{A}}\big(|x-y|+|u-v|\big),\qquad
  \big|a(x) - a(y)\big| \leq 	L\,\omega_a\big(|x-y|\big).
\end{eqnarray}

%
With the above notations in hand, we are now to state our main result as follows.

\begin{theorem}\label{thm_1}
Let ${\varphi_1}, {\varphi_2} \in C^1[0,+\infty) \cap C^2(0,+\infty)$ satisfy \eqref{cond_p_q} with ${\varphi_1} \prec {\varphi_2}$, and $u \in W^{1,1}(\Omega,\mathbb{R}^N)$ be a local minimizer of the functional \eqref{main_funl}. Assume that $\mathbb{A}(\cdot,\cdot)$  with \eqref{cond_A} and $a(\cdot)$ are uniformly continuous with the modulus of continuity \eqref{cond_continuous} satisfying the gap condition
\begin{equation}\label{gap_1}
\limsup_{r \to 0^+}\,\omega_a(r)\, \frac{\big({\varphi_2} \circ \varphi_1^{-1}\big)\left(r^{-n}\right)}{r^{-n}} =0.
\end{equation}
Then we have $u \in C_\mathrm{loc}^{0,\alpha}(\Omega_0, \mathbb{R}^N)$ for any $\alpha \in (0,1)$, where the regular set $\Omega_0 \subset \Omega$ is defined by
\begin{equation}\label{def_Omega_0}
\Omega_0:=\left\{ y \in \Omega:
\liminf _{r \rightarrow 0^{+}} \fint_{B_r(y)} |u - \langle u \rangle_{B_r(y)}| \,dx  <  \delta_0
\right\},
\end{equation}
with the small constant $\delta_0 > 0 $. Moreover, there exists a  small
$\epsilon   \in (0, n-p]$ such that $\mathcal{H}^{n - p-\epsilon}\big(\Omega \setminus \Omega_0\big) = 0$,
with $\mathcal{H}^{s}(\mathcal{D})$ as the $s$-dimensional Hausdorff measure of $\mathcal{D}$.
\end{theorem}

It would be particularly worth noting that the gap condition \eqref{gap_1} does not capture the endpoint of gap. For example, it excludes the possibility of  $q/p = 1 + \beta/n$ while we consider the usual double-phase energy functional for $\omega_a(t) = t^{\,\beta}$, ${\varphi_1}(t) = t^p$ and ${\varphi_2}(t) = t^q$. In fact, Baroni-Colombo-Mingione \cite{BaM18} proved that sharp regularity results can also be achieved for the double-phase functional $\mathcal{F}_{p,q}$ while allowing the full range $q/p \leq 1 + \alpha/n $.
To complete our result, we are devoted to presenting the second result to describe an optimal H{\"o}lder regularity for this borderline setting.

\begin{theorem}\label{thm_2}
Let ${\varphi_1}, {\varphi_2} \in C^1[0,+\infty) \cap C^2(0,+\infty)$ satisfy \eqref{cond_p_q} with ${\varphi_1} \prec {\varphi_2}$, and $u \in W^{1,1}(\Omega,\mathbb{R}^N)$ be a local minimizer of the functional \eqref{main_funl}. Assume that $\mathbb{A}(\cdot,\cdot)$ with \eqref{cond_A} is uniformly continuous, and $a(\cdot)$ is H\"older continuous with the modulus of continuity \eqref{cond_continuous} satisfying the gap condition
\begin{equation}\label{gap_2}
\limsup_{r \to 0^+}\,\omega_a(r)\, \frac{\big({\varphi_2} \circ \varphi_1^{-1}\big)\left(r^{-n}\right)}{r^{-n}} < \infty.
\end{equation}
Then $u \in C_\mathrm{loc}^{0,\alpha}(\Omega_0, \mathbb{R}^N)$ for any $\alpha \in (0,1)$ with $\Omega_0 \subset \Omega$ defined as \eqref{def_Omega_0}.
\end{theorem}

\begin{remark}
We make the following two observations regarding Theorems~\ref{thm_1} and~\ref{thm_2}:

\begin{itemize}
\item We point out that the gap condition \eqref{gap_2} imposed on ${\varphi_1}, {\varphi_2}$ in Theorem~\ref{thm_2} is weaker than \eqref{gap_1} in Theorem~\ref{thm_1}. Here, we also emphasize that the H{\"o}lder continuity of $a(x)$ is assumed in Theorem \ref{thm_2} instead of the continuity of $a(x)$  in Theorem \ref{thm_1}. We also refer to \cite{BaB25} for the gap constraint of functionals concerning double orlicz phases.
\item For the sake of simplicity, we only confine  $\mathbb{A}(x,u)$  as continuous form in the $x$-variable in Theorems~\ref{thm_1} and~\ref{thm_2}. Indeed, we can make use of the argument from \cite{DiES04} to get the same results by relaxing   $\mathbb{A}(x,u)$ as the so-called vanishing mean oscillation with respect to the $x$-variable.
\end{itemize}
\end{remark}

To realize the partial $C_\mathrm{loc}^{1,\gamma}$-regularity of local minimizers, it is necessary to impose higher regularity assumptions on the given data. More precisely, if ${\varphi''_1}(\cdot), {\varphi''_2}(\cdot)$ and $\mathbb{A}(\cdot,\cdot), a(\cdot)$ are always  H\"older continuous for any variables,  then we can attain the partial $C_\mathrm{loc}^{1,\gamma}$-regularity for local minimizers of the functional \eqref{main_funl}. To this end, we need to assume that
$\varphi_k^{\prime \prime}(t)\, (k = 1,2)$ satisfy the H\"older-type condition (cf. \cite[Assumption 2.2]{DiSV09}) as follows: for any $t>0$ and $s \in \mathbb{R}$ with $|s|<\frac{1}{2} t$, there exists  $\beta \in (0,1)$ and the constant $c>0$ such that
\begin{equation}\label{cond_Holder_Phi}
\left|\varphi_k^{\prime \prime}(s+t)-\varphi_k^{\prime \prime}(t)\right| \leq c\, \varphi_k^{\prime \prime}(t)\left(\frac{|s|}{t}\right)^\beta,
\end{equation}
and $\mathbb{A}(\cdot,\cdot)$ and $a(\cdot)$, respectively, are uniformly $\beta$-H\"older continuous in all variables with form
\begin{equation}\label{cond_Holder_A_a}
\omega_{\mathbb{A}}(t) = \omega_a(t) = \min\{1,t^{\,\beta}\}\quad \forall t > 0.
\end{equation}
With further regularity assumptions imposed on the datum, we are now in a position to state the partial $C_\mathrm{loc}^{1,\gamma}$-regularity result of local minimizers of the functional $\mathcal{F}_{\mathbb{A},{\varphi_1},{\varphi_2}}(u ; \Omega)$ as follows.

\begin{theorem}\label{thm_3}
Let ${\varphi_1}, {\varphi_2} \in C^1[0,+\infty) \cap C^2(0,+\infty)$ satisfy \eqref{cond_p_q} \eqref{cond_Holder_Phi} with ${\varphi_1} \prec {\varphi_2}$, and $u \in W^{1,1}(\Omega,\mathbb{R}^N)$ be the local minimizer of the functional \eqref{main_funl}. Assume that $\mathbb{A}(\cdot,\cdot)$ with \eqref{cond_A} and $a(\cdot)$, respectively, are uniformly H\"older continuous with H\"older continuity modulus \eqref{cond_Holder_A_a} satisfying  the gap condition \eqref{gap_2}. Then we have $Du \in C_\mathrm{loc}^{0,\gamma}(\Omega_0, \mathbb{R}^{N \times n})$ for some $\gamma \in (0,1)$, where  the open subset $\Omega_0 \subset \Omega$ is defined as in \eqref{def_Omega_0}.
\end{theorem}

\begin{remark}
 As a special case, while $F(x, Du) = {\varphi_1}\big(|Du|_{\mathbb{A}}\big) + a(x){\varphi_2} \big(|Du|_{\mathbb{A}}\big)$ with $\mathbb{A} = \mathbb{A}(x)$ as being explicitly independent of $u$, this is a class of functional with Uhlenbeck structure. Then the full regularity of minimizers to such functional \eqref{main_funl} follows immediately. On the other hand, inspired by the literatures \cite{HaO23, BaB25}, we leave for future consideration the relaxation of the gap conditions \eqref{gap_1} or \eqref{gap_2} to show the partial regularity for the vector-valued minimizers of functional \eqref{main_funl} under  a-priori boundedness or higher integrability assumptions on the solution.
\end{remark}

To the best of our knowledge, the partial regularity of such vectorial double-phase Orlicz problem was still not studied, although it has been investigated for those functionals with variable exponent growth  \cite{RaTT13}, Orlicz growth \cite{GiPT17} and double-phase variable exponent growth \cite{Ta24}.  A key point of this paper is regarded as the vectorial counterpart of the scalar Orlicz phase problems from Baasandorj-Byun \cite{BaB25}, so that it only gets the partial regularity of local minimizers. Note that the classical De~Giorgi-Nash-Moser iteration \cite{BaB25} is invalid in the vectorial setting. To this end, the main strategy of our proof is to focus on the Morrey and Campanato estimates of gradients of minimizers under the constraint of small excess energy, which is  based on the perturbation method via double-phase harmonic approximation and  some essential improvements of several inequalities. As is a well-known fact, the $p$-harmonic approximation for $p$-growth  and ${\varphi}$-harmonic approximation for more general Orlicz growth have been considered in \cite{DuM04} and \cite{DiSV12}, respectively; while Baroni-Colombo-Mingione also extended this argument to that for double-phase growth problems, see \cite[Lemma 5.1]{BaM18}. Here, we make use of a refined version from  \cite[Lemma 4.13]{HaO23} to deal with our double-phase functional. To this end, let us freeze the coefficients $a(x)$ and $\mathbb{A}(x,u)$ to $\bar{a}_r = \inf_{B_r(x_0)} a(x)$ and $\bar{\mathbb{A}}_r^u = \mathbb{A}(x_0, \langle u \rangle_{B_r})$ on the ball $B_r(x_0)\Subset \Omega$. Then, we make several comparison estimates among the local minimizer of the original functional and the minimizers of locally regularized functionals with frozen coefficients.
This idea is inspired by the paper of Kuusi-Mingione \cite{KuM18} by introducing nonlinear potential theory for the vectorial setting to establish the comparison estimates in the $L^1$-sense independent of $\bar{a}_r$. This not only makes the subsequent iterative arguments more convenient compared with \cite{Ta24}, but also it yields a more refined form of the singular set than in \cite{GiPT17}.  Finally, we apply the isomorphism relation of the Campanato space and the H\"older space (cf. Lemma \ref{lem_isomorphism}) to attain an optimal partial H\"older continuity of local minimizers and the partial H{\"o}lder continuity of their gradients.

The remainder of this paper is organized as follows. We devote Section 2 to some basic notations concerning function spaces and some related lemmas. In Section 3, we give a higher integrability of gradients to the local minimizers of functional \eqref{main_funl}, and show comparison estimates by way of harmonic approximation. In Section 4, we further show  a few of  comparison estimates in the $L^1$-sense between the local minimizers of the original functional. Finally, we focus Section 5 on the proofs of the main results.

\section{Preliminaries}
\setcounter{equation}{0}
\setcounter{theorem}{0}

Throughout this paper, we use $c$ to denote a generic positive constant that may vary from line to line, where some important constants are distinguished by subscripts (e.g., $c_1, c_2$). To distinguish specific instances, we usually emphasize a relevant parameter depending only on prescribed quantities with parentheses
if necessary, for instance, $c = c(n, N, \Lambda, \dots)$ indicates that $c$ depends on $ n, N, \Lambda ,\cdots$.  We also denote by $a \lesssim b$ if there exists a constant $c > 0$ such that $a \leq cb$, and $a \gtrsim b$ means $a \geq cb$; and the equivalence is denoted by the symbol $a \thickapprox b$ if there exist two universal constants $c_1, c_2 > 0$ such that $c_1a \leq b \leq c_2a$.
In the context, let us denote by $B_r(x_0)$ an open ball centered at $x_0 \in \mathbb{R}^n$ with radius $r > 0$. If the center is clear, we shall omit the center point by writing $B_r = B_r(x_0)$. For any $f\in L^1(\mathcal{D},\mathbb{R}^m)$ defined on the measurable set $\mathcal{D} \subset \mathbb{R}^n$ with positive measure, we denote by $\langle f \rangle_{\mathcal{D}}$ the integral average of $f(x)$ over the subset  $\mathcal{D}$, i.e.
\begin{equation*}
\langle f \rangle_{\mathcal{D}} =
\fint_{\mathcal{D}} f \,dx =
\frac{1}{|\mathcal{D}|} \int_{\mathcal{D}} f \,dx.
\end{equation*}

Let us first recall the notation of \emph{$N$-function} and summarize their related properties. The function ${\phi } : [0, \infty) \to [0, \infty)$ is said to be an \emph{$N$-function} if ${\phi } (\cdot)$ is non-decreasing convex function with the following conditions:
\begin{equation*}
{\phi } (0)=0,\quad \lim_{t \to \infty} \phi(t) = \infty, \quad \lim_{t \to 0^+} \frac{{\phi } (t)}{t} = 0, \quad \lim_{t \to \infty} \frac{{\phi } (t)}{t} = \infty.
\end{equation*}
\emph{Its inverse} of the $N$-function ${\phi }$ is shown as
\begin{equation*}
{\phi } ^{-1}(t) = \inf\Big\{s > 0 : {\phi } (s) > t\Big\},
\end{equation*}
then we immediately get that ${\phi } ^{-1}(t)$ is non-decreasing, concave function.

For the sake of convenience, we write by $\mathbf{\Phi}_{p,q}$ the set of all $N$-functions ${\phi } \in C^1[0,+\infty) \cap C^2(0,+\infty)$ satisfying
\begin{equation*}
0 < p - 1 \leq \frac{t{\phi } ^{\prime\prime}(t)}{{\phi } ^{\prime}(t)} \leq q - 1 < \infty \qquad  \forall t > 0,
\end{equation*}
As a direct consequence of $\phi \in \mathbf{\Phi}_{p,q}$, there exist the hidden constants depending only on $p $ and $q $ such that
\begin{equation}\label{equiv}
t^2{\phi } ^{\prime\prime}(t)
\approx	t{\phi } ^{\prime}(t)
\approx	{\phi } (t) \qquad  \forall t > 0.
\end{equation}
Moreover, for any $N$-function $\phi \in \mathbf{\Phi}_{p,q}$ it follows from \cite[Lemma 2.1]{BaB25} that for a constant $\nu \in (0, 1)$ there holds
\begin{equation}\label{est_<1}
{\phi } \,(\nu t) \leq \nu^{\,p } {\phi } (t) \quad \text{and} \quad {\phi } ^{-1}(\nu t) \leq \nu^{\,\frac{1}{q}} {\phi } ^{-1}(t)\quad  \forall t > 0;
\end{equation}
similarly, for the constant $L \in [1, \infty)$ one has
\begin{equation}\label{est_>1}
{\phi } \,(Lt) \leq  L^{q }{\phi } (t) \quad \text{and} \quad {\phi } ^{-1}(Lt) \leq L^\frac{1}{p } {\phi } ^{-1}(t)\quad  \forall t > 0.
\end{equation}

We are next to construct a concave function which is equivalent to  $\big({\phi } (t)\big)^{\delta}$ as the $\delta$-power of $N$-function ${\phi } \in \mathbf{\Phi}_{p,q}$ for $\delta \in (0,1/q )$. To be precise, we get the following lemma.
\begin{lemma}\label{lem_concave}
    Let ${\phi } \in \mathbf{\Phi}_{p,q}$ and $\delta \in (0, \frac{1}{q})$. Then there exists a continuous, concave function $\psi : [0, \infty) \to [0, \infty)$ such that $\psi(t) \thickapprox ({\phi } (t))^{\delta}$ for any $t \geq 0$.
\end{lemma}

\begin{proof}
Let us construct $\psi$ as follows:
\begin{equation*}
\psi(t) := \int_0^t \frac{\big({\phi } (s)\big)^{\delta}}{s} \,ds.
\end{equation*}
We notice that $\delta q<1$ implies that $\psi(\cdot)$ is a continuous function.
It follows from  \eqref{est_>1} that for ${\phi } \in \mathbf{\Phi}_{p,q}$ and $0 < s < t$ we get the following facts:
\begin{equation*}
\frac{\big({\phi } (t)\big)^{\delta}}{t}
= \frac{\big({\phi} (\frac{t}{s} s)\big)^{\delta}}{\frac{t}{s} s}
\leq \bigg(\frac{t}{s}\bigg)^{q\delta -1}\frac{\big({\phi } (s)\big)^{\delta}}{s}
\leq \frac{\big({\phi } (s)\big)^{\delta}}{s}.
\end{equation*}
This indicates that $\psi'(t)=\frac{({\phi } (t))^{\delta}}{t}$ is non-increasing, which implies that $\psi(\cdot)$ is concave function.

On the other hand, by the non-increasing property of $t \mapsto \frac{(\phi(t))^{\delta}}{t}$, then for any $0<s<t$ one has
\begin{equation*}
\big({\phi } (t)\big)^{\delta} = \int_0^t \frac{\big({\phi } (t)\big)^{\delta}}{t} \,ds \leq \psi(t).
\end{equation*}
Moreover, by the growth condition \eqref{est_<1} for ${\phi } \in \mathbf{\Phi}_{p,q}$ it implies the facts that
\begin{equation*}
\psi(t) = \int_0^t \frac{\big({\phi } (\frac{s}{t} t)\big)^{\delta}}{\frac{s}{t} t} \,ds
\leq \int_0^t \bigg(\frac{s}{t}\bigg)^{p\delta -1} \frac{\big({\phi } (t)\big)^{\delta}}{t} \,ds
= \frac{\big({\phi } (t)\big)^{\delta}}{p\delta},
\end{equation*}
which completes the equivalence of $\psi(t)$ and $({\phi } (t))^{\delta}$.
\end{proof}

For any $N$-function ${\phi } \in \mathbf{\Phi}_{p,q}$, let us introduce an auxiliary vector field $V_{\phi}:\mathbb{R}^{m}\rightarrow\mathbb{R}^{m}$ that will be useful to our main proof, which is  written by
\begin{equation*}
V_{\phi}(z)
:=	\bigg(\frac{{\phi }^{\prime}(|z|)}{|z|}
\bigg)^{\frac{1}{2}} z.
\end{equation*}
We immediately know that $|V_{\phi}(z)|^2 \thickapprox  {\phi }(|z|)$  for $z \in \mathbb{R}^{m}\setminus \{\mathbf{0}\}$, and for any $z_1, z_2 \in \mathbb{R}^{m}\setminus \{\mathbf{0}\}$ there holds
\begin{equation}\label{est_V}
\big|V_{\phi}(z_1) - V_{\phi}(z_2)\big|^2 \thickapprox \frac{{\phi }'(|z_1|+ |z_2|)}{|z_1| + |z_2|} |z_1 - z_2|^2,
\end{equation}
where the implied constant depends only on $p$ and $q$.

In the sequel, let us introduce the \emph{Musielak-Orlicz spaces}.

\begin{definition}
For open subset  $\mathcal{D} \subset \mathbb{R}^n$, the function $\varPsi : \mathcal{D} \times [0, \infty) \to [0, \infty)$ is called a Musielak-Orlicz function if
$\varPsi(x,\cdot)$ is an $N$-function in $t\ge 0$ for every $x \in \mathcal{D}$ and $\varPsi(\cdot, t)$ is measurable in $x\in \mathcal{D}$ for any $t \geq 0$. We say that a function $f : \mathcal{D} \to \mathbb{R}^m$ belongs to the Musielak-Orlicz class {$K^\varPsi(\mathcal{D}, \mathbb{R}^m)$} if
\begin{equation*}
\int_{\mathcal{D}} \varPsi\big(x, |f(x)|\big)\, dx < \infty.
\end{equation*}
The Musielak-Orlicz space is all vector functions generated by $f(x)\in K^\varPsi(\mathcal{D}, \mathbb{R}^m)$ with $\varPsi(x,\cdot) \in \mathbf{\Phi}_{p,q}$ for any $x\in \mathcal{D}$, denoted by $L^{\varPsi}(\mathcal{D},\mathbb{R}^m)$, equipped with the  Luxemburg  norm
\begin{equation*}
\|f\|_{L^{\varPsi}(\mathcal{D},\mathbb{R}^m)} = \inf\left\{ \lambda > 0 : \int_{\mathcal{D}} \varPsi\,\bigg(x, \frac{|f(x)|}{\lambda}\bigg)\, dx \leq 1 \right\}.
\end{equation*}
\end{definition}
A function $f$ belongs to \emph{the Musielak-Orlicz-Sobolev space} $W^{1,{\varPsi}}(\mathcal{D}, \mathbb{R}^m)$, if itself and its distributional gradient $Df$ belong to the Musielak-Orlicz space equipped with the norm
\begin{equation*}
\|f\|_{W^{1,{\varPsi}}(\mathcal{D}, \mathbb{R}^m)} = \|f\|_{L^{\varPsi}(\mathcal{D}, \mathbb{R}^m)} + \|Df\|_{L^{\varPsi}(\mathcal{D}, \mathbb{R}^{m \times n})}.
\end{equation*}
With $ \varPsi\big(x, t\big) = {\varphi_1}\big(t\big) + a(x)\,{\varphi_2} \big(t\big)$, we directly get the \emph{Sobolev-Poincar{\'e} type inequality} in the sense of Musielak-Orlicz-Sobolev space $W^{1,{\varPsi}}(B_r, \mathbb{R}^{m})$, which is as a special double-phase setting with $b(x)=0$ concerning the multi-phase function shown in  \cite[Theorem 4.1]{BaB25}. To be complete, we state it as follows.

\begin{lemma}\label{lem_s_p}
Let ${\varphi_1}, {\varphi_2} \in \mathbf{\Phi}_{p,q}$ with ${\varphi_1} \prec  {\varphi_2}$, and $ \varPsi\big(x, t\big) = {\varphi_1}\big(t\big) + a(x)\,{\varphi_2} \big(t\big)$ with $a(x)$ being continuous satisfying the gap condition \eqref{gap_2}. Then, for any  $f \in W^{1,{\varPsi}}(B_r(y), \mathbb{R}^m)$ and $d_2 \in [1, \frac{n^2}{n^2-1})$, there exists a constant $d_1 = d_1(n, m, p, q, d_2) \in (0,1)$ such that
\begin{equation*}
\Bigg(\fint_{B_r(y)} \bigg(\varPsi\Big(x, \frac{|f - \langle f \rangle_{B_r(y)}|}{r}\Big)\bigg)^{\,d_2} \, dx\Bigg)^{\frac{1}{d_2}} \leq c\, \Bigg(\fint_{B_r(y)} \Big(\varPsi\big(x, |Df|\big)\Big)^{\,d_1} \, dx\Bigg)^{\frac{1}{d_1}},
\end{equation*}
where $c = c(n, m, p, q, d_2, {L}, \|\varPsi(x,|Df|)\|_{L^1(B_r)}) > 0$.
\end{lemma}

Our main proof is proved by using the Morrey and Campanato estimates of gradients to the minimizers of various localized functionals. To this end, let us now recall the precise definition of two spaces.
\begin{definition}
For open subset $\mathcal{D}\subset \mathbb{R}^n$, $p \in[1,+\infty)$ and $\lambda \in (0,\infty)$, we have the following:
\begin{enumerate}
\item [(i)]  The Morrey space $L^{p, \lambda}(\mathcal{D}, \mathbb{R}^m)$ is defined by the set of all vector-valued functions with
\begin{equation*}
L^{p, \lambda}\left(\mathcal{D}, \mathbb{R}^m\right) :=\Bigg\{f \in L^p\left(\mathcal{D} ; \mathbb{R}^m\right): \sup_{\substack{y \in \mathcal{D} \\ r>0}} \frac{1}{r^\lambda} \int_{\mathcal{D} \cap B_r(y)}|f|^p \,dx <+\infty \Bigg\}
\end{equation*}
equipped with the norm
$$
\|f\|^p_{L^{p, \lambda}(\mathcal{D}, \mathbb{R}^m)}:= \sup_{\substack{y \in \mathcal{D}, r>0}} \frac{1}{r^\lambda} \int_{\mathcal{D} \cap B_r(y)}|f|^p \,dx.
$$
\item [(ii)]  The Campanato space $\mathcal{L}^{p,\lambda}(\mathcal{D}, \mathbb{R}^m)$ is defined as the set of all vector-valued functions with
\begin{equation*}
\mathcal{L}^{p,\lambda}(\mathcal{D}, \mathbb{R}^m) := \Bigg\{f \in L^p\left(\mathcal{D} ; \mathbb{R}^m\right): \sup_{\substack{y \in \mathcal{D}, r>0}} \frac{1}{r^\lambda} \int_{\mathcal{D} \cap B_r(y)}|f - \langle f \rangle_{B_r(y)}|^p \,dx <+\infty \Bigg\}
\end{equation*}
endowed with the norm
\begin{equation*}
\|f\|^p_{\mathcal{L}^{p,\lambda}(\mathcal{D}, \mathbb{R}^m)} := \|f\|^p_{L^{p}(\mathcal{D}, \mathbb{R}^m)} + \sup_{\substack{y \in \mathcal{D}, r>0}} \frac{1}{r^\lambda} \int_{\mathcal{D} \cap B_r(y)}|f - \langle f \rangle_{B_r(y)}|^p \,dx.
\end{equation*}
\end{enumerate}
\end{definition}

To show the equivalent relationships among the Morrey space, the Campanato space and the H\"{o}lder space, let us recall the so-called \emph{$A$-type domain} or the \emph{Ahlfors regular domain}.   $\mathcal{D}$ is $A$-type domain, if there exists a constant $A > 0$ such that
$$
\left|B_r\left(y\right) \cap \mathcal{D}\right| \geq A r^{\,n}\qquad \text{for\ any}\ y \in \mathcal{D}\  \text{and}\ r \in (0, \text{diam}(\mathcal{D})).
$$
For any fixed bounded $A$-type domain $\mathcal{D} \subset \mathbb{R}^n$, the following lemma describes the isomorphism among the Morrey space, the Campanato space, and the H\"{o}lder space, which can be found in \cite[Proposition 1.2 and Theorem 1.2]{Gi83}.

\begin{lemma}\label{lem_isomorphism}
Let $\mathcal{D} \subset \mathbb{R}^n$ be a bounded $A$-type domain. For $p \in[1,+\infty)$ and $\lambda \in (0,\infty)$ we conclude that
\begin{enumerate}

\item[(i)] if $0 \leq \lambda < n$, then $L^{p,\lambda}(\mathcal{D}, \mathbb{R}^m)$ is isomorphic to $\mathcal{L}^{p,\lambda}(\mathcal{D}, \mathbb{R}^m)$;
\item[(ii)] if $n < \lambda \leq n + p$, then $\mathcal{L}^{p,\lambda}(\mathcal{D}, \mathbb{R}^m)$ is isomorphic to $C^{0,\frac{\lambda - n}{p}}(\mathcal{D}, \mathbb{R}^m)$;
\item[(iii)] if $n - p < \lambda < n$ and $f \in W^{1,p}(\mathcal{D}, \mathbb{R}^m)$ with $Df \in L^{p,\lambda}(\mathcal{D}, \mathbb{R}^{m \times n})$, then $f \in C^{0,1-\frac{n-\lambda}{p}}(\mathcal{D}, \mathbb{R}^m)$.
\end{enumerate}
\end{lemma}

In the sequel, we are to introduce the well-known iterating lemma, see \cite[Chapter V, Lemma 3.1]{Gi83}.

\begin{lemma}\label{lem_iteration}
Let $f(\cdot)$ be a non-negative, bounded function on the interval $[r,R] \subset (0,\infty)$. If there exist constants $\theta \in (0,1)$ and $A,B,\alpha > 0$ such that for any $r \leq \varrho< \sigma \leq R$ it holds
\begin{equation*}
f(\varrho) \leq \theta f(\sigma) + \Big(B\,(\sigma-\varrho)^{-\alpha} + A\Big)\,,
\end{equation*}
then we have
\begin{equation*}
f(r) \leq c\,\Big(B\,(R-r)^{-\alpha} + A\Big)
\end{equation*}
with $c = c(\alpha, \theta) > 0$.
\end{lemma}

We are in a position to recall the following two local higher integrability results based on the reverse H\"older inequality. Let us first introduce the following \emph{Gehring-type lemma}, which can be referred to \cite[Chapter V, Proposition 1.1]{Gi83}.

\begin{lemma}\label{lem_gehring}
Let $B_r \Subset \mathcal{D}\subset \mathbb{R}^n$. If for non-negative functions $g \in L_\mathrm{loc}^1(\mathcal{D})$ and $f\in L_\mathrm{loc}^q (\mathcal{D})$ with $q>1$, there exist two constants $b > 0$ and $\delta \in (0,1)$ such that
\begin{equation*}
\fint_{B_{r/2}} g \, dx \leq  b\, \Bigg( \bigg(\fint_{B_{r}} g^{\,\delta} \, dx\bigg)^\frac{1}{\delta} + \fint_{B_{r}} f \, dx \Bigg)\,,
\end{equation*}
then there exists an $\varepsilon_0 = \varepsilon_0\big(n, q, b, \delta\big) > 0$ such that for any $\varepsilon \in (0, \varepsilon_0]$ it holds
\begin{equation*}
\bigg(\fint_{B_{r/2}} g^{1+\varepsilon} \, dx\bigg)^\frac{1}{1+\varepsilon} \leq B\,\Bigg(\fint_{B_{r}} g \, dx +  \bigg(\fint_{B_{r}} f^{1+\varepsilon} \, dx\bigg)^\frac{1}{1+\varepsilon}\Bigg)
\end{equation*}
with $B = B\big(n, q, b, \delta\big) > 0$.
\end{lemma}

Moreover, the next lemma describes a  \emph{self-improving property} based on the backward iteration of the reverse H{\"o}lder inequality, see \cite[Remark 6.12]{Gi03} or \cite[Lemma 3.1]{KuM18}.

\begin{lemma}\label{lem_self_improve}
Let $B_r \Subset \mathcal{D}\subset \mathbb{R}^n$. If for a  non-negative function $g \in L_\mathrm{loc}^{1+\varepsilon}(\mathcal{D})$, there exist two constants $A, b > 0$ such that
\begin{equation*}
\bigg(\fint_{B_{r/2}} g^{1+\varepsilon} \, dx\bigg)^\frac{1}{1+\varepsilon} \leq b\,\fint_{B_{r}} g \, dx +  A,
\end{equation*}
then, for any \(\tau > 0\) we have
\begin{equation*}
\bigg(\fint_{B_{r/2}} g^{1+\varepsilon} \, dx\bigg)^\frac{1}{1+\varepsilon} \leq B\,\Bigg(\bigg(\fint_{B_{r}} g^{\tau} \, dx\bigg)^\frac{1}{\tau} +  A \Bigg)
\end{equation*}
with $B = B\big(n, b, \tau\big) > 0$.
\end{lemma}

Finally, it is useful for the \emph{Hausdorff dimensional estimate} of singular set that we introduce the following Hausdorff measure estimate concerning an exceptional set, see \cite[Chapter IV, Theorem 2.2]{Gi83}.

\begin{lemma}\label{lem_Hausdorff}
Let $f\in L^1_\mathrm{loc}(\mathcal{D})$ with $\mathcal{D}$ as an open subset of $\mathbb{R}^n$ and $0\leq s<n$. Then for the set
\begin{equation*}
\mathcal{Q}_s:=
\left\{
y \in \mathcal{D}:
\limsup_{r \rightarrow 0^{+}} r^{-s} \int_{B_r(y)} |f| \,dx >0 \quad \text {for\ any} \ \ B_r(y)\Subset \mathcal{D}
\right\},
\end{equation*}
we have $\mathcal{H}^s(\mathcal{Q}_s)=0$, where $\mathcal{H}^s(\cdot)$ denotes  the $s$-dimensional Hausdorff measure.
\end{lemma}

\section{Higher integrability and harmonic approximation}
\setcounter{equation}{0}
\setcounter{theorem}{0}

In this section, we are to show a higher integrability of the gradients to the local minimizers of  functional $\mathcal{F}_{\mathbb{A},{\varphi_1},{\varphi_2}}(u ; \Omega)$, and recall  useful \emph{harmonic approximation} in the form of functional. First of all, let us begin with introducing some notations that will be useful throughout this section. Denote
\begin{equation}\label{Gr}
\mathcal{G}(r) := \omega_a(r)\, \frac{\big({\varphi_2} \circ \varphi_1^{-1}\big)\left(r^{-n}\right)}{r^{-n}}.
\end{equation}
By the gap condition \eqref{gap_2}, then there exist two constants $R_0 \in (0,1)$ and $K \geq 1$ such that
\begin{equation}\label{R1}
\mathcal{G}(r)\le  K
\quad \text{for\ any}\ r \in (0, R_0).
\end{equation}
Let $\Omega' \Subset \Omega$ be a bounded open subset.
In what follows, we write
\begin{equation*}
\varPsi(x,t) := {\varphi_1}(t) + a(x)\,{\varphi_2}(t),
\end{equation*}
and for any given ball $B_r = B_r(x_0) \subset \Omega' \Subset \Omega$, we write
\begin{equation}\label{bar-Psi}
\bar{a}_r := \inf_{x \in B_r} a(x)
\quad \text{and} \quad
\bar{\varPsi}_r(t) := {\varphi_1}(t) + \bar{a}_r\, {\varphi_2}(t).
\end{equation}
In order to simplify the dependence on parameters, we collect the parameters appearing in the context and denote them by
\begin{eqnarray*}
\texttt{data} &:=& \Big(n, N, p, q, \Lambda, L, K, \|a\|_{L^\infty(\Omega)}, \|\varPsi(x, |Du|)\|_{L^1(\Omega)}\Big)\,,\notag\\
\texttt{data}_0 &:=& \Big(\texttt{data}, \mathrm{dist}(\Omega', \Omega), \mathrm{diam}(\Omega')\Big)\,.
\end{eqnarray*}

Let us  now give a higher integrability of gradients to the local minimizer of functional $\mathcal{F}_{\mathbb{A},{\varphi_1},{\varphi_2}}(u ; \Omega)$.

\begin{lemma}\label{lem_high_int}
Let ${\varphi_1}, {\varphi_2} \in \mathbf{\Phi}_{p,q}$ with ${\varphi_1} \prec  {\varphi_2}$, $\mathbb{A}(\cdot,\cdot)$ satisfy \eqref{cond_A}, and  $u \in W^{1,1}(\Omega,\mathbb{R}^N)$ be a local minimizer of the functional \eqref{main_funl}. If $a(\cdot)$ is continuous satisfying the gap condition \eqref{gap_2}, then there exists an $\varepsilon_0 > 0$ such that for any $B_r \Subset \Omega$  and $\varepsilon \in (0, \varepsilon_0]$ one has
\begin{equation}\label{ineq_rev}
\bigg(\fint_{B_{r/2}} \varPsi^{1+\varepsilon}\big(x,|Du|\big) \, dx\bigg)^{\frac{1}{1+\varepsilon}} \leq c\, \bar{\varPsi}_r\,\bigg(\fint_{B_{r}} |Du| \, dx\bigg)\,
\end{equation}
where $c = c\,\big(\texttt{data}\big) \geq 1$.
Moreover, for any open subset $\Omega^\prime \Subset \Omega$, there is an $M = M\big(\texttt{data}_0\big) \geq 1$ such that
\begin{equation}\label{high_norm}
\|\varPsi(x,|Du|)\|_{L^{1+\varepsilon}(\Omega^\prime)}  \leq  {M}.
\end{equation}
\end{lemma}

\begin{proof}
For $1/2 \leq \rho < \sigma \leq 1$, let $\eta \in C^\infty_0(B_r)$ be a cut-off function with $\eta \equiv 1$ on $B_{\rho r}$, $\eta \equiv 0$ outside $B_{\sigma r}$, and $|D \eta| \leq \frac{4}{(\sigma - \rho)r}$. We now take $\phi = \eta\,(u - Q) \in W^{1,1}_0(B_{\sigma r}, \mathbb{R}^N)$ with $Q \in \mathbb{R}^N$ as  a test function of \eqref{main_funl}. By  the  ellipticity \eqref{cond_A}, the local minimality of $u$, and the properties \eqref{est_>1} for  ${\varphi_1} , {\varphi_2} \in \mathbf{\Phi}_{p,q}$, we conclude the following facts:
\begin{eqnarray*}
&&\Lambda^{-\frac{q}{2}}\fint_{B_{\sigma r}} \varPsi\big(x, |Du|\big) \, dx\\
&\leq & \fint_{B_{\sigma r}} \varPsi\big(x, \Lambda^{-\frac{1}{2}}|Du|\big) \, dx\\
&\leq & \fint_{B_{\sigma r}} \varPsi\big(x, |Du|_{\mathbb{A}^u}\big) \, dx \\
&\leq& \fint_{B_{\sigma r}} \varPsi\big(x, |Du - D\phi|_{\mathbb{A}^{u-\phi}}\big) \, dx
\\
&\le &  \Lambda^\frac{q}{2}\fint_{B_{\sigma r}} \varPsi\Big(x, \big|(1-\eta)Du - D\eta (u - Q)\big|\Big) \, dx\notag\\
&\leq& 2^{q-1}\Lambda^\frac{q}{2}\bigg(\fint_{B_{\sigma r}} \varPsi\Big(x, \big|(1-\eta)Du\big|\Big) \, dx  +  \fint_{B_{\sigma r}} \varPsi\Big(x, \big| (u - Q)D\eta\big|\Big) \, dx\bigg)\notag\\
&\leq& 2^{q-1}\Lambda^\frac{q}{2}\Bigg(\fint_{B_{\sigma r}} \varPsi\Big(x, \big|(1-\eta)Du\big|\Big) \, dx  +  \fint_{B_{\sigma r}} \varPsi\bigg(x, \frac{4|u - Q|}{(\sigma - \rho)r}\bigg) \, dx\Bigg)\notag\\
&\leq& (2\Lambda)^{q}\Bigg(\fint_{B_{\sigma r} \setminus B_{\rho r}} \varPsi\big(x, |Du|\big) \, dx  +  \bigg(\frac{4}{\sigma - \rho}\bigg)^{q}\fint_{B_{\sigma r}} \varPsi\,\bigg(x, \frac{|u - Q|}{r}\bigg) \, dx\Bigg)\,,
\end{eqnarray*}
where we used  $\frac{4}{\sigma - \rho} > 1$ and \eqref{est_>1} in  the last inequality. Multiplying both sides of the above inequality by $\Lambda^{\frac{q}{2}}$, we attain
\begin{equation*}
\fint_{B_{\rho r}} \varPsi\big(x, |Du|\big) \, dx
\leq \big(2\Lambda\big)^{2q}\fint_{B_{\sigma r} \setminus B_{\rho r}} \varPsi\big(x, |Du|\big) \, dx  +  \frac{(8\Lambda)^{2q}}{(\sigma - \rho)^{q}}\fint_{B_{\sigma r}} \varPsi\,\Big(x, \frac{|u - Q|}{r}\Big) \, dx\,.
\end{equation*}

Let us now employ the ``hole-filling"  method by adding $(2\Lambda)^{2q}\fint_{B_{\rho r}} \varPsi\big(x, |Du|\big) \, dx$, to deduce
\begin{equation*}
\fint_{B_{\rho r}} \varPsi\big(x, |Du|\big) \, dx
\leq \frac{\big(2\Lambda\big)^{2q}}{\big(2\Lambda\big)^{2q} + 1}\fint_{B_{\sigma r}} \varPsi\big(x, |Du|\big) \, dx  +  \frac{c}{(\sigma - \rho)^{q}}\fint_{B_{\sigma r}} \varPsi\,\Big(x, \frac{|u - Q|}{r}\Big) \, dx.
\end{equation*}
Using the iterating Lemma \ref{lem_iteration} on the interval $[r/2, r]$, we obtain the Caccioppoli inequality
\begin{equation}\label{cacc}
\fint_{B_{r/2}} \varPsi\big(x, |Du|\big) \, dx
\leq c\fint_{B_{r}} \varPsi\,\Big(x, \frac{|u - Q|}{r}\Big) \, dx.
\end{equation}
We now pick up $Q = \langle u \rangle_{r}$, and make use of the Sobolev-Poincar\'e inequality shown as Lemma \ref{lem_s_p} with $d_2=1$, to conclude the reverse H\"older inequality that for some $\theta \in (0,1)$,
\begin{equation*}
\fint_{B_{r/2}} \varPsi\big(x, |Du|\big) \, dx  \leq c\, \bigg(\fint_{B_{r}} \Big(\varPsi\big(x, |Du|\big)\Big)^{\,\theta} \, dx\bigg)^\frac{1}{\theta}.
\end{equation*}
By the Gehring-type Lemma \ref{lem_gehring}, then there exists an $\varepsilon_0 > 0$ such that for any $\varepsilon \in (0, \varepsilon_0]$,
\begin{equation*}
\bigg(\fint_{B_{r/2}} \Big(\varPsi\big(x, |Du|\big)\Big)^{1+\varepsilon} \, dx\bigg)^\frac{1}{1+\varepsilon}  \leq c \fint_{B_{r}} \varPsi\big(x, |Du|\big) \, dx.
\end{equation*}
Further, we employ the self-improving property of  Lemma \ref{lem_self_improve}, we  pick up an index $0<\delta < 1/q$ {to get} the following estimate
\begin{equation}\label{ineq_rev_s_i}
\bigg(\fint_{B_{r/2}} \Big(\varPsi\big(x, |Du|\big)\Big)^{1+\varepsilon} \, dx\bigg)^\frac{1}{1+\varepsilon}
\leq c\, \bigg(\fint_{B_{r}} \Big(\varPsi\big(x, |Du|\big)\Big)^{\,\delta } \, dx\bigg)^\frac{1}{\delta}.
\end{equation}
The elementary calculations yield that
\begin{equation*}
\Big(\varPsi(x, |Du|)\Big)^{\delta }\le \Big(\bar{\varPsi}_r(|Du|)\Big)^{\delta }+\Big(L\omega_a(r)\,{\varphi_2}\big(|Du|\big)\Big)^\delta
\end{equation*}
From $\bar{\varPsi}_r,\, \varphi_2 \in \mathbf{\Phi}_{p,q}$ and Lemma \ref{lem_concave}, we see that $\big(\bar{\varPsi}_r(\cdot)\big)^{\delta}$ and $\big({\varphi_2(\cdot)}\big)^{\delta}$ are equivalent to two concave functions respectively.
Then by the Jensen inequality, \eqref{ineq_rev_s_i} is estimated as
\begin{eqnarray}\label{right}
\bigg(\fint_{B_{r/2}} \Big(\varPsi(x, |Du|)\Big)^{1+\varepsilon} \, dx\bigg)^\frac{1}{1+\varepsilon}
&\leq& c\,\bar{\varPsi}_r\,\bigg(\fint_{B_{r}} |Du| \, dx\bigg)  +  c\,\omega_a(r)\,{\varphi_2}\,\bigg(\fint_{B_{r}} |Du| \, dx\bigg)\,.
\end{eqnarray}

In what follows, we are to estimate the second term of the right-hand side in \eqref{right}. Since $|Du|= \varphi_1^{-1}\Big({\varphi_1}(|Du|)\Big)$,  we  apply  the Jensen inequality to deduce that
\begin{equation}\label{est_l_1}
\fint_{B_{r}} |Du|\,dx \leq {\varphi_1^{-1}}\,\bigg(\fint_{B_{r}} {\varphi_1}(|Du|)\,dx\bigg)
\leq \varphi_1^{-1}\Big({\|\varPsi(x, |Du|)\|_{L^1(B_r)}}\big(\omega_n r^n\big)^{-1}\Big)\,
\end{equation}
where $\omega_n$ is the measure of the $n$-dimensional unit ball. Byt $\varphi_1 \prec \varphi_2$ it leads to that both $\varphi_2 \circ \varphi^{-1}_1(\cdot)$ and $\frac{{\varphi_2}(\cdot)}{{\varphi_1}(\cdot)}$ are non-decreasing, see \cite[Lemma 2.7]{ByO20}. Let $T = \|\varPsi(x, |Du|)\|_{L^1(B_r)}\omega^{-1}_n + 1$, it yields $\fint_{B_{r}} |Du|\,dx \le \varphi_1^{-1}\big(T r^{-n}\big)$ in accordance with the non-decreasing property of $\varphi_1^{-1}(\cdot)$ and \eqref{est_l_1}, which implies that
\begin{equation}\label{omega}
\frac{{\varphi_2}\,\Big(\fint_{B_{r}} |Du| \, dx\Big)}{{\varphi_1}\,\Big(\fint_{B_{r}} |Du| \, dx\Big)}
\le		\frac{{\varphi_2} \circ \varphi_1^{-1}\big(T r^{-n}\big)}{{\varphi_1} \circ \varphi_1^{-1}\big(T r^{-n}\big)}
\leq	(T)^{\frac{q}{p}-1}\,\frac{{\varphi_2}\circ \varphi_1^{-1}\big(r^{-n}\big)}{r^{-n}},
\end{equation}
where we used \eqref{est_>1} in the last inequality.
Combining \eqref{omega} with ${\varphi_1}(\cdot)\le {\varphi_1}(\cdot)+\bar{a}_r\, {\varphi_2}(\cdot)=\bar{\varPsi}_r\,(\cdot)$, the second term on the right-hand side in \eqref{right} is estimated as follows:
\begin{eqnarray}\label{right-2}
\omega_a(r)\,{\varphi_2}\,\bigg(\fint_{B_{r}} |Du| \, dx\bigg)
&\leq&  \omega_a(r)\,\frac{{\varphi_2}\,\Big(\fint_{B_{r}} |Du| \, dx\Big)}{{\varphi_1}\,\Big(\fint_{B_{r}} |Du| \, dx\Big)}\,{\varphi_1}\,\bigg(\fint_{B_{r}} |Du| \, dx\bigg)\notag\\
&\leq& {c\, \omega_a(r)\,\frac{{\varphi_2}\circ \varphi_1^{-1}\big(r^{-n}\big)}{ r^{-n}}\,\bar{\varPsi}_r\,\bigg(\fint_{B_{r}} |Du| \, dx\bigg)}   \notag\\
&=& c \, \mathcal{G}(r)\,\bar{\varPsi}_r\,\bigg(\fint_{B_{r}} |Du| \, dx\bigg)\,,
\end{eqnarray}
Let us put \eqref{ineq_rev_s_i} \eqref{right} together with \eqref{right-2} to obtain
\begin{eqnarray*}
\bigg(\fint_{B_{r/2}} \Big(\varPsi\big(x, |Du|\big)\Big)^{1+\varepsilon} dx\bigg)^\frac{1}{1+\varepsilon}
&\leq & c\,\Big(1+  \mathcal{G}(r)\Big)\,\bar{\varPsi}_r\,\bigg(\fint_{B_{r}} |Du| \, dx\bigg) \notag\\
&\leq & c\,\Big(1+ K\Big)\,\bar{\varPsi}_r\,\bigg(\fint_{B_{r}} |Du| \, dx\bigg)\,,
\end{eqnarray*}
where we used \eqref{R1} in the last inequality.
The boundedness of  \eqref{high_norm}  readily follows from a standard finite covering argument.
\end{proof}

\begin{remark}
By Definition of $\bar{\varPsi}_r(\cdot)$ as in \eqref{bar-Psi} we see that $\bar{\varPsi}_r(t) \leq \varPsi(x, t)$ for  $x \in B_r$. Then, by the H\"older inequality and reverse H\"older inequality  \eqref{ineq_rev} we have
\begin{equation}\label{ineq_rev_jense}
\fint_{B_{r/2}} \bar{\varPsi}_r\big(|Du|\big)\, dx
\leq
\bigg(\fint_{B_{r/2}} \Big(\varPsi\big(x, |Du|\big)\Big)^{1+\varepsilon} dx\bigg)^\frac{1}{1+\varepsilon}
\leq c\,\bar{\varPsi}_r\,\bigg(\fint_{B_{r}} |Du| \, dx\bigg)\,.
\end{equation}
\end{remark}

In what follows, we give the Caccioppoli-type inequality to the local minimizer $u$ of \eqref{main_funl} in $L^1$-sense.

\begin{lemma}\label{lem_cacc}
For ${\varphi_1}, {\varphi_2} \in \mathbf{\Phi}_{p,q}$ with ${\varphi_1} \prec  {\varphi_2}$ and $\mathbb{A}(\cdot,\cdot)$ satisfying \eqref{cond_A}, let $u \in W^{1,1}(\Omega,\mathbb{R}^N)$ be a local minimizer of the functional \eqref{main_funl}. If $a(\cdot)$ is continuous with the modulus of continuity  \eqref{cond_continuous} satisfying the gap condition \eqref{gap_2}, then for any $B_r \Subset \Omega$ we have
\begin{equation}\label{cacc_L_1}
\fint_{B_{r/4}} |Du| \, dx \leq c_5\, \fint_{B_r} \frac{|u - \langle u \rangle_{r}|}{r} \, dx\,,
\end{equation}
where $c_5 = c_5\,\big(\texttt{data}\big) \geq 1$.
\end{lemma}

\begin{proof}
According to the Sobolev-Poincar\'e inequality as in Lemma \ref{lem_s_p} and the Caccioppoli inequality \eqref{cacc}, then for $d \in \Big(1, \frac{n^2}{n^2 - 1}\Big)$ we get
\begin{eqnarray*}
\Bigg(\fint_{B_{r/2}} \bigg(\varPsi\,\Big(x, \frac{|u - \langle u \rangle_{r/2}|}{r}\Big)\,\bigg)^{d} \, dx\Bigg)^\frac{1}{d}
&\leq &  c \fint_{B_{r/2}} \varPsi\big(x, |Du|\big) \, dx  \\
&\leq &  c \fint_{B_{r}} \varPsi\,\bigg(x, \frac{|u - \langle u \rangle_{r/2}|}{r}\bigg) \, dx.
\end{eqnarray*}
By Lemma \ref{lem_self_improve} it  yields that for any $0<\delta<1/q$ it holds
\begin{eqnarray*}
\Bigg(\fint_{B_{r/2}} \bigg(\varPsi\,\Big(x, \frac{|u - \langle u \rangle_{r/2}|}{r}\Big) \,\bigg)^d dx\Bigg)^\frac{1}{d}
&\leq&  c\,\Bigg(\fint_{B_{r}} \bigg(\varPsi\,\Big(x, \frac{|u - \langle u \rangle_{r/2}|}{r}\Big)\,\bigg)^{\delta} \, dx\Bigg)^\frac{1}{\delta}.
\end{eqnarray*}
It follows from Lemma \ref{lem_concave} that both $\big(\bar{\varPsi}_r(\cdot)\big)^\delta$ and $\big(\varphi_2(\cdot)\big)^\delta$ are two concave functions, respectively. Then we use the Jensen inequality to deduce that
\begin{eqnarray}\label{cacc_2}
&&\Bigg(\fint_{B_{r/2}} \bigg(\varPsi\,\Big(x, \frac{|u - \langle u \rangle_{r/2}|}{r}\Big) \,\bigg)^d dx\Bigg)^\frac{1}{d}\notag\\
&\leq& c\, \Bigg(\fint_{B_{r}} \bigg(\bar{\varPsi}_r\,\Big(\frac{|u - \langle u \rangle_{r/2}|}{r}\Big)\,\bigg)^{\delta } \, dx\Bigg)^\frac{1}{\delta}  + c\, \Bigg(\fint_{B_{r}} \big|a(x)-\bar{a}_r\big|^{\,\delta}\bigg({\varphi_2}\Big(\frac{|u - \langle u \rangle_{r/2}|}{r}\Big)\bigg)^{\delta } \, dx\Bigg)^\frac{1}{\delta}\notag\\
&\leq& c\, \Bigg(\fint_{B_{r}} \bigg(\bar{\varPsi}_r\,\Big(\frac{|u - \langle u \rangle_{r/2}|}{r}\Big)\,\bigg)^{\delta } \, dx\Bigg)^\frac{1}{\delta}  + c\,L\,\omega_a(r) \Bigg(\fint_{B_{r}} \bigg({\varphi_2}\Big(\frac{|u - \langle u \rangle_{r/2}|}{r}\Big)\bigg)^{\delta } \, dx\Bigg)^\frac{1}{\delta}\notag\\
&\leq & c\, \bar{\varPsi}_r\,\Bigg(\fint_{B_{r}} \frac{|u - \langle u \rangle_{r/2}|}{r} \, dx\Bigg)  +  c\,L\,\omega_a(r)\,{\varphi_2}\,\Bigg(\fint_{B_{r}} \frac{|u - \langle u \rangle_{r/2}|}{r} \, dx\Bigg)\,.
\end{eqnarray}
Using the Poincar\'e inequality and the Jensen inequality, together with \eqref{est_>1}, we deduce that
\begin{eqnarray}\label{poincare}
\fint_{B_{r}} \frac{|u - \langle u \rangle_{r/2}|}{r} \, dx
&\leq& c_6 \fint_{B_{r}} |Du|\,dx \notag\\
&\leq& \varphi_1^{-1}\,\Bigg(\fint_{B_{ r}} {\varphi_1}\,\Big(c_6|Du|\Big)\,dx\Bigg) \notag\\
&\leq&  \varphi_1^{-1}\,\Bigg( \Big(c_6 +1\Big)^{q}\fint_{B_{ r}} {\varphi_1}\,\big(|Du|\big)\,dx\Bigg) \notag\\
&\leq& \varphi_1^{-1}\bigg(\Big(c_6+1\Big)^{q}\,\omega^{-1}_n  r^{-n}{\|\varPsi(x, |Du|)\|_{L^1(B_r)}}\bigg) \notag\\[6pt]
&\leq& \varphi_1^{-1}\big(T\, r^{-n}\big)
\end{eqnarray}
with $T = \big(c_6+1\big)^{q}\omega^{-1}_n \,{\|\varPsi(x, |Du|)\|_{L^1(B_r)}} + 1$, which implies
$\frac{{\varphi_2}\,\big(\fint_{B_{r}} \frac{|u - \langle u \rangle_{r/2}|}{r} \, dx\big)}{{\varphi_1}\,\big(\fint_{B_{r}} \frac{|u - \langle u \rangle_{r/2}|}{r} \, dx\big)}  \leq T^{\frac{q}{p}-1} \frac{{\varphi_2}\circ \varphi_1^{-1}\big(r^{-n}\big)}{r^{-n}}$
by using the same way as in \eqref{omega}. Then it deduces that
\begin{eqnarray*}
\omega_a(r)\,{\varphi_2}\,\bigg(\fint_{B_{r}} \frac{|u - \langle u \rangle_{r/2}|}{r} \, dx\bigg)
&=&  \omega_a(r)\,\frac{{\varphi_2}\,\Big(\fint_{B_{r}} \frac{|u - \langle u \rangle_{r/2}|}{r} \, dx\Big)}{{\varphi_1}\,\Big(\fint_{B_{r}} \frac{|u - \langle u \rangle_{r/2}|}{r} \, dx\Big)}\,{\varphi_1}\,\bigg(\fint_{B_{r}} \frac{|u - \langle u \rangle_{r/2}|}{r} \, dx\bigg)\notag\\
&\leq& c \,\omega_a\big( r\big)\,\frac{{\varphi_2}\circ \varphi_1^{-1}\big(r^{-n}\big)}{r^{-n}}\,\bar{\varPsi}_r\,\bigg(\fint_{B_{r}} \frac{|u - \langle u \rangle_{r/2}|}{r} \, dx\bigg)\notag\\
&=&  c\,\mathcal{G}(r)\,\bar{\varPsi}_r\,\bigg(\fint_{B_{r}} \frac{|u - \langle u \rangle_{r/2}|}{r} \, dx\bigg)\,.
\end{eqnarray*}
This combines with \eqref{cacc_2} to get
\begin{eqnarray*}
\Bigg(\fint_{B_{r/2}} \bigg(\varPsi\,\Big(x, \frac{|u - \langle u \rangle_{r/2}|}{r}\Big) \,\bigg)^d dx\Bigg)^\frac{1}{d}
\leq  c \,\bar{\varPsi}_r\,\bigg(\fint_{B_{r}} \frac{|u - \langle u \rangle_{r/2}|}{r} \, dx\bigg)\,,
\end{eqnarray*}
where we used $\mathcal{G}(r) \leq K$ for any $r \in (0, R_0)$ as in \eqref{R1}.

Now we again make use of the Jensen inequality, the Caccioppoli inequality \eqref{cacc} and the H\"older inequality to obtain that
\begin{eqnarray*}
\bar{\varPsi}_r\,\bigg(\fint_{B_{r/4}} |Du| \, dx\bigg)
&\leq&  \fint_{B_{r/4}} \bar{\varPsi}_r\big(|Du|\big) \, dx  \notag\\
&\leq&  \fint_{B_{r/4}} \varPsi\big(x, |Du|\big) \, dx   \notag\\
&\leq&  c\fint_{B_{r/2}} \varPsi\bigg(x, \frac{|u - \langle u \rangle_{r/2}|}{r}\bigg) \, dx\notag\\
&\leq&  c \,\Bigg(\fint_{B_{r/2}} \Bigg(\varPsi\,\bigg(x, \frac{|u - \langle u \rangle_{r/2}|}{r}\bigg) \,\Bigg)^d dx\Bigg)^\frac{1}{d}\notag\\
&\leq& c\,\bar{\varPsi}_r\,\bigg(\fint_{B_{r}} \frac{|u - \langle u \rangle_{r/2}|}{r} \, dx\bigg)\,,
\end{eqnarray*}
which implies, by the non-decreasing property of $t \mapsto \bar{\varPsi}^{-1}_r(t)$, that
\begin{eqnarray*}
\fint_{B_{r/4}} |Du| \, dx
\leq  c\, \fint_{B_r} \frac{|u - \langle u \rangle_{r/2}|}{r} \, dx.
\end{eqnarray*}
Then, a simple calculation yields that
\begin{eqnarray*}
\fint_{B_{r/4}} |Du| \, dx
&\leq& c\, \bigg(\fint_{B_r} \frac{|u - \langle u \rangle_{r}|}{r} \, dx + \fint_{B_{r/2}} \frac{|u - \langle u \rangle_{r}|}{r} \, dx\bigg)\notag\\
&\leq& c\fint_{B_r} \frac{|u - \langle u \rangle_{r}|}{r} \, dx,
\end{eqnarray*}
which completes the proof.
\end{proof}

We are now to give the comparison estimates by using the locally frozen coefficient approach. To this end, for $\bar{\mathbb{A}}_r^u = \mathbb{A}(x_0, \langle u \rangle_{B_r})$ let us consider $v \in u + W_0^{1,1}\big(B_{r/8}, \mathbb{R}^N\big)$ as a minimizer of the following local frozen functional
\begin{equation}\label{fun_v}
v \mapsto \int_{B_{r/8}} \bar{\varPsi}_r\big(|Dv|_{\bar{\mathbb{A}}_r^u}\big) \, dx.
\end{equation}
It follows from ${\varphi_1}, {\varphi_2} \in \mathbf{\Phi}_{p,q}$ that $\bar{\varPsi}_r(\cdot) \in \mathbf{\Phi}_{p,q}$ and
\begin{equation}\label{cond_p_q_2}
0 < p - 1
\leq
\frac{t\bar{\varPsi}_r^{\prime\prime}(t)}{\bar{\varPsi}_r^{\prime}(t)}
\leq
q  - 1 < \infty.
\end{equation}
By the ellipticity and boundedness of $\mathbb{A}(\cdot, \cdot)$ in \eqref{cond_A} and the properties \eqref{est_<1} \eqref{est_>1} of $\varphi_1(t),\varphi_2(t)$ to $\bar{\varPsi}_r(\cdot)$, we deduce that
$\Lambda^{-\frac{q}{2}}\,\bar{\varPsi}_r\big(|\xi|\big) \leq \bar{\varPsi}_r\big(|\xi|_{\bar{\mathbb{A}}_r^u}\big) \leq \Lambda^{\frac{q}{2}}\,\bar{\varPsi}_r\big(|\xi|\big)
$ for any $\xi \in \mathbb{R}^{N \times n}$.
Then by the minimality of $v$, we have the following energy estimates:
\begin{eqnarray}\label{energy_v}
\fint_{B_{r/8}} \bar{\varPsi}_r\big(|Dv|\big) \, dx
&\leq& \Lambda^{\frac{q}{2}} \fint_{B_{r/8}} \bar{\varPsi}_r\big(|Dv|_{\bar{\mathbb{A}}_r^u}\big) \, dx \notag\\
&\leq& \Lambda^{\frac{q}{2}} \fint_{B_{r/8}} \bar{\varPsi}_r\big(|Du|_{\bar{\mathbb{A}}_r^u}\big) \, dx \notag\\
&\leq&  \Lambda^{q}\fint_{B_{r/8}} \bar{\varPsi}_r\big(|Du|\big) \, dx.
\end{eqnarray}
Thanks to the convexity and the non-decreasing property of $t \mapsto \bar{\varPsi}_r(t)$, together with \eqref{energy_v} and \eqref{ineq_rev_jense}, we get that
\begin{eqnarray}\label{energy_v_L_1}
\fint_{B_{r/8}} |Dv|\, dx
&\leq&  \bar{\varPsi}^{-1}_r\bigg(\fint_{B_{r/8}} \bar{\varPsi}_r\big(|Dv|\big) \, dx \bigg)  \notag\\
&\leq& \bar{\varPsi}^{-1}_r\bigg(\Lambda^{q} \fint_{B_{r/8}} \bar{\varPsi}_r\big(|Du|\big) \, dx\bigg)    \notag\\
&\leq& \bar{\varPsi}^{-1}_r\Bigg(\Lambda^{q} c\,\bar{\varPsi}_r\,\bigg(\fint_{B_{r/4}} |Du| \, dx\bigg)\Bigg)\notag\\
&\leq& c_0 \fint_{B_{r/4}} |Du| \, dx
\end{eqnarray}
with $c_0 = c_0\,\big(\texttt{data}\big) \geq 1$, where we used  \eqref{est_>1} in the last step.

We now present the gradient estimates for the minimizer of the locally frozen functional.
To this end, we recall the following gradient bound for the minimizer of a functional with Orlicz growth (cf. \cite{GiPT17}), which is also suitable for our setting when $a(x)$ and $\mathbb{A}(x,u)$ of the functional \eqref{main_funl} are locally frozen.

\begin{lemma}\label{est_v}\cite[Theorem 2.9]{GiPT17}
For any $N$-function $\phi \in \mathbf{\Phi}_{p,q}$ and a fixed $\bar{\mathbb{A}}$ satisfying \eqref{cond_A}, let $v \in W^{1,1}(B_R,\mathbb{R}^N)$ be a minimizer of the functional $\mathcal{F}(v, B_R) = \int_{B_R} \phi\,(|Dv|_{\bar{\mathbb{A}}}) \, dx$. Then for any ball $B_\rho \subset B_R$, there holds the local Lipschitz estimate as follows:
\begin{equation*}
\sup_{x \in B_{\rho/2}} \phi\,\big(|Dv|\big) \leq c\fint_{B_{\rho}} \phi\,\big(|Dv|\big) \, dx,
\end{equation*}
where $c = c(n,N,\Lambda, p,q) \geq 1$. Furthermore, if $\phi \in \mathbf{\Phi}_{p,q}$ satisfies \eqref{cond_Holder_Phi}, then there exists $\sigma'=\sigma'(n,\beta, N,\Lambda,p,q)\in (0,1)$ such that, for any $\delta\in(0,1)$ we have the Campanato-type estimate as
\begin{equation*}
\fint_{B_{\delta  \rho}} \big|V_{\phi}\big(Dv\big) - \big\langle V_{\phi}\big(Dv\big)\big\rangle _{B_{\delta  \rho}}\big|^2 \, dx     \leq  c\,\delta^{\,\sigma'}\fint_{B_{\rho}} \big|V_{\phi}\big(Dv\big) - \big\langle V_{\phi}\big(Dv\big)\big\rangle _{B_{\rho}}\big|^2 \, dx,
\end{equation*}
where $c = c\,(n,N,\Lambda, p,q) \geq 1$.
\end{lemma}

With Lemma \ref{est_v} in  hand, we can show the gradient estimate of the minimizers of the local frozen functional \eqref{fun_v} in $L^1$-sense. More precisely, we have the following:

\begin{lemma}
For ${\varphi_1}, {\varphi_2} \in \mathbf{\Phi}_{p,q}$ with ${\varphi_1} \prec {\varphi_2}$, let $v \in W^{1,1}\big(B_{r/8}, \mathbb{R}^N\big)$ be a minimizer of the local frozen functional \eqref{fun_v}. Then there exists a constant $c_7 = c_7(n,N,\Lambda, p,q) \geq 1$ such that
\begin{equation}\label{lip_v}
\sup_{x \in B_{r/32}} |Dv| \leq c_7\,\fint_{B_{r/8}} |Dv| \, dx.
\end{equation}
Furthermore, if ${\varphi_1}, {\varphi_2} \in \mathbf{\Phi}_{p,q}$ satisfy \eqref{cond_Holder_Phi}, then there exists  $\sigma_1 = \sigma_1(n,\beta, N,\Lambda, p,q) \in (0,1)$ such that for any $\theta \in (0,1/4)$ one has
\begin{equation}\label{est_camp_v}
\fint_{B_{\theta r/8}} \big|Dv - \langle Dv\rangle _{B_{\theta r/8}}\big| \, dx     \leq  c\,\theta^{\,\sigma_1}\fint_{B_{r/8}} |Dv| \, dx,
\end{equation}
where $c = c\,(n,N,\Lambda, p,q) \geq 1$.
\end{lemma}

\begin{proof}

Note that $\bar{\varPsi}_r(\cdot)=\varphi_1(\cdot)+\bar{a}_r\,\varphi_2(\cdot)\in\mathbf{\Phi}_{p,q}$ and the frozen tensor $\bar{\mathbb{A}}_r^u$ satisfying \eqref{cond_A}. Then, it follows from Lemma \ref{est_v} that
\begin{equation}\label{boundness_v}
\sup_{x \in B_{r/32}} \bar{\varPsi}_r\big(|Dv|\big) \leq c\fint_{B_{r/16}} \bar{\varPsi}_r\big(|Dv|\big) \, dx.
\end{equation}
The reverse Jensen inequality \eqref{ineq_rev_jense} shows that $\fint_{B_{r/16}} \bar{\varPsi}_r\big(|  Dv|\big) \, dx \lesssim \bar{\varPsi}_r\Big(\fint_{B_{r/8}} |Dv| \, dx\Big)$. Therefore, we get
\begin{equation*}
\sup_{x \in B_{r/32}} \bar{\varPsi}_r\big(|Dv|\big) \leq c\,\bar{\varPsi}_r\bigg(\fint_{B_{r/8}} |Dv| \, dx\bigg)
\end{equation*}
Combining this with the non-decreasing property of $t \mapsto \bar{\varPsi}^{-1}_r(t)$, for any $x \in B_{r/32}$ we conclude that
\begin{eqnarray*}
|Dv(x)| &=&  \bar{\varPsi}^{-1}_r \Big(\bar{\varPsi}_r\big(|Dv|\big)\Big) \\
&\leq& \bar{\varPsi}^{-1}_r\bigg(\sup_{x \in B_{r/32}} \bar{\varPsi}_r\big(|Dv|\big)\bigg) \notag\\
&\leq& \bar{\varPsi}^{-1}_r\Bigg(c\,\bar{\varPsi}_r\,\bigg(\fint_{B_{r/8}} |Dv| \, dx\bigg)\Bigg)\notag\\
&\leq& c\fint_{B_{r/8}} |Dv| \, dx,
\end{eqnarray*}
where we used  \eqref{est_>1} for $\bar{\varPsi}^{-1}_r(\cdot)$ in the last inequality. This completes the proof of \eqref{lip_v}.

We are next to prove \eqref{est_camp_v}. By using the Campanato-type estimate of Lemma \ref{est_v}, we deduce that for any $\delta \in (0,1/2)$,
\begin{equation}\label{est_decay}
\fint_{B_{\delta  r/16}} \big|V_{\bar{\varPsi}_r}\big(Dv\big) - \big\langle V_{\bar{\varPsi}_r}\big(Dv\big)\big\rangle _{B_{\delta  r/16}}\big|^2 \, dx     \leq  c\,\delta^{\,\sigma'}\fint_{B_{r/16}} \big|V_{\bar{\varPsi}_r}\big(Dv\big) - \big\langle V_{\bar{\varPsi}_r}\big(Dv\big)\big\rangle _{B_{r/16}}\big|^2 \, dx,
\end{equation}
where $c=c\,(n,N,\Lambda,p,q)>0$.
Thanks to the continuity of $V_{\bar{\varPsi}_r}(\cdot)$, we make use of the mean value theorem of continuous function to find a fixed matrix $\textbf{H} \in \mathbb{R}^{N \times n}$ such that $V_{\bar{\varPsi}_r}\big(\textbf{H}\big) = \big\langle V_{\bar{\varPsi}_r}\big(Dv\big)\big\rangle _{B_{\delta r/16}}$.
By \cite[Lemma 2.2]{BaBL21} we see that
\begin{equation*}
{\bar{\varPsi}_r}\big(|Dv - \textbf{H}|\big)
\leq  \delta^{\frac{\sigma'}{2}} \bar{\varPsi}_r\big(|Dv|\big)
+ c\,\delta^{-\frac{\sigma'}{2}} \big|V_{\bar{\varPsi}_r}\big(Dv\big) - \big\langle V_{\bar{\varPsi}_r}\big(Dv\big)\big\rangle _{B_{\delta r/16}}\big|^2
\end{equation*}
for $c = c(n,N, p,q) > 0$. Setting $\theta = \frac{\delta}{2}$, let us combine the above inequality with \eqref{boundness_v} and \eqref{est_decay} to deduce that
\begin{eqnarray*}
\bar{\varPsi}_r\, \bigg(\fint_{B_{\theta r/8}} \big|Dv -\textbf{H}\big|\, dx\bigg)
&\leq&  \fint_{B_{\delta  r/16}} {\bar{\varPsi}_r}\big(|Dv-\textbf{H}|\big) \, dx \notag\\
&\leq&  \delta^{\frac{\sigma'}{2}} \fint_{B_{\delta  r/16}} \bar{\varPsi}_r\big(|Dv|\big)\,dx
+ c\,\delta^{-\frac{\sigma'}{2}} \fint_{B_{\delta  r/16}} \big|V_{\bar{\varPsi}_r}\big(Dv\big) - \big\langle V_{\bar{\varPsi}_r}\big(Dv\big)\big\rangle _{B_{\delta r/16}}\big|^2dx \notag\\
&\leq&  \delta^{\frac{\sigma'}{2}}\,\big\|\bar{\varPsi}_r\big(|Dv|\big)\big\|_{L^\infty(B_{\delta r/16})}  +  c\,\delta^{\frac{\sigma'}{2}}\fint_{B_{r/16}} \big|V_{\bar{\varPsi}_r}\big(Dv\big) - \big\langle V_{\bar{\varPsi}_r}\big(Dv\big)\big\rangle _{B_{r/16}}\big|^2 \, dx\notag\\
&\leq& c\,\delta^{\frac{\sigma'}{2}}\fint_{B_{r/16}} \bar{\varPsi}_r\big(|Dv|\big)\,dx.
\end{eqnarray*}
According to  \eqref{ineq_rev_jense}, we obtain
\begin{equation*}
\bar{\varPsi}_r\, \bigg(\fint_{B_{\theta r/8}} \big|Dv -\textbf{H}\big|\, dx\bigg)
\leq  c\,\delta^{\frac{\sigma'}{2}}\fint_{B_{r/16}} \bar{\varPsi}_r\big(|Dv|\big) \, dx
\leq  c\,\delta^{\frac{\sigma'}{2}}\,\bar{\varPsi}_r\,\bigg(\fint_{B_{r/8}} |Dv| \, dx\bigg)
\end{equation*}
Using the non-decreasing nature, \eqref{est_>1} and \eqref{est_<1} for $\bar{\varPsi}_r^{-1}(\cdot)$, we arrive at that
\begin{eqnarray*}
\fint_{B_{\theta r/8}} \big|Dv -\textbf{H}\big|\, dx &\leq& \bar{\varPsi}^{-1}_r\Bigg(c\,\delta^{\frac{\sigma'}{2}}\bar{\varPsi}_r\,\bigg(\fint_{B_{r/8}} |Dv| \, dx\bigg)\,\Bigg)\notag\\
&\leq& c\,\theta^{\frac{\sigma'}{2q}}\fint_{B_{r/8}} |Dv| \, dx.
\end{eqnarray*}
With the above estimate in hand, a direct calculation shows that
\begin{eqnarray*}
\fint_{B_{\theta r/8}} \big|Dv - \langle Dv\rangle _{B_{\theta r/8}}\big|
&\leq& \fint_{B_{\theta r/8}} \big|Dv -\textbf{H}\big|\, dx  + \fint_{B_{\theta r/8}} \big|\textbf{H} - \langle Dv\rangle _{B_{\theta r/8}}\big| \, dx\notag\\
&\leq& 2\fint_{B_{\theta r/8}} \big|Dv -\textbf{H}\big|\, dx \notag\\
&\leq& c\,\theta^{\frac{\sigma'}{2q}}\fint_{B_{r/8}} |Dv| \, dx,
\end{eqnarray*}
which leads to the desired result \eqref{est_camp_v} by letting $\sigma_1 = \frac{\sigma'}{2q}$.
\end{proof}

To make further comparison estimates, it is necessary to use the approximating harmonic way between the original functional and the local frozen functional. Before stating the harmonic approximation lemma, let us present the concept of quasi-isotropic $(p, q)$-growth concerning the integrand $F(\cdot)$.

\begin{definition}\label{def_p_q} \cite[Definition 4.1]{HaO23}
We say that  $F: \mathbb{R}^{m} \rightarrow[0, \infty)$ has the property of quasi-isotropic $(p, q)$-growth for $1 < p \leq q < \infty$, if it satisfies the following:
\begin{enumerate}
\item[(i)] $F \in C^1(\mathbb{R}^m) \cap C^2(\mathbb{R}^m \setminus \{\mathbf{0}\})$ such that $F(\textbf{0}) = DF(\textbf{0}) = 0$ and $L^{-1} \leq F(e) \leq L$  for $e \in \mathbb{R}^{m}$ with $|e| = 1$;
\item[(ii)] $\langle D^2 F(\zeta)\, \xi, \xi \rangle  \gtrsim |D^2 F(\zeta)|\,|\xi|^2 $  for any $\zeta, \xi \in \mathbb{R}^{m}$;
\item[(iii)] For any $0 < s \leq t < \infty$, there holds $\frac{|D^2 F (se)|}{s^{p-2}} \lesssim \frac{|D^2 F (te)|}{t^{p-2}}$ and $\frac{|D^2 F (te)|}{t^{q-2}} \lesssim \frac{|D^2 F (se)|}{s^{q-2}}$ for $e \in \mathbb{R}^{m}$ with $|e| = 1$.
\end{enumerate}
\end{definition}

We  now introduce the so-called harmonic approximation lemma. Here, we only provide a version of integral functional concerning harmonic approximation (cf. \cite[Lemma 4.13]{HaO23}), although there are two harmonic approximation ways to cover the integral functional and elliptic equations in the original context.

\begin{lemma}\label{lem_app}
Let $F: \mathbb{R}^{N \times n} \rightarrow[0, \infty)$ have quasi-isotropic $(p, q)$-growth satisfying $\Lambda^{-1}\,\phi(|\xi|) \leq F(\xi) \leq \Lambda \,\phi(|\xi|)$ for $\xi \in \mathbb{R}^{N \times n}$ and $\phi \in \mathbf{\Phi}_{p,q}$.
Assume $\phi(|Du|) \in L^{1+\varepsilon}(B_\rho)$ for $\varepsilon > 0$, and we write
$$
E(u,\rho) := \bigg(\fint_{B_\rho} \Big(\phi(|Du|)\Big)^{1+\varepsilon}dx\bigg)^\frac{1}{1+\varepsilon}.
$$
If $u$ is a local almost minimizer of the functional $\int_{B_\rho} F(Du)\, dx$ in the sense that for any $\eta \in W_0^{1, \infty}(B_\rho, \mathbb{R}^N)$ and  $b > 0$ it holds that
\begin{equation*}
\fint_{B_\rho} F(Du)\, dx   \leq   \fint_{B_\rho} F(Du+D\eta)\, dx  +  \delta E(u, \rho)\,\bigg(\frac{\|D \eta\|_{L^{\infty}\left(B_\rho\right)}}{\phi^{-1}\big(E(u,\rho)\big)}+1\bigg)^b,
\end{equation*}
then there exists  $v \in u + W_0^{1, \phi}\big(B_\rho, \mathbb{R}^N\big)$ as the minimizer of functional $\int_{B_\rho} F(Dv)\, dx$ with the estimate
\begin{equation*}
\fint_{B_\rho}  \frac{\phi'\big(|Du|+ |Dv|\big)}{|Du| + |Dv|} |Du - Dv|^2  dx \leq c\,\delta^\frac{\varepsilon p}{b + \varepsilon p}E(u, \rho),
\end{equation*}
where $c = c\,(p, q, n, N, \Lambda, \varepsilon) \geq 1$.
\end{lemma}

As an application of Lemma \ref{lem_app}, we are next to  prove that the integrand $\bar{\varPsi}_r\big(|\zeta|_{\bar{\mathbb{A}}_r^u}\big) $ in functional \eqref{fun_v} actually has the property of quasi-isotropic $(p, q)$-growth, which is suitable to make harmonic approximation. More precisely, we have

\begin{lemma}\label{rem_app}
Let  $\bar{\varPsi}_r(\cdot) = {\varphi_1}(\cdot) + \bar{a}_r {\varphi_2}(\cdot) \in \mathbf{\Phi}_{p,q}$ and $\bar{\mathbb{A}}_r^u := \mathbb{A}(x_0, \langle u \rangle_{B_r})$ satisfying \eqref{cond_A} for the ball $B_r \subset \Omega $. Then the function $\zeta \mapsto \bar{\varPsi}_r\big(|\zeta|_{\bar{\mathbb{A}}_r^u}\big)$ has quasi-isotropic $(p, q)$-growth.
\end{lemma}

\begin{proof}

We first observe that $\zeta \mapsto \bar{\varPsi}_r\big(|\zeta|_{\bar{\mathbb{A}}_r^u}\big) \in C^1(\mathbb{R}^{N \times n}) \cap C^2(\mathbb{R}^{N \times n} \setminus \{\mathbf{0}\})$ since ${\varphi_1}, \varphi_2 \in C^1[0,+\infty) \cap C^2(0,+\infty)$. Moreover, it follows from \eqref{cond_A} and \eqref{est_<1}-\eqref{est_>1} that $c^{-1} \leq \bar{\varPsi}_r\big(|e|_{\bar{\mathbb{A}}_r^u}\big) \leq c$ for $|e| = 1$ and some $c \geq 1$. Obviously, it holds that $\bar{\varPsi}_r\big(|\mathbf{0}|_{\bar{\mathbb{A}}_r^u}\big) = 0$. We also claim that $D_\zeta \bar{\varPsi}_r\big(|\mathbf{0}|_{\bar{\mathbb{A}}_r^u}\big) = \mathbf{0}$. In fact,
for $\zeta \neq \mathbf{0}$, a direct computation shows that $D_\zeta \bar{\varPsi}_r\big(|\zeta|_{\bar{\mathbb{A}}_r^u}\big) = \bar{\varPsi}'_r\big(|\zeta|_{\bar{\mathbb{A}}_r^u}\big) \frac{\bar{\mathbb{A}}_r^u \zeta}{|\zeta|_{\bar{\mathbb{A}}_r^u}}$, which yields $|D_\zeta \bar{\varPsi}_r\big(|\zeta|_{\bar{\mathbb{A}}_r^u}\big) | \lesssim \bar{\varPsi}'_r\big(|\zeta|_{\bar{\mathbb{A}}_r^u}\big)$. Consequently, we obtain $D_\zeta \bar{\varPsi}_r\big(|\mathbf{0}|_{\bar{\mathbb{A}}_r^u}\big) := \lim_{\zeta\rightarrow \mathbf{0}} D_\zeta \bar{\varPsi}_r\big(|\zeta|_{\bar{\mathbb{A}}_r^u}\big) = \mathbf{0}$ since $\bar{\varPsi}'_r\big(|\zeta|_{\bar{\mathbb{A}}_r^u}\big)\rightarrow 0$ as $\zeta\rightarrow \mathbf{0}$. So we complete the verification of Definition \ref{def_p_q} (i).

We are now to verify that $\bar{\varPsi}_r\big(|\zeta|_{\bar{\mathbb{A}}_r^u}\big)$ meets Definition \ref{def_p_q} (ii). Since $t \mapsto \bar{\varPsi}_r(t)$ is a convex, non-decreasing function of class $C^2(0,+\infty)$, a direct calculation shows
\begin{equation}\label{D_2}
D^2_\zeta \bar{\varPsi}_r\big(|\zeta|_{\bar{\mathbb{A}}_r^u}\big) = \bigg(\bar{\varPsi}''_r\big(|\zeta|_{\bar{\mathbb{A}}_r^u}\big)-\frac{\bar{\varPsi}'_r\big(|\zeta|_{\bar{\mathbb{A}}_r^u}\big)}{|\zeta|_{\bar{\mathbb{A}}_r^u}}\bigg) \frac{\bar{\mathbb{A}}_r^u \zeta \otimes \bar{\mathbb{A}}_r^u \zeta}{|\zeta|_{\bar{\mathbb{A}}_r^u}^2} + \frac{\bar{\varPsi}'_r\big(|\zeta|_{\bar{\mathbb{A}}_r^u}\big)}{|\zeta|_{\bar{\mathbb{A}}_r^u}}\bar{\mathbb{A}}_r^u.
\end{equation}
Note that $\langle (\bar{\mathbb{A}}_r^u \zeta \otimes \bar{\mathbb{A}}_r^u \zeta)\xi, \xi \rangle = \langle \bar{\mathbb{A}}_r^u \zeta, \xi \rangle^2$ (cf. \cite[(1.152)]{Ho00}), then we get the following facts:
\begin{eqnarray}\label{rem_cond_2}
\langle D^2_\zeta \bar{\varPsi}_r\big(|\zeta|_{\bar{\mathbb{A}}_r^u}\big)\, \xi, \xi \rangle
&=& \bigg(\bar{\varPsi}''_r\big(|\zeta|_{\bar{\mathbb{A}}_r^u}\big)-\frac{\bar{\varPsi}'_r\big(|\zeta|_{\bar{\mathbb{A}}_r^u}\big)}{|\zeta|_{\bar{\mathbb{A}}_r^u}}\bigg) \frac{\langle \bar{\mathbb{A}}_r^u \zeta, \xi \rangle^2}{|\zeta|_{\bar{\mathbb{A}}_r^u}^2} + \frac{\bar{\varPsi}'_r\big(|\zeta|_{\bar{\mathbb{A}}_r^u}\big)}{|\zeta|_{\bar{\mathbb{A}}_r^u}} \langle \bar{\mathbb{A}}_r^u \xi,\xi \rangle \notag\\
&=& \bar{\varPsi}''_r\big(|\zeta|_{\bar{\mathbb{A}}_r^u}\big) \frac{\langle \bar{\mathbb{A}}_r^u \zeta, \xi \rangle^2}{|\zeta|_{\bar{\mathbb{A}}_r^u}^2} + \frac{\bar{\varPsi}'_r\big(|\zeta|_{\bar{\mathbb{A}}_r^u}\big)}{|\zeta|_{\bar{\mathbb{A}}_r^u}} \bigg(\langle \bar{\mathbb{A}}_r^u \xi,\xi \rangle - \frac{\langle \bar{\mathbb{A}}_r^u \zeta, \xi \rangle^2}{|\zeta|_{\bar{\mathbb{A}}_r^u}^2}\bigg)\,.
\end{eqnarray}
To estimate the right-hand side of \eqref{rem_cond_2}, we notice that $\frac{\bar{\varPsi}'_r(|\zeta|_{\bar{\mathbb{A}}_r^u})}{|\zeta|_{\bar{\mathbb{A}}_r^u}} \geq \frac{\bar{\varPsi}''_r(|\zeta|_{\bar{\mathbb{A}}_r^u})}{q-1} $  due to \eqref{cond_p_q_2}, and the Cauchy-Schwarz inequality yields $\langle \bar{\mathbb{A}}_r^u \zeta, \xi \rangle^2 \leq \langle \bar{\mathbb{A}}_r^u \zeta, \zeta \rangle \langle \bar{\mathbb{A}}_r^u \xi, \xi \rangle = |\zeta|_{\bar{\mathbb{A}}_r^u}^2 \langle \bar{\mathbb{A}}_r^u \xi, \xi \rangle$, which means $\langle \bar{\mathbb{A}}_r^u \xi,\xi \rangle - \frac{\langle \bar{\mathbb{A}}_r^u \zeta, \xi \rangle^2}{|\zeta|_{\bar{\mathbb{A}}_r^u}^2} \geq 0$. Then \eqref{rem_cond_2} is  estimated as follows:
\begin{eqnarray}\label{D_2_1}
\langle D^2_\zeta \bar{\varPsi}_r\big(|\zeta|_{\bar{\mathbb{A}}_r^u}\big)\, \xi, \xi \rangle
&\geq& \bar{\varPsi}''_r\big(|\zeta|_{\bar{\mathbb{A}}_r^u}\big) \frac{\langle \bar{\mathbb{A}}_r^u \zeta, \xi \rangle^2}{|\zeta|_{\bar{\mathbb{A}}_r^u}^2}
+ \frac{\bar{\varPsi}''_r\big(|\zeta|_{\bar{\mathbb{A}}_r^u}\big)}{q-1}  \bigg(\langle \bar{\mathbb{A}}_r^u \xi,\xi \rangle - \frac{\langle \bar{\mathbb{A}}_r^u \zeta, \xi \rangle^2}{|\zeta|_{\bar{\mathbb{A}}_r^u}^2}\bigg)\notag\\
&\geq& \bar{\varPsi}''_r\big(|\zeta|_{\bar{\mathbb{A}}_r^u}\big) \frac{\langle \bar{\mathbb{A}}_r^u \zeta, \xi \rangle^2}{|\zeta|_{\bar{\mathbb{A}}_r^u}^2}
+ \min\big\{1, \tfrac{1}{q-1}\big\} \bar{\varPsi}''_r\big(|\zeta|_{\bar{\mathbb{A}}_r^u}\big)  \bigg(\langle \bar{\mathbb{A}}_r^u \xi,\xi \rangle - \frac{\langle \bar{\mathbb{A}}_r^u \zeta, \xi \rangle^2}{|\zeta|_{\bar{\mathbb{A}}_r^u}^2}\bigg)\notag\\
&=& \Big(1 - \min\big\{1, \tfrac{1}{q-1}\big\} \Big) \bar{\varPsi}''_r\big(|\zeta|_{\bar{\mathbb{A}}_r^u}\big) \frac{\langle \bar{\mathbb{A}}_r^u \zeta, \xi \rangle^2}{|\zeta|_{\bar{\mathbb{A}}_r^u}^2}  +  \min\big\{1, \tfrac{1}{q-1}\big\} \bar{\varPsi}''_r\big(|\zeta|_{\bar{\mathbb{A}}_r^u}\big) \langle \bar{\mathbb{A}}_r^u \xi,\xi \rangle\notag\\
&\geq& \min\big\{1, \tfrac{1}{q-1}\big\} \bar{\varPsi}''_r\big(|\zeta|_{\bar{\mathbb{A}}_r^u}\big) \langle \bar{\mathbb{A}}_r^u \xi,\xi  \rangle.
\end{eqnarray}
On the other hand, thanks to  \eqref{cond_A} and \eqref{cond_p_q_2} it follows from \eqref{D_2} that
\begin{eqnarray}\label{D_2_3}
| D^2_\zeta \bar{\varPsi}_r\big(|\zeta|_{\bar{\mathbb{A}}_r^u}\big)|
&\leq&  \bigg(\bar{\varPsi}''_r\big(|\zeta|_{\bar{\mathbb{A}}_r^u}\big)  +  \frac{\bar{\varPsi}'_r\big(|\zeta|_{\bar{\mathbb{A}}_r^u}\big)}{|\zeta|_{\bar{\mathbb{A}}_r^u}}\bigg) \frac{|\bar{\mathbb{A}}_r^u \zeta \otimes \bar{\mathbb{A}}_r^u \zeta|}{|\zeta|_{\bar{\mathbb{A}}_r^u}^2} +  \frac{\bar{\varPsi}'_r\big(|\zeta|_{\bar{\mathbb{A}}_r^u}\big)}{|\zeta|_{\bar{\mathbb{A}}_r^u}}|\bar{\mathbb{A}}_r^u|\notag\\
&\leq&  2 \Lambda^4\, \bigg(\bar{\varPsi}''_r\big(|\zeta|_{\bar{\mathbb{A}}_r^u}\big)  +  \frac{\bar{\varPsi}'_r\big(|\zeta|_{\bar{\mathbb{A}}_r^u}\big)}{|\zeta|_{\bar{\mathbb{A}}_r^u}}\bigg)\notag\\
&\leq& 2 \,\Big(1 + \tfrac{1}{p-1}\Big)\,\Lambda^4\,\bar{\varPsi}''_r\big(|\zeta|_{\bar{\mathbb{A}}_r^u}\big),
\end{eqnarray}
which combines \eqref{D_2_1} and \eqref{cond_p_q_2} implies $\langle D^2_\zeta \bar{\varPsi}_r\big(|\zeta|_{\bar{\mathbb{A}}_r^u}\big)\, \xi, \xi \rangle \geq c^*\, | D^2_\zeta \bar{\varPsi}_r\big(|\zeta|_{\bar{\mathbb{A}}_r^u}\big)|\,|\xi|^2 $ with $c^* = \frac{(p-1)\min\{1,\frac{1}{q-1}\}}{2p\Lambda^5}$.

Next, we are to verify Condition (iii) of Definition \ref{def_p_q}. It follows from \eqref{D_2_1} and \eqref{cond_A} that $|D^2_\zeta \bar{\varPsi}_r\big(|\zeta|_{\bar{\mathbb{A}}_r^u}\big)|\,|\xi|^2 \geq c\,\bar{\varPsi}''_r\big(|\zeta|_{\bar{\mathbb{A}}_r^u}\big) |\xi|^2$, which combines \eqref{D_2_3} yields $| D^2_\zeta \bar{\varPsi}_r\big(|\zeta|_{\bar{\mathbb{A}}_r^u}\big)| \approx |\bar{\varPsi}''_r\big(|\zeta|_{\bar{\mathbb{A}}_r^u}\big)|$; moreover, \eqref{cond_A} also implies $|be|_{\bar{\mathbb{A}}_r^u} \thickapprox b$ for $b > 0$ and $|e| = 1$, then $\frac{|D^2_\zeta \bar{\varPsi}_r\big(|be|_{\bar{\mathbb{A}}_r^u}\big)|}{b^{p-2}} \approx  \frac{\bar{\varPsi}''_r\big(|be|_{\bar{\mathbb{A}}_r^u}\big)}{|be|_{\bar{\mathbb{A}}_r^u}^{p-2}}  \approx  \frac{\bar{\varPsi}_r\big(b\big)}{b^{p}} $. Therefore, by \eqref{cond_p_q_2}  we obtain that for $0 < s \leq t < \infty$,	\begin{equation*}
\frac{|D^2_\zeta \bar{\varPsi}_r\big(|se|_{\bar{\mathbb{A}}_r^u}\big)|}{s^{p-2}}
\leq  c\frac{\bar{\varPsi}_r\big(s\big)}{s^{p}}
=  c\frac{\bar{\varPsi}_r\big(\frac{s}{t}t\big)}{s^{p}}
\leq  c\,\Big(\frac{s}{t}\Big)^{p}\frac{\bar{\varPsi}_r\big(t\big)}{s^{p}}
=  c\frac{\bar{\varPsi}_r\big(t\big)}{t^{p}}
\leq  c\frac{|D^2_\zeta \bar{\varPsi}_r\big(|te|_{\bar{\mathbb{A}}_r^u}\big)|}{t^{p-2}}.
\end{equation*}
Similarly, there hold the following:
\begin{equation*}
\frac{|D^2_\zeta \bar{\varPsi}_r\big(|te|_{\bar{\mathbb{A}}_r^u}\big)|}{t^{q-2}}
\leq  c\frac{\bar{\varPsi}_r\big(t\big)}{t^{q}}
=  c\frac{\bar{\varPsi}_r\big(\frac{t}{s}s\big)}{t^{q}}
\leq  c\,\Big(\frac{t}{s}\Big)^{q}\frac{\bar{\varPsi}_r\big(s\big)}{t^{q}}
=  c\frac{\bar{\varPsi}_r\big(s\big)}{s^{q}}
\leq  c\frac{|D^2_\zeta \bar{\varPsi}_r\big(|se|_{\bar{\mathbb{A}}_r^u}\big)|}{s^{q-2}}.
\end{equation*}
This finishes the verification of Definition \ref{def_p_q} (iii).

In summary, we conclude that $\bar{\varPsi}_r\big(|\zeta|_{\bar{\mathbb{A}}_r^u}\big)$ has quasi-isotropic $(p, q)$-growth.
\end{proof}

\section{ Several comparison estimates}
\setcounter{equation}{0}
\setcounter{theorem}{0}

We devote this section to a few of comparison estimates between the original functional and various locally regularized functionals with frozen coefficients. With the harmonic approximation lemma from the previous section, we can deal with the comparison estimates between $Du$ and $Dv$ in the two cases of $a(\cdot) \in C(\Omega)$ or $a(\cdot) \in C^{0,\beta }(\Omega)$.  To this end, we  denote an excess energy functional of $u$ by
\begin{equation*}
\mathcal{E}(y, r) := \fint_{B_r(y)} |u - \langle u \rangle_{B_r(y)}| \, dx.
\end{equation*}
In this section, let us write $B_r = B_r(y)$ and $\mathcal{E}(x_0, r) = \mathcal{E}(r)$ if the center $y$ is not essential for the sake of simplicity.
Let us now state the first approximating estimate as follows.

\begin{lemma}\label{lem_comp}
For ${\varphi_1}, {\varphi_2} \in \mathbf{\Phi}_{p,q}$ with ${\varphi_1} \prec  {\varphi_2}$ and $\mathbb{A}(\cdot,\cdot)$ satisfying \eqref{cond_A}, let $u \in W^{1,1}(\Omega,\mathbb{R}^N)$ be a local minimizer of the functional \eqref{main_funl}, and $v \in u + W_0^{1,1}\big(B_{r/8}, \mathbb{R}^N\big)$ be a minimizer of the local frozen functional \eqref{fun_v}. If $\mathbb{A}(\cdot,\cdot)$ and $a(\cdot)$ are continuous with the modulus of continuity
shown in \eqref{cond_continuous} satisfying the gap condition \eqref{gap_2}, then for any $B_r \Subset \Omega$  we have
\begin{equation*}
\fint_{B_{r/8}} \big|V_{\bar{\varPsi}_r}(Du) - V_{\bar{\varPsi}_r}(Dv)\big|^2\, dx \leq c_1\,\Big(\omega_{\mathbb{A}}^{\,\mu}\big(r + \mathcal{E}(r)\big) + \mathcal{G}(r) + r^\frac{n\varepsilon^2}{1+\varepsilon}\Big)^{\kappa_1}\,\bar{\varPsi}_r\,\bigg(\fint_{B_{r/4}} |Du|\, dx\bigg) ,
\end{equation*}
where $c_1 = c_1\,(\texttt{data}_0) \geq 1$, $\mathcal{G}(r)$ as shown in \eqref{Gr} and $\mu = \min\{\frac{\varepsilon}{1+\varepsilon},\frac{1}{2}\}, {\kappa_1} = \frac{\varepsilon p}{q + 1+ \varepsilon p}$ for $\varepsilon <\varepsilon_0$ with $\varepsilon_0$ as in Lemma \ref{lem_high_int}.
\end{lemma}

\begin{proof}
Thanks to Lemma \ref{lem_high_int}, we see that $\bar{\varPsi}_r\big(|Du|\big) \in L^{1 + \varepsilon}(B_r)$, which yields the fact
\begin{equation}\label{def_E}
E(u, r)= \bigg(\fint_{B_{r/8}} \Big(\bar{\varPsi}_r\big(|Du|\big)\Big)^{1+\varepsilon} dx\bigg)^\frac{1}{1+\varepsilon}<\infty.
\end{equation}
To employ the harmonic approximation Lemma  \ref{lem_app},
we claim that there is ${\delta}_1(r) := \omega_{\mathbb{A}}^{\,\mu}\big(r + \mathcal{E}(r)\big) + \mathcal{G}(r) + r^\frac{n\varepsilon^2}{1+\varepsilon}$  such that for any $\eta \in W_0^{1, \infty}(B_{r/8}, \mathbb{R}^N)$ we have
\begin{equation}\label{u_al_min}
\fint_{B_{r/8}} \bar{\varPsi}_r\big(|Du|_{\bar{\mathbb{A}}_r^u}\big)\, dx   \leq   \fint_{B_{r/8}} \bar{\varPsi}_r\big(|Du + D\eta|_{\bar{\mathbb{A}}_r^u}\big)\, dx  +  c\, {\delta}_1(r)\, E(u, r)\,\bigg(\frac{\|D\eta\|_{L^\infty(B_{r/8})}}{\bar{\varPsi}^{-1}_r\big(E(u, r)\big)}+1\bigg)^{q+1}.
\end{equation}
With the validity of \eqref{u_al_min} and $\bar{\varPsi}_r\big(|\cdot|_{\bar{\mathbb{A}}_r^u}\big)$ as a quasi-isotropic $(p, q)$-growth shown in Lemma \ref{rem_app}, then we use the harmonic approximation of Lemma \ref{lem_app} to find
\begin{equation*}
\fint_{B_{r/8}} \frac{\bar{\varPsi}'_r\big(|Du|+ |Dv|\big)}{|Du| + |Dv|} |Du - Dv|^2 dx \leq c\,\Big({\delta}_1(r)\Big)^\frac{\varepsilon p}{q + 1+ \varepsilon p}E(u, r).
\end{equation*}
This combines with the equivalence in \eqref{est_V} and the inverse H\"older inequality \eqref{ineq_rev} to get that
\begin{equation*}
\fint_{B_{r/8}} \big|V_{\bar{\varPsi}_r}(Du) - V_{\bar{\varPsi}_r}(Dv)\big|^2\, dx \leq c\,\Big({\delta}_1(r)\Big)^\frac{\varepsilon p}{q + 1+ \varepsilon p}\,\bar{\varPsi}_r\,\bigg(\fint_{B_{r/4}} |Du|\, dx\bigg)\,,
\end{equation*}
which completes the proof of Lemma \ref{lem_comp} with ${\kappa_1} = \frac{\varepsilon p}{q + 1+ \varepsilon p}$.

We are now in a position to prove \eqref{u_al_min}. Obviously, $\bar{\varPsi}_r\big(|Du|_{{\mathbb{A}}^u}\big) \leq \varPsi\big(x, |Du|_{\mathbb{A}^u}\big)$ for any $x \in B_{r}$ due to Definition of $\bar{\varPsi}_r(\cdot)$, and we deduce that
\begin{eqnarray}\label{est_al_min}
\fint_{B_{r/8}} \bar{\varPsi}_r\big(|Du|_{\bar{\mathbb{A}}_r^u}\big)\, dx
& = &    \fint_{B_{r/8}} \bar{\varPsi}_r\big(|Du|_{\bar{\mathbb{A}}_r^u}\big)\, dx - \fint_{B_{r/8}} \bar{\varPsi}_r\big(|Du|_{\mathbb{A}^u}\big)\, dx + \fint_{B_{r/8}} \bar{\varPsi}_r\big(|Du|_{\mathbb{A}^u}\big)\, dx \notag\\
& \leq &    \fint_{B_{r/8}} \big|\bar{\varPsi}_r\big(|Du|_{\bar{\mathbb{A}}_r^u}\big) -  \bar{\varPsi}_r\big(|Du|_{\mathbb{A}^u}\big)\big|\, dx + \fint_{B_{r/8}} \varPsi\big(x, |Du|_{\mathbb{A}^u}\big)\, dx.
\end{eqnarray}
In the following we prove \eqref{u_al_min} by parting in two steps.
\vspace{3pt}

\noindent\emph{Step 1.}\ \ We are first to estimate the first-term on the right-hand side of \eqref{est_al_min}.
Note that $|s^{1/2} - t^{1/2}|  \leq |s-t|^{1/2}$ for $s,t > 0$, then it follows from the ellipticity  \eqref{cond_A} of  $\mathbb{A}_1 $ and $\mathbb{A}_2$ that for any $\xi \in \mathbb{R}^{N \times n}$ we get
\begin{equation}\label{A_1-A_2}
\big|\,|\xi|_{\mathbb{A}_1} - |\xi|_{\mathbb{A}_2}\big|
\leq \big|\big\langle \mathbb{A}_1\,\xi, \xi \big\rangle - \big\langle \mathbb{A}_2\,\xi,\xi \big\rangle\big|^{\frac{1}{2}}
\leq |\mathbb{A}_1 - \mathbb{A}_2|^\frac{1}{2}\,|\xi|.
\end{equation}
This together with the Newton-Leibniz formula yields that
\begin{eqnarray*}
&&\fint_{B_{r/8}} \big|\bar{\varPsi}_r\big(|Du|_{\bar{\mathbb{A}}_r^u}\big) -  \bar{\varPsi}_r\big(|Du|_{\mathbb{A}^u}\big)\big|\, dx \notag\\
&=& \fint_{B_{r/8}} \int^1_0 \big|\,|Du|_{\bar{\mathbb{A}}_r^u} - |Du|_{\mathbb{A}^u}\big|\,\,\bar{\varPsi}_r^{\prime}\big(t\,|Du|_{\bar{\mathbb{A}}_r^u} + (1-t)\,|Du|_{\mathbb{A}^u}\big)\, dt \, dx\notag\\
&\leq& \fint_{B_{r/8}} \big|\mathbb{A}(x_0, \langle u \rangle_{B_r}) - \mathbb{A}(x, u)\big|^\frac{1}{2}\,|Du| \,\bar{\varPsi}_r^{\prime}\big(2\Lambda|Du|\big)\,dx  \notag\\
&\leq& c \fint_{B_{r/8}}\omega_{\mathbb{A}}^{\frac{1}{2}}\big(r + |u - \langle u \rangle_{B_r}|\big)\,\bar{\varPsi}_r\big(|Du|\big)\, dx,
\end{eqnarray*}
where we used $t \bar{\varPsi}_r'(t) \approx \bar{\varPsi}_r(t)$ and the modulus of continuity of $\mathbb{A}(\cdot,\cdot)$
in the last step. By the boundedness and the concavity of $t \mapsto \omega_{\mathbb{A}}(t)$, together with the higher integrability of $Du$ in \eqref{ineq_rev}, it implies that
\begin{eqnarray}\label{est_al_min_ft}
&&\fint_{B_{r/8}} \big|\bar{\varPsi}_r\big(|Du|_{\bar{\mathbb{A}}_r^u}\big) -  \bar{\varPsi}_r\big(|Du|_{\mathbb{A}^u}\big)\big|\, dx \notag\\
&\leq& c\, \bigg(\fint_{B_{r/8}} \Big(\omega_{\mathbb{A}}\big(r + |u - \langle u \rangle_{B_r}|\big)\Big)^\frac{1+\varepsilon}{2\varepsilon} dx\bigg)^{\frac{\varepsilon}{1+\varepsilon}} \bigg(\fint_{B_{r/8}} \Big(\bar{\varPsi}_r\big(|Du|\big)\Big)^{1+ \varepsilon} dx\bigg)^\frac{1}{1+\varepsilon}\notag\\
&\leq& c\, \omega_{\mathbb{A}}^{\frac{\varepsilon}{1+\varepsilon}}\big(r + \mathcal{E}(r)\big)\, E(u, r) \notag\\[6pt]
&\leq& c\,{\delta}_1(r)\,E(u, r),
\end{eqnarray}
where $c = c\big(p, q, \Lambda, L, K\big) \geq 1$.
\vspace{3pt}

\noindent\emph{Step 2.}\ \  We focus on the second term of the right-hand side of \eqref{est_al_min}. By the local minimality of $u$, for any $\eta \in W^{1,\infty}_0(B_{r/8}, \mathbb{R}^N)$ we have
\begin{equation}\label{min_u}
\fint_{B_{r/8}} \varPsi\big(x, |Du|_{\mathbb{A}^u}\big)\, dx \leq \fint_{B_{r/8}} \varPsi\big(x, |Du + D\eta|_{\mathbb{A}^{u + \eta}}\big)\, dx.
\end{equation}
Let us write $w = u + \eta$ and divide the right-hand side into two parts
\begin{equation*}
\mathcal{D}_1 = {\Big\{x \in B_{r/8} : |Dw| \leq \varphi_1^{-1}(r^{-n})\Big\}},
\qquad
\mathcal{D}_2 = {\Big\{x \in B_{r/8} : |Dw| > \varphi_1^{-1}(r^{-n})\Big\}}.
\end{equation*}
Then we  decompose \eqref{min_u} as
\begin{eqnarray}\label{est_al_min_st}
\fint_{B_{r/8}} \varPsi\big(x, |Du|_{\mathbb{A}^u}\big)\, dx
&\leq & \frac{1}{|B_{r/8}|}\int_{\mathcal{D}_1} \varPsi\big(x, |Dw|_{\mathbb{A}^w}\big)\, dx + \frac{1}{|B_{r/8}|}\int_{\mathcal{D}_2} \varPsi\big(x, |Dw|_{\mathbb{A}^w}\big)\, dx \notag\\[6pt]
&:=& I_{1} + I_{2}.
\end{eqnarray}
By direct calculation and enlarging the integration domain, we estimate the term $I_{1}$ as follows:
\begin{eqnarray}\label{est_I1_1}
I_{1} &\leq&   \frac{1}{|B_{r/8}|}\int_{\mathcal{D}_1} \big|{\varPsi}\big(x, |Dw|_{\mathbb{A}^w}\big)-\bar{\varPsi}_r\big(|Dw|_{\mathbb{A}^w}\big)\big|\, dx +  \fint_{B_{r/8}} \big|\bar{\varPsi}_r\big(|Dw|_{\mathbb{A}^w}\big)-\bar{\varPsi}_r\big(|Dw|_{\mathbb{A}^u}\big)\big|\, dx \notag\\
&&+\,  \fint_{B_{r/8}} \big|\bar{\varPsi}_r\big(|Dw|_{\mathbb{A}^u}\big)-\bar{\varPsi}_r\big(|Dw|_{\bar{\mathbb{A}}_r^u}\big)\big|\, dx +  \fint_{B_{r/8}} \bar{\varPsi}_r\big(|Dw|_{\bar{\mathbb{A}}_r^u}\big)\, dx \notag\\
&:=& J_1+J_2+J_3+ \fint_{B_{r/8}} \bar{\varPsi}_r\big(|Dw|_{\bar{\mathbb{A}}_r^u}\big)\, dx.
\end{eqnarray}

To estimate $J_1$, we notice that $\varphi_2(|Dw|_{\mathbb{A}^w}) \leq \Lambda^\frac{q}{2} \varphi_2(|Dw|)$ due to the boundedness of $\mathbb{A}(\cdot, \cdot)$ and  \eqref{est_>1} for $\varphi_2(\cdot)$. This yields that
\begin{eqnarray}\label{est_I1_ft_1}
J_1
\leq \frac{1}{|B_{r/8}|}\int_{\mathcal{D}_1} |a(x) - \bar{a}_r| \,{\varphi_2}\big(|Dw|_{\mathbb{A}^w}\big)\, dx
\leq  \frac{c}{|B_{r/8}|}\int_{\mathcal{D}_1} \omega_a(r) \,{\varphi_2}\big(|Dw|\big) dx.
\end{eqnarray}
Since $|Dw| \leq \varphi_1^{-1}(r^{-n})$ in $\mathcal{D}_1$,  we get $\frac{{\varphi_2}(|Dw|)}{{\varphi_1}(|Dw|)} \leq \frac{{\varphi_2}(\varphi_1^{-1}(r^{-n}))}{{\varphi_1}(\varphi_1^{-1}(r^{-n}))} = \frac{{\varphi_2}\circ \varphi_1^{-1}(r^{-n})}{r^{-n}}$ due to the non-decreasing property of $t \mapsto \frac{\varphi_2(t)}{\varphi_1(t)}$. Obviously, $\varphi_1(|Dw|) \leq \varphi_1(|Dw|) + \bar{a}_r\, \varphi_2(|Dw|) = \bar{\varPsi}_r\big(|Dw|\big)$, then
\begin{eqnarray}\label{est_I1_ft_2}
\int_{\mathcal{D}_1} \omega_a(r) \,{\varphi_2}\big(|Dw|\big) dx &=& \int_{\mathcal{D}_1} \omega_a(r) \,\frac{{\varphi_2}\big(|Dw|\big)}{{\varphi_1}\big(|Dw|\big)} \,{\varphi_1}\big(|Dw|\big)\,  dx \notag\\
&\le& \int_{\mathcal{D}_1} \omega_a(r) \,\frac{{\varphi_2} \circ \varphi_1^{-1}\big(r^{-n}\big)}{r^{-n}} \,\bar{\varPsi}_r\big(|Dw|\big)\,  dx \notag\\
&\leq& \mathcal{G}(r) \int_{B_{r/8}} \bar{\varPsi}_r\big(|Dw|\big)\,  dx,
\end{eqnarray}
By $w = u + \eta$ we get
$\bar{\varPsi}_r\big(|Dw|\big)\leq 2^{q-1}\Big (\bar{\varPsi}_r\big(|Du|\big)  +  \bar{\varPsi}_r\big(|D\eta|\big)\Big)$, this combines with \eqref{est_I1_ft_1} \eqref{est_I1_ft_2} to yield
\begin{equation}\label{est_J_1}
J_1 \leq c\,\mathcal{G}(r) \fint_{B_{r/8}} \Big (\bar{\varPsi}_r\big(|Du|\big)  +  \bar{\varPsi}_r\big(|D\eta|\big)\Big)\, dx.
\end{equation}
Moreover, it follows from \eqref{est_>1} and the non-decreasing property of $t \mapsto \bar{\varPsi}_r(t)$ that
\begin{eqnarray}\label{est_Deta}
\bar{\varPsi}_r\big(|D\eta|\big)
&\leq& \bar{\varPsi}_r\,\Big(\|D\eta\|_{L^\infty(B_{r/8})}\Big)
= \bar{\varPsi}_r\,\bigg(\bar{\varPsi}^{-1}_r\big(E(u, r)\big) \,\frac{\|D\eta\|_{L^\infty(B_{r/8})}}{\bar{\varPsi}^{-1}_r\big(E(u, r)\big)} \bigg)\notag\\
&\leq&  \bar{\varPsi}_r\,\Big(\bar{\varPsi}^{-1}_r\big(E(u, r)\big)\Big) \,\bigg(\frac{\|D\eta\|_{L^\infty(B_{r/8})}}{\bar{\varPsi}^{-1}_r\big(E(u, r)\big)}+1 \bigg)^{q}\notag\\
&= & E(u, r) \,\bigg(\frac{\|D\eta\|_{L^\infty(B_{r/8})}}{\bar{\varPsi}^{-1}_r\big(E(u, r)\big)}+1\bigg)^{q}.
\end{eqnarray}
Let us put $\fint_{B_{r/8}} \bar{\varPsi}_r\big(|Du|\big)\, dx  \leq E(u,r)$ and \eqref{est_Deta} into \eqref{est_J_1} to get
\begin{equation}\label{est_I1_ft}
J_1 \leq c\, {\delta}_1(r)\, E(u, r)\,\bigg(\frac{\|D\eta\|_{L^\infty(B_{r/8})}}{\bar{\varPsi}^{-1}_r\big(E(u, r)\big)}+1\bigg)^{q}.
\end{equation}

To estimate $J_2$, by the Newton-Leibniz formula, \eqref{A_1-A_2} and $t \bar{\varPsi}_r'(t) \approx \bar{\varPsi}_r(t)$ we deduce that
\begin{eqnarray}\label{est_I1_st_1}
J_2&=&\fint_{B_{r/8}}\big|\bar{\varPsi}_r\big(|Dw|_{\mathbb{A}^w}\big)-\bar{\varPsi}_r\big(|Dw|_{\mathbb{A}^u}\big)\big|\, dx \notag\\
&=& \fint_{B_{r/8}} \int^1_0 \big|\,|Dw|_{\mathbb{A}^w} -
|Dw|_{\mathbb{A}^u}\big|\,\bar{\varPsi}_r^{\prime}\big(t\,|Dw|_{\mathbb{A}^w} + (1-t)\,|Dw|_{\mathbb{A}^u}\big)\, dt \, dx \notag\\
&\leq& \fint_{B_{r/8}} \big|\mathbb{A}(x, w) - \mathbb{A}(x, u)\big|^\frac{1}{2}\,|Dw| \,\bar{\varPsi}_r^{\prime}\big(2\Lambda|Dw|\big)\,dx \notag\\
&\leq& c \fint_{B_{r/8}} \omega_{\mathbb{A}}^{\frac{1}{2}}\big(|\eta|\big)\, \bar{\varPsi}_r\big(|Dw|\big)\, dx  \notag\\
&\leq& c\, \omega_{\mathbb{A}}^{\frac{1}{2}}\big(\|\eta\|_{L^\infty(B_{r/8})}\big)\, \bigg(\fint_{B_{r/8}} \bar{\varPsi}_r\big(|Du|\big)\, dx + \fint_{B_{r/8}} \bar{\varPsi}_r\big(|D\eta|\big)\, dx\bigg)\,.
\end{eqnarray}
To estimate $\omega_{\mathbb{A}}^{\frac{1}{2}}\big(\|\eta\|_{L^\infty(B_{r/8})}\big)$, we see $\|\eta\|_{L^\infty(B_{r/8})} \leq \frac{r}{2}\|D\eta\|_{L^\infty(B_{r/8})} \allowbreak$ since $\eta$ is vanishing on $\partial B_{r/8}$. We first note the non-decreasing property of $t \mapsto \bar{\varPsi}^{-1}_r(t)$, and that $\Big(\fint_{B_{r/2}} \bar{\varPsi}_r^{1+\varepsilon}\big(x,|Du|\big) \, dx\Big)^{\frac{1}{1+\varepsilon}} \leq c_* \bar{\varPsi}_r\,\Big(\fint_{B_{r}} |Du| \, dx\Big)$ due to Lemma \ref{lem_high_int}, to get that
\begin{eqnarray*}
\bar{\varPsi}^{-1}_r\big(E(u, r)\big)
&=& \bar{\varPsi}^{-1}_r\Bigg(\bigg(\fint_{B_{r/8}} \Big(\bar{\varPsi}_r\big(|Du|\big)\Big)^{1+\varepsilon} dx\bigg)^\frac{1}{1+\varepsilon}\Bigg)\notag\\
&\leq& \bar{\varPsi}^{-1}_r\,\Bigg(c_*\,\bar{\varPsi}_r\bigg(\fint_{B_{r/4}} |Du|\, dx\bigg)\Bigg)\,.
\end{eqnarray*}
Considering that $\bar{\varPsi}^{-1}_r\,\Big(c_*\,\bar{\varPsi}_r\big(\fint_{B_{r/4}} |Du|\, dx\big)\Big) \leq c_*^{\frac{1}{p}} \fint_{B_{r/4}} |Du| \,dx$  in accordance with \eqref{est_>1}, it follows from Lemma \ref{lem_cacc} that
\begin{eqnarray}\label{Epsi}
\bar{\varPsi}^{-1}_r\big(E(u, r)\big)
&\leq& c'_1 \fint_{B_{r/4}} |Du| \,dx
\leq \frac{c'_1}{r} \fint_{B_r} {|u - \langle u \rangle_r|}\, dx
\equiv \frac{c'_1}{r} \,\mathcal{E}(r)\,.
\end{eqnarray}
Next, we notice $\omega_{\mathbb{A}}(t) \ge \frac{1}{A} \omega_{\mathbb{A}}(At)$ for $A > 1, t > 0$ due to its concavity and $\omega_{\mathbb{A}}(0) = 0$, to conclude that
\begin{eqnarray}\label{omega_eta}
\omega_{\mathbb{A}} \big(\|\eta\|_{L^\infty(B_{r/8})}\big)
&=&  \omega_{\mathbb{A}} \bigg( \frac{\|\eta\|_{L^\infty(B_{r/8})}}{\bar{\varPsi}^{-1}_r\big(E(u, r) \big) } \bar{\varPsi}^{-1}_r\big(E(u, r)\big)  \bigg)\notag\\
&\overset{\eqref{Epsi}}{\leq}&  \omega_{\mathbb{A}}\, \bigg( \frac{c_1\,\|D\eta\|_{L^\infty(B_{r/8})}}{\bar{\varPsi}^{-1}_r\big(E(u, r) \big) } \,\mathcal{E}(r)  \bigg)\notag\\
&\leq&  c'_1\, \bigg(\frac{\|D\eta\|_{L^\infty(B_{r/8})}}{\bar{\varPsi}^{-1}_r\big(E(u, r)\big)}+1\bigg)\,\omega_{\mathbb{A}} \big( \mathcal{E}(r) \big)\,,
\end{eqnarray}
which yields
$\omega_{\mathbb{A}}^{\frac{1}{2}}\big(\|\eta\|_{L^\infty(B_{r/8})}\big) \leq (c'_1)^\frac{1}{2} {\delta}_1(r) \Big(\frac{\|D\eta\|_{L^\infty(B_{r/8})}}{\bar{\varPsi}^{-1}_r\big(E(u, r)\big)}+1\Big)^{\frac{1}{2}}$.
Now we put this and \eqref{est_Deta} into the inequality \eqref{est_I1_st_1} and using $\fint_{B_{r/8}} \bar{\varPsi}_r\big(|Du|\big)\, dx  \leq E(u,r)$ to conclude
\begin{eqnarray}\label{est_I1_st}
J_2 \leq  c\,{\delta}_1(r) \,E(u, r)\,\bigg(\frac{\|D\eta\|_{L^\infty(B_{r/8})}}{\bar{\varPsi}^{-1}_r\big(E(u, r)\big)}+1\bigg)^{q+1}.
\end{eqnarray}

To estimate $J_3$, by a similar way to the estimate of $J_2$ we also get
\begin{eqnarray}\label{est_I1_tt_1}
J_3	\leq  c\fint_{B_{r/8}} \omega_{\mathbb{A}}^{\frac{1}{2}}\big(r + |u - \langle u \rangle_{B_r}|\big)\,\Big(\bar{\varPsi}_r\big(|Du|\big) + \bar{\varPsi}_r\big(|D\eta|\big)\Big)\, dx.
\end{eqnarray}
To this end, let us make use of the same argument as the estimate  \eqref{est_al_min_ft} to deduce
\begin{equation}\label{est_du}
\fint_{B_{r/8}} \omega_{\mathbb{A}}^{\frac{1}{2}}\big(r + |u - \langle u \rangle_{B_r}|\big)\,\bar{\varPsi}_r\big(|Du|\big) \, dx \leq c \,{\delta}_1(r)\,E(u, r).
\end{equation}
On the other hand, we see that $\fint_{B_{r/8}}  \omega_{\mathbb{A}}^{\frac{\varepsilon}{1+\varepsilon}}\big(r + |u - \langle u \rangle_{B_r}|\big)\,dx \leq \omega_{\mathbb{A}}^{\frac{\varepsilon}{1+\varepsilon}}\big(r + \mathcal{E}(r)\big) \leq {\delta}_1(r)$ due to the concavity of $t \mapsto \omega_{\mathbb{A}}^{\frac{\varepsilon}{1+\varepsilon}}(t)$, and the estimate of $\bar{\varPsi}_r\big(|D\eta|\big)$ in \eqref{est_Deta}, yield that
\begin{equation}\label{est_omega_eta}
\fint_{B_{r/8}} \omega_{\mathbb{A}}^{\frac{\varepsilon}{1+\varepsilon}}\big(r + |u - \langle u \rangle_{B_r}|\big)\,\bar{\varPsi}_r\big(|D\eta|\big) \, dx \leq {\delta}_1(r)\,E(u, r)\,\bigg(\frac{\|D\eta\|_{L^\infty(B_{r/8})}}{\bar{\varPsi}^{-1}_r\big(E(u, r)\big)}+1\bigg)^{q+1}.
\end{equation}
Putting \eqref{est_du} and \eqref{est_omega_eta} into \eqref{est_I1_tt_1} implies
\begin{equation}\label{est_I1_tt}
J_3
\leq  c\,{\delta}_1(r)\,E(u, r)\,\bigg(\frac{\|D\eta\|_{L^\infty(B_{r/8})}}{\bar{\varPsi}^{-1}_r\big(E(u, r)\big)}+1\bigg)^{q+1}.
\end{equation}
In summary, we insert the estimates  \eqref{est_I1_ft} \eqref{est_I1_st}  \eqref{est_I1_tt} for $J_1,J_2,J_3$ into \eqref{est_I1_1}, to attain
\begin{eqnarray}\label{est_I1}
I_{1} &\leq& J_1 + J_2 + J_3 + \fint_{B_{r/8}} \bar{\varPsi}_r\big(|Dw|_{\bar{\mathbb{A}}_r^u}\big)\, dx \notag\\
&\leq& c\,{\delta}_1(r) \,E(u, r)\,\bigg(\frac{\|D\eta\|_{L^\infty(B_{r/8})}}{\bar{\varPsi}^{-1}_r\big(E(u, r)\big)}+1\bigg)^{q+1} +  \fint_{B_{r/8}} \bar{\varPsi}_r\big(|Du + D\eta|_{\bar{\mathbb{A}}_r^u}\big)\, dx,
\end{eqnarray}
where $c = c\,\big(\texttt{data}_0\big) \geq 1$.

We devote the following to the estimate of $I_{2}$. Since $|Dw| > \varphi_1^{-1}\big(r^{-n}\big)$ for $x\in \mathcal{D}_2$, it holds that $\frac{1}{\big(\varphi_1(|Dw|)\big)^{\varepsilon}} \leq r^{n\varepsilon}$. By the ellipticity and boundedness of $\mathbb{A}(\cdot, \cdot)$ and $\bar{\varPsi}_r(|\cdot|) \leq \varPsi(x, |\cdot|)$, we show that
\begin{eqnarray}\label{est_I2_1}
I_{2}
&\leq&  \frac{c}{|B_{r/8}|}\int_{\mathcal{D}_2} \varPsi\big(x, |Dw|\big)\, dx  \notag\\
&=&  \frac{c}{|B_{r/8}|}\int_{\mathcal{D}_2} \frac{\big({\varPsi}(x, |Dw|)\big)^{1+\varepsilon}}{\big({\varPsi}(x, |Dw|)\big)^{\varepsilon}}\, dx \notag\\
&\leq & \frac{c\,r^{n\varepsilon}}{|B_{r/8}|}\,\int_{\mathcal{D}_2} \Big({\varPsi}\big(x, |Dw|\big)\Big)^{1+\varepsilon}\, dx \notag\\
&\leq&  c\,r^{n\varepsilon}\fint_{B_{r/8}}\bigg( \Big({\varPsi}\big(x, |Du|\big)\Big)^{1+\varepsilon} + \Big({\varPsi}\big(x, |D\eta|\big)\Big)^{1+\varepsilon} \bigg)\, dx.
\end{eqnarray}
For the first term of the right-hand side to \eqref{est_I2_1}, we employ the reverse H\"older to get  that
\begin{eqnarray}\label{est_I2_ft_1}
\fint_{B_{r/8}} \Big({\varPsi}\big(x, |Du|\big)\Big)^{1+\varepsilon} dx
&=&  \bigg(\fint_{B_{r/8}} \Big(\varPsi\big(x, |Du|\big)\Big)^{1+\varepsilon}dx\bigg)^\frac{\varepsilon}{1+\varepsilon}\bigg(\fint_{B_{r/8}}\Big(\varPsi\big(x, |Du|\big)\Big)^{1+\varepsilon} dx\bigg)^\frac{1}{1+\varepsilon}\notag\\
&\leq&  c\, \bigg(\fint_{B_{r/8}} \Big({\varPsi}\big(x, |Du|\big)\Big)^{1+\varepsilon} dx\bigg)^\frac{\varepsilon}{1+\varepsilon}E(u, r).
\end{eqnarray}
Note that  $\|{\varPsi}\big(x,|Du|\big)\|_{L^{1+\varepsilon}(\Omega')} \leq {M}$, it yields  that
\begin{equation}\label{est_I2_ft_2}
r^{n\varepsilon}\, \bigg(\fint_{B_{r/8}} \Big({\varPsi}\big(x, |Du|\big)\Big)^{1+\varepsilon} dx\bigg)^\frac{\varepsilon}{1+\varepsilon}
\leq  r^\frac{n\varepsilon^2}{1+\varepsilon}\,\|{\varPsi}\big(x,|Du|\big)\|^\varepsilon_{L^{1+\varepsilon}(B_{r/2})}
\leq  {\delta}_1(r)\,{M}^\varepsilon,
\end{equation}
where we used $r^\frac{n\varepsilon^2}{1+\varepsilon} \leq {\delta}_1(r)$ in the last inequality.
Putting \eqref{est_I2_ft_2} into \eqref{est_I2_ft_1}, we arrive at
\begin{equation}\label{est_I2_ft}
r^{n\varepsilon}\fint_{B_{r/8}} \Big({\varPsi}\big(x, |Du|\big)\Big)^{1+\varepsilon} dx \leq c\, {\delta}_1(r)\, E(u, r).
\end{equation}
Let us now turn to the second term of the right-hand side to \eqref{est_I2_1}. Thanks to \eqref{est_>1}, we have that
\begin{eqnarray}\label{est_I2_st_1}
&&\fint_{B_{r/8}} \Big({\varPsi}\big(x, |D\eta|\big)\Big)^{1+\varepsilon} dx \notag\\
&\leq&  c \fint_{B_{r/8}}  \Bigg({\varPsi}\,\bigg(x, \frac{|D\eta|}{\bar{\varPsi}^{-1}_r\big(E(u, r)\big)} \,\bar{\varPsi}^{-1}_r\big(E(u, r)\big)\bigg)\Bigg)^{1+\varepsilon} \,dx\notag\\
&\leq&  c \fint_{B_{r/8}} \bigg({\varPsi}\Big(x, \bar{\varPsi}^{-1}_r\big(E(u, r)\big)\Big)\bigg)^{1+\varepsilon}\,dx\,\bigg(\frac{\|D\eta\|_{L^\infty(B_{r/8})}}{\bar{\varPsi}^{-1}_r\big(E(u, r)\big)}+1\bigg)^{q+1}.
\end{eqnarray}
We notice $\bar{\varPsi}^{-1}_r(E(u, r)) \lesssim \fint_{B_r} |Du| \, dx$ by applying $\bar{\varPsi}^{-1}_r(\cdot)$ to \eqref{ineq_rev}, and the convexity of $\varphi_1(\cdot)$ yields that $ \fint_{B_r} |Du| \, dx \lesssim \varphi^{-1}_1\Big(\fint_{B_r} \varphi_1(|Du|) \, dx\Big) \lesssim \varphi^{-1}_1\Big({r}^{-n} \|\varphi_1(|Du|)\|_{L^1(B_r)}\Big) \lesssim \varphi^{-1}_1\big( r^{-n}\big)$; then it implies $\bar{\varPsi}^{-1}_r(E(u, r)) \leq \varphi^{-1}_1\big(T r^{-n}\big)$ for $T = T(p,q,\|\varPsi(x,Du)\|_{L^1(\Omega)}) \geq 1$. This yields
$\frac{\varphi_2\big(\bar{\varPsi}^{-1}_r(E(u, r))\big)}{\varphi_1\big(\bar{\varPsi}^{-1}_r(E(u, r))\big)}
\leq T^{\frac{q}{p}-1} \frac{\varphi_2 \circ \varphi_1^{-1} (r^{-n})}{r^{-n}}$
by the same reasoning as in \eqref{omega}, and then we have that
\begin{eqnarray}\label{Psi-Psi-1}
{\varPsi}\Big(x, \bar{\varPsi}^{-1}_r\big(E(u, r)\big)\Big)
&\leq& \bar{\varPsi}_r\Big(\bar{\varPsi}^{-1}_r\big(E(u, r)\big)\Big) + |a(x)- \bar{a}_r| \,\varphi_2\Big(\bar{\varPsi}^{-1}_r\big(E(u, r)\big)\Big)    \notag\\
&\leq& E(u, r) + L\,\omega_a(r)\, \frac{\varphi_2\big(\bar{\varPsi}^{-1}_r(E(u, r))\big)}{\varphi_1\big(\bar{\varPsi}^{-1}_r(E(u, r))\big)}\, \varphi_1\big(\bar{\varPsi}^{-1}_r(E(u, r))\big) \notag\\
&\leq& E(u, r) + c \,\omega_a\big( r\big)\, \frac{\varphi_2 \circ \varphi_1^{-1} \big(r^{-n}\big)}{r^{-n}}\, \bar{\varPsi}_r\big(\bar{\varPsi}^{-1}_r(E(u, r))\big)    \notag\\[6pt]
&=& c \,\Big(1 + \mathcal{G}(r)\Big)\, E(u, r).
\end{eqnarray}
We note $\big(E(u, r)\big)^\varepsilon \leq c r^{\frac{-n\varepsilon}{1+\varepsilon}} \|\varPsi(x, |Du|)\|^\varepsilon_{L^{1+\varepsilon}(B_{r/8})} $, then
inserting \eqref{Psi-Psi-1}	into \eqref{est_I2_st_1}, and considering $r^{n\varepsilon - \frac{n\varepsilon}{1+\varepsilon}} \leq {\delta}_1(r)$ and $\mathcal{G}(r) \leq K$, lead to the following facts:
\begin{eqnarray}\label{est_I2_st}
&&r^{n\varepsilon}\fint_{B_{r/8}} \Big({\varPsi}\big(x, |D\eta|\big)\Big)^{1+\varepsilon} dx\notag\\
&\leq&  c\,r^{n\varepsilon - \frac{n\varepsilon}{1+\varepsilon}} \Big(1 + \mathcal{G}(r)\Big)^{1+\varepsilon}\, E(u, r)\,\bigg(\frac{\|D\eta\|_{L^\infty(B_{r/8})}}{\bar{\varPsi}^{-1}_r\big(E(u, r)\big)}+1\bigg)^{q+1}\notag\\
&\leq&  c\, {\delta}_1(r)\,E(u, r)\,\bigg(\frac{\|D\eta\|_{L^\infty(B_{r/8})}}{\bar{\varPsi}^{-1}_r\big(E(u, r)\big)}+1\bigg)^{q+1}.
\end{eqnarray}
All in all, we put \eqref{est_I2_ft} and \eqref{est_I2_st} into \eqref{est_I2_1} to conclude
\begin{equation}\label{est_I2}
I_{2} \leq c\,{\delta}_1(r)\,E(u, r)\,\bigg(\frac{\|D\eta\|_{L^\infty(B_{r/8})}}{\bar{\varPsi}^{-1}_r\big(E(u, r)\big)}+1\bigg)^{q+1},
\end{equation}
where $c = c\,\big(\texttt{data}_0\big) \geq 1$.

Let us now put the estimates \eqref{est_I1}\,\eqref{est_I2} for $I_1,I_2$  together into \eqref{est_al_min_st}, to deduce that
\begin{equation*}
\fint_{B_{r/8}} \varPsi\big(x, |Du|_{\mathbb{A}^u}\big)\, dx \leq \fint_{B_{r/8}} \bar{\varPsi}_r\big(|Du + D\eta|_{\bar{\mathbb{A}}_r^u}\big)\, dx + c\,{\delta}_1(r) \,E(u, r)\,\bigg(\frac{\|D\eta\|_{L^\infty(B_{r/8})}}{\bar{\varPsi}^{-1}_r\big(E(u, r)\big)}+1\bigg)^{q+1}.
\end{equation*}
This combines \eqref{est_al_min} \eqref{est_al_min_ft} to conclude that $u$ is an almost minimizer of the local frozen functional \eqref{u_al_min}.
\end{proof}

Furthermore, let us reformulate the comparison estimate between $Du$ and $Dv$ in $L^1$-sense.

\begin{corollary}\label{cor_comp_1}
Under the same assumptions of Lemma \ref{lem_comp}, then for any $B_r \subset \Omega' \Subset \Omega$ with $ r \in (0, R_0)$ it holds
\begin{equation}\label{est_comp_1}
\int_{B_{r/8}} \big|Du - Dv\big|\, dx \leq c_2\,\Big(\omega_{\mathbb{A}}^{\,\mu}\big(r + \mathcal{E}(r)\big) + \mathcal{G}(r) + r^\frac{n\varepsilon^2}{1+\varepsilon}\Big)^{\kappa_2}\,\int_{B_{r/4}} |Du|\, dx ,
\end{equation}
where $c_2 = c_2\big(\texttt{data}_0\big) \geq 1$, $\mathcal{G}(r)$ as in \eqref{Gr} and $\mu = \min\{\frac{\varepsilon}{1+\varepsilon},\frac{1}{2}\}, {\kappa_2} = {\kappa_2}(p, q,\varepsilon) \in (0,1)$ for $\varepsilon <\varepsilon_0$ with $\varepsilon_0$ in Lemma \ref{lem_high_int}.
\end{corollary}

\begin{proof}
Let ${\delta}_2(r) := \Big(\omega_{\mathbb{A}}^{\,\mu}\big(r + \mathcal{E}(r)\big) + \mathcal{G}(r) + r^\frac{n\varepsilon^2}{1+\varepsilon}\Big)^\frac{\kappa_1}{2}$. It follows  that
\begin{eqnarray*}
|Du - Dv| &=& \Big({\delta}_2(r)\,\big(|Du| + |Dv|\big)\Big)^\frac{1}{2}\cdot\Big({\delta}_2(r)\,\big(|Du| + |Dv|\big)\Big)^{-\frac{1}{2}}\,|Du - Dv|\notag\\
&\leq& \frac{1}{2}\bigg({\delta}_2(r)\,\big(|Du| + |Dv|\big) + \Big({\delta}_2(r)\,\big(|Du| + |Dv|\big)\Big)^{-1}\,|Du - Dv|^2\bigg)\,.
\end{eqnarray*}
Since $t\bar{\varPsi}'_r(t) \approx \bar{\varPsi}_r(t)$ and \eqref{est_V}, we have that
\begin{eqnarray*}
\bar{\varPsi}_r\big(|Du - Dv|\big)
&\leq & c\,\bar{\varPsi}'_r\big(|Du| + |Dv|\big)\,|Du - Dv| \notag\\[6pt]
&\leq&  c\bar{\varPsi}'_r\big(|Du| + |Dv|\big)\,\bigg({\delta}_2(r)\,\big(|Du| + |Dv|\big)  +  \Big({\delta}_2(r)\big({|Du| + |Dv|}\big)\Big)^{-1}|Du - Dv|^2\bigg)\notag\\
&=&  c\,\bigg({\delta}_2(r)\bar{\varPsi}'_r\big(|Du| + |Dv|\big)\big(|Du| + |Dv|\big)  +  \big({\delta}_2(r)\big)^{-1}\frac{\bar{\varPsi}'_r\big(|Du| + |Dv|\big)}{|Du| + |Dv|}|Du - Dv|^2\bigg)\notag\\
&\leq&  c\,\bigg({\delta}_2(r)\bar{\varPsi}_r\big(|Du| + |Dv|\big)  +  \big({\delta}_2(r)\big)^{-1}\big|V_{\bar{\varPsi}_r}(Du) - V_{\bar{\varPsi}_r}(Dv)\big|^2\bigg)
\end{eqnarray*}
In view of the convexity of $t \mapsto \bar{\varPsi}_r(t)$ and \eqref{energy_v}, we deduce that
\begin{eqnarray*}
&&\bar{\varPsi}_r\,\bigg(\fint_{B_{r/8}} |Du -Dv|\, dx\bigg)
\leq \fint_{B_{r/8}} \bar{\varPsi}_r\big(|Du - Dv|\big) \, dx \notag\\
&\leq&  c\, \bigg({\delta}_2(r)\fint_{B_{r/8}} \bar{\varPsi}_r\big(|Du|+ |Dv|\big) \, dx + \big({\delta}_2(r)\big)^{-1}\fint_{B_{r/8}} \big|V_{\bar{\varPsi}_r}(Du) - V_{\bar{\varPsi}_r}(Dv)\big|^2 \, dx\bigg)\notag\\
&\leq&  c\, \bigg({\delta}_2(r)\fint_{B_{r/8}} \bar{\varPsi}_r\big(|Du|\big) \, dx + {\delta}_2(r)\fint_{B_{r/8}} \bar{\varPsi}_r\big(|Dv|\big) \, dx + \big({\delta}_2(r)\big)^{-1}\fint_{B_{r/8}} \big|V_{\bar{\varPsi}_r}(Du) - V_{\bar{\varPsi}_r}(Dv)\big|^2 \, dx\bigg)\notag\\
&\leq&  c\, \bigg({\delta}_2(r)\fint_{B_{r/8}} \bar{\varPsi}_r\big(|Du|\big) \, dx + \big({\delta}_2(r)\big)^{-1}\fint_{B_{r/8}} \big|V_{\bar{\varPsi}_r}(Du) - V_{\bar{\varPsi}_r}(Dv)\big|^2 \, dx\bigg)\,.
\end{eqnarray*}
Thanks to  \eqref{ineq_rev_jense} and
Lemma \ref{lem_comp}, we further obtain that
\begin{eqnarray*}
&&\bar{\varPsi}_r\,\bigg(\fint_{B_{r/8}} |Du -Dv|\, dx\bigg)\notag\\
&\leq&  c\, \Bigg(\,{\delta}_2(r)\,\bigg(\fint_{B_{r/8}}\Big(\bar{\varPsi}_r\big( |Du|\big)\Big) \, dx\bigg) + \big({\delta}_2(r)\big)^{-1} {c}\,\big({\delta}_2(r)\big)^2\,\bar{\varPsi}_r\,\bigg(\fint_{B_{r/4}} |Du| \, dx\bigg) \Bigg) \notag\\
&\leq& c\, {\delta}_2(r)\,\bar{\varPsi}_r\,\bigg(\fint_{B_{r/4}} |Du| \, dx \bigg)\,,
\end{eqnarray*}
We now use \eqref{est_<1} \eqref{est_>1} to get that
$c\, {\delta}_2(r)\,\bar{\varPsi}_r\,\Big(\fint_{B_{r/4}} |Du| \, dx \Big)
\leq  \bar{\varPsi}_r\,\Big(\big(c\, {\delta}_2(r)\big)^\frac{1}{q}\fint_{B_{r/4}} |Du| \, dx \Big)$ for $0<c\, {\delta}_2(r)\le 1$, and $c\, {\delta}_2(r)\,\bar{\varPsi}_r\,\Big(\fint_{B_{r/4}} |Du| \, dx \Big)
\leq  \bar{\varPsi}_r\,\Big(\big(c\, {\delta}_2(r)\big)^\frac{1}{p}\fint_{B_{r/4}} |Du| \, dx \Big)$ for $c\, {\delta}_2(r)>1$. Then it implies
\begin{equation*}
\bar{\varPsi}_r\,\bigg(\fint_{B_{r/8}} |Du -Dv|\, dx\bigg) \leq  \bar{\varPsi}_r\,\bigg(\Big(c\, {\delta}_2(r)\Big)^\frac{1}{s}\fint_{B_{r/4}} |Du| \, dx \bigg)
\end{equation*}
with $s = s(p, q) > 1$.
Acting $\bar{\varPsi}^{-1}_r(\cdot)$ on both sides, we conclude
\begin{equation*}
\fint_{B_{r/8}} |Du -Dv|\, dx \leq \Big(c\, {\delta}_2(r)\Big)^\frac{1}{s} \fint_{B_{r/4}} |Du|\, dx,
\end{equation*}
which yields the desired result by ${\kappa_2} = \frac{\kappa_1}{2s}$.
\end{proof}

In the borderline setting, we additionally assume that $a(\cdot)$ is H\"older continuous with the gap condition \eqref{gap_2}. In this way, we are to give a new comparison estimate between $Du$ and $Dv$.

\begin{lemma}\label{lem_comp_2}
For ${\varphi_1}, {\varphi_2} \in \mathbf{\Phi}_{p,q}$ with ${\varphi_1} \prec  {\varphi_2}$ and $\mathbb{A}(\cdot,\cdot)$ satisfying \eqref{cond_A}, let $u \in W^{1,1}(\Omega,\mathbb{R}^N)$ be a local minimizer of the functional \eqref{main_funl}, and $v \in u + W_0^{1,1}\big(B_{r/8}, \mathbb{R}^N\big)$ be a minimizer of the local frozen functional \eqref{fun_v}. If $\mathbb{A}(\cdot,\cdot)$ is uniformly continuous and $a(\cdot)$ is H\"older continuous with  $\omega_a(r) = r^{\,\beta}$ satisfying the gap condition \eqref{gap_2}, then for any $B_r \Subset \Omega$  we have
\begin{equation*}
\fint_{B_{r/8}} \big|V_{\bar{\varPsi}_r}(Du) - V_{\bar{\varPsi}_r}(Dv)\big|^2\, dx \leq c_3\,\bigg(\omega_{\mathbb{A}}^{\,\mu}\big(r + \mathcal{E}(r)\big) + r^\lambda\Big(\mathcal{G}(r^\frac{1}{1+\varepsilon}) + \mathcal{G}(r) + 1\Big) \bigg)^{\kappa_3}\,\bar{\varPsi}_r\,\bigg(\fint_{B_{r/4}} |Du|\, dx\bigg)
\end{equation*}
where $c_3 = c_3\,\big(\texttt{data}_0\big) \geq 1$, $\mu = \min\{\frac{\varepsilon}{1+\varepsilon},\frac{1}{2}\}$, $\lambda = \min\big\{\frac{\beta\varepsilon}{1+\varepsilon}, \frac{n\varepsilon(\varepsilon_0-\varepsilon)}{(1+\varepsilon)(1+\varepsilon_0)}\big\} > 0$, ${\kappa_3} = \frac{\varepsilon p}{q + 1+ \varepsilon p}$ for $\varepsilon <\varepsilon_0$ with $\varepsilon_0$ as in Lemma \ref{lem_high_int}, and $\mathcal{G}(r)$ as shown in \eqref{Gr}.
\end{lemma}

\begin{proof}
We follow the same strategy as in the proof of Lemma \ref{lem_comp} with only minor modifications based on the argument of harmonic approximation. To this end, it suffices to verify that for any $\eta \in W^{1,\infty}_0(B_{r/8}, \mathbb{R}^N)$ with
\begin{equation}\label{u_al_min_2}
\fint_{B_{r/8}} \bar{\varPsi}_r\big(|Du|_{\bar{\mathbb{A}}_r^u}\big)\, dx   \leq   \fint_{B_{r/8}} \bar{\varPsi}_r\big(|Du + D\eta|_{\bar{\mathbb{A}}_r^u}\big)\, dx  +  c\, {\delta}_3(r) E(u, r)\,\bigg(\frac{\|D\eta\|_{L^\infty(B_{r/8})}}{\bar{\varPsi}^{-1}_r\big(E(u, r)\big)}+1\bigg)^{q+1},
\end{equation}
where ${\delta}_3(r) = \omega_{\mathbb{A}}^{\,\mu}\big(r + \mathcal{E}(r)\big) + r^\lambda \big(\mathcal{G}(r^\frac{1}{1+\varepsilon}) + \mathcal{G}(r) +  1\big)$ and $E(u, r)$ as in \eqref{def_E}. This leads to the desired result due to Lemma \ref{lem_app}.

To prove \eqref{u_al_min_2},  let us decompose it into the following:
\begin{equation}\label{est2_al_min}
\fint_{B_{r/8}} \bar{\varPsi}_r\big(|Du|_{\bar{\mathbb{A}}_r^u}\big)\, dx
\leq    \fint_{B_{r/8}} \big|\bar{\varPsi}_r\big(|Du|_{\bar{\mathbb{A}}_r^u}\big) -  \bar{\varPsi}_r\big(|Du|_{\mathbb{A}^u}\big)\big|\, dx + \fint_{B_{r/8}} \varPsi\big(x, |Du|_{\mathbb{A}^u}\big)\, dx.
\end{equation}
The first term of the right-hand side is estimated as \eqref{est_al_min_ft}. Therefore, we only consider the second term of \eqref{est2_al_min}. By the local minimality of $u$, then  for any $\eta \in W^{1,\infty}_0(B_{r/8}, \mathbb{R}^N)$  we have
\begin{equation*}
\fint_{B_{r/8}} \varPsi\big(x, |Du|_{\mathbb{A}^u}\big)\, dx \leq \fint_{B_{r/8}} \varPsi\big(x, |Du + D\eta|_{\mathbb{A}^{u+\eta}}\big)\, dx.
\end{equation*}
Let $w = u + \eta$ and write
\begin{equation*}
\mathcal{D}_1 = {\Big\{x \in B_{r/8} : |Dw| \leq \varphi_1^{-1}\big(r^\frac{-n}{1+\varepsilon}\big)\Big\}},
\qquad
\mathcal{D}_2 = {\Big\{x \in B_{r/8} : |Dw| > \varphi_1^{-1}\big(r^\frac{-n}{1+\varepsilon}\big)\Big\}};
\end{equation*}
then we get
\begin{eqnarray}\label{est2_alm_min}
\fint_{B_{r/8}} \varPsi\big(x, |Du|_{\mathbb{A}^u}\big)\, dx
&\leq& \frac{1}{|B_{r/8}|}\int_{\mathcal{D}_1} \varPsi\big(x, |Dw|_{\mathbb{A}^w}\big)\, dx + \frac{1}{|B_{r/8}|}\int_{\mathcal{D}_2} \varPsi\big(x, |Dw|_{\mathbb{A}^w}\big)\, dx \notag \\[6pt]
&:=& H_1 + H_2.
\end{eqnarray}

By direct calculation, we immediately show $H_{1}$ as follows:
\begin{eqnarray}\label{est_P1_1}
H_1
&\leq&  \frac{1}{|B_{r/8}|}\int_{\mathcal{D}_1} \big|{\varPsi}\big(x,
|Dw|_{\mathbb{A}^w}\big)-\bar{\varPsi}_r\big(|Dw|_{\mathbb{A}^w}\big)\big|\, dx +  \fint_{B_{r/8}} \big|\bar{\varPsi}_r\big(|Dw|_{\mathbb{A}^w}\big)-\bar{\varPsi}_r\big(|Dw|_{\mathbb{A}^u}\big)\big|\, dx \notag\\
&&+\,  \fint_{B_{r/8}}
\big|\bar{\varPsi}_r\big(|Dw|_{\mathbb{A}^u}\big)-\bar{\varPsi}_r\big(|Dw|_{\bar{\mathbb{A}}_r^u}\big)\big|\, dx + \fint_{B_{r/8}} \bar{\varPsi}_r\big(|Dw|_{\bar{\mathbb{A}}_r^u}\big)\, dx\notag\\
&:=& P_1 + P_2 + P_3 +  \fint_{B_{r/8}} \bar{\varPsi}_r\big(|Du +
D\eta|_{\bar{\mathbb{A}}_r^u}\big)\,dx.
\end{eqnarray}
Here, the estimates of $P_2$ and $P_3$ are just the same as in \eqref{est_I1_st} and \eqref{est_I1_tt}. In the sequel, it suffices to focus on the estimate of $P_1$. Obviously, $|Dw| \leq \varphi_1^{-1}(r^{-\frac{n}{1+\varepsilon}})$ implies $\frac{{\varphi_2}(|Dw|)}{{\varphi_1}(|Dw|)} \leq \frac{{\varphi_2}(\varphi_1^{-1}(r^{-\frac{n}{1+\varepsilon}}))}{{\varphi_1}(\varphi_1^{-1}(r^{-\frac{n}{1+\varepsilon}}))} = \frac{{\varphi_2}\circ \varphi_1^{-1}(r^{-\frac{n}{1+\varepsilon}})}{r^{-\frac{n}{1+\varepsilon}}}$ due to the non-decreasing property of $t \mapsto \frac{\varphi_2(t)}{\varphi_1(t)}$, and we obtain the following facts:

\begin{eqnarray*}
&&\int_{\mathcal{D}_1} \big|{\varPsi}\big(x, |Dw|_{\mathbb{A}^w}\big)-\bar{\varPsi}_r\big(|Dw|_{\mathbb{A}^w}\big)\big|\, dx\notag\\
&\leq& \int_{\mathcal{D}_1} |a(x) - \bar{a}_r| \,{\varphi_2}\big(|Dw|_{\mathbb{A}^w}\big)\, dx \notag\\
&\leq&  c\int_{\mathcal{D}_1} \omega_a(r) \,\frac{{\varphi_2}\big(|Dw|\big)}{{\varphi_1}\big(|Dw|\big)} \,{\varphi_1}\big(|Dw|\big)\,  dx \notag\\
&\leq&  c\int_{\mathcal{D}_1} \omega_a(r) \,\frac{{\varphi_2} \circ \varphi_1^{-1}\big(r^{-\frac{n}{1+\varepsilon}}\big)}{r^{-\frac{n}{1+\varepsilon}}} \,{\varphi_1}\big(|Dw|\big)\,  dx.
\end{eqnarray*}
Recalling  $a(\cdot)$ is H\"older continuous with  $\omega_a(r) = r^{\,\beta}$, it yields $\omega_a(r) \,\frac{{\varphi_2} \circ \varphi_1^{-1}\big(r^{-\frac{n}{1+\varepsilon}}\big)}{r^{-\frac{n}{1+\varepsilon}}}
=\frac{\omega_a(r)}{\omega_a(r^\frac{1}{1+\varepsilon})}\mathcal{G}(r^\frac{1}{1+\varepsilon})
=r^\frac{\beta\varepsilon}{1+\varepsilon}\mathcal{G}(r^\frac{1}{1+\varepsilon})$. By ${\varphi_1}(|Dw|) \leq \bar{\varPsi}_r(|Dw|)\lesssim \bar{\varPsi}_r\big(|Du|\big) + \bar{\varPsi}_r\big(|D\eta|\big)$, we infer
$$
\int_{\mathcal{D}_1} \big|{\varPsi}\big(x, |Dw|_{\mathbb{A}^w}\big)-\bar{\varPsi}_r\big(|Dw|_{\mathbb{A}^w}\big)\big|\, dx
\le c\,r^\frac{\beta\varepsilon}{1+\varepsilon} \,\mathcal{G}(r^\frac{1}{1+\varepsilon}) \,\int_{B_{r/8}} \Big(\bar{\varPsi}_r\big(|Du|\big) + \bar{\varPsi}_r\big(|D\eta|\big)\Big)\,  dx.
$$
Since $r^{\frac{\beta\varepsilon}{1+\varepsilon}}\mathcal{G}(r^\frac{1}{1+\varepsilon})\le r^\lambda\mathcal{G}(r^\frac{1}{1+\varepsilon})\leq {\delta}_3(r)$, by
$\fint_{B_{r/8}} \bar{\varPsi}_r\big(|Du|\big)\, dx \leq  E(u,r)$ and $\fint_{B_{r/8}} \,\bar{\varPsi}_r\big(|D\eta|\big)\,  dx \leq \, E(u, r)\,\Big(\frac{\|D\eta\|_{L^\infty(B_{r/8})}}{\bar{\varPsi}^{-1}_r\big(E(u, r)\big)}+1\Big)^{q}$ in \eqref{est_Deta}, we deduce that
\begin{eqnarray*}
P_1
\leq  c\,{\delta}_3(r)\,E(u, r)\,\bigg(\frac{\|D\eta\|_{L^\infty(B_{r/8})}}{\bar{\varPsi}^{-1}_r\big(E(u, r)\big)}+1\bigg)^{q}.
\end{eqnarray*}
Putting the estimates of $P_2,P_3$ as in \eqref{est_I1_st}\,\eqref{est_I1_tt} and $P_1$ into \eqref{est_P1_1}, it yields
\begin{equation}\label{est_P1}
H_1 \leq  c\,{\delta}_3(r)\,E(u, r)\,\bigg(\frac{\|D\eta\|_{L^\infty(B_{r/8})}}{\bar{\varPsi}^{-1}_r\big(E(u, r)\big)}+1\bigg)^{q+1} + \fint_{B_{r/8}} \bar{\varPsi}_r\big(|Dw|_{\bar{\mathbb{A}}_r^u}\big)\, dx.
\end{equation}

For $H_2$, we use a similar way as in \eqref{est_I2_1}  to obtain
\begin{eqnarray}\label{est_P2_1}
H_2
\leq  c\,r^\frac{n\varepsilon}{1+\varepsilon}\fint_{B_{r/8}} \bigg(\Big({\varPsi}\big(x, |Du|\big)\Big)^{1+\varepsilon} +  \Big({\varPsi}\big(x, |D\eta|\big)\Big)^{1+\varepsilon}\bigg)\,\,dx.
\end{eqnarray}
By considering the higher integrability of $Du$ as in \eqref{ineq_rev}, it yields that
\begin{eqnarray*}
\fint_{B_{r/8}} \Big({\varPsi}\big(x, |Du|\big)\Big)^{1+\varepsilon} dx
&\leq&  c\, \bigg(\fint_{B_{r/4}} \Big({\varPsi}\big(x, |Du|\big)\Big)^{1+\varepsilon_0} dx\bigg)^\frac{\varepsilon}{1+\varepsilon_0}E(u, r).
\end{eqnarray*}
Therefore,
\begin{eqnarray}\label{est_P2_ft}
r^\frac{n\varepsilon}{1+\varepsilon}\fint_{B_{r/8}} \Big({\varPsi}\big(x, |Du|\big)\Big)^{1+\varepsilon} dx
&\leq& c\,r^\frac{n\varepsilon}{1+\varepsilon} \bigg(\fint_{B_{r/4}} \Big({\varPsi}\big(x, |Du|\big)\Big)^{1+\varepsilon_0} dx\bigg)^\frac{\varepsilon}{1+\varepsilon_0}E(u, r)\notag\\
&=&   c\,r^{\frac{n\varepsilon(\varepsilon_0-\varepsilon)}{(1+\varepsilon)(1+\varepsilon_0)}}\,
\|{\varPsi}\big(x,|Du|\big)\|^\varepsilon_{L^{1+\varepsilon_0}(B_{r/4})}E(u, r)\notag\\[6pt]
&\leq&  c\,{\delta}_3(r)\,{M}^\varepsilon E(u, r)\,,
\end{eqnarray}
where we used $r^{\frac{n\varepsilon(\varepsilon_0-\varepsilon)}{(1+\varepsilon)(1+\varepsilon_0)}} \leq r^\lambda \leq {\delta}_3(r)$ for $\lambda = \min\big\{\frac{\beta\varepsilon}{1+\varepsilon}, \frac{n\varepsilon(\varepsilon_0-\varepsilon)}{(1+\varepsilon)(1+\varepsilon_0)}\big\} > 0$ and $\|{\varPsi}\big(x,|Du|\big)\|_{L^{1+\varepsilon_0}(B_{r/4})} \leq {M}$ in \eqref{high_norm} as in the last inequality.

On the other hand, we use the same way as in \eqref{est_I2_st} and $\big(E(u, r)\big)^\varepsilon \leq c r^{\frac{-n\varepsilon}{1+\varepsilon_0}} \|\varPsi(x, |Du|)\|^\varepsilon_{L^{1+\varepsilon_0}(B_{r/8})} $ to obtain that
\begin{eqnarray}\label{est_P2_st}
r^\frac{n\varepsilon}{1+\varepsilon}\fint_{B_{r/8}} \Big({\varPsi}\big(x, |D\eta|\big)\Big)^{1+\varepsilon} dx
&\leq&  c\,r^{\frac{n\varepsilon(\varepsilon_0-\varepsilon)}{(1+\varepsilon)(1+\varepsilon_0)}}\Big(1+\mathcal{G}(r)\Big)^{1 +\varepsilon}\, E(u, r)\,\bigg(\frac{\|D\eta\|_{L^\infty(B_{r/8})}}{\bar{\varPsi}^{-1}_r\big(E(u, r)\big)}\bigg)^{q+1}\notag\\
&\leq&  c\,{\delta}_3(r)\,E(u, r)\,\bigg(\frac{\|D\eta\|_{L^\infty(B_{r/8})}}{\bar{\varPsi}^{-1}_r\big(E(u, r)\big)}\bigg)^{q+1},
\end{eqnarray}
where  we used $r^{\frac{n\varepsilon(\varepsilon_0-\varepsilon)}{(1+\varepsilon)(1+\varepsilon_0)}} \leq {\delta}_3(r)$ in the last step.
Putting \eqref{est_P2_ft}\,\eqref{est_P2_st} into \eqref{est_P2_1} yields
\begin{equation}\label{est_P2}
H_2 \leq c\,{\delta}_3(r)\,E(u, r)\,\bigg(\frac{\|D\eta\|_{L^\infty(B_{r/8})}}{\bar{\varPsi}^{-1}_r\big(E(u, r)\big)}+1\bigg)^{q+1}.
\end{equation}
Let us replace \eqref{est2_alm_min} by the estimates of $H_1,H_2$ shown in \eqref{est_P1} \eqref{est_P2}, to get
\begin{equation*}
\fint_{B_{r/8}} \varPsi\big(x, |Du|_{\mathbb{A}^u}\big)\, dx \leq \fint_{B_{r/8}} \bar{\varPsi}_r\big(|Du + D\eta|_{\bar{\mathbb{A}}_r^u}\big)\, dx +  c\,{\delta}_3(r) \,E(u, r)\,\bigg(\frac{\|D\eta\|_{L^\infty(B_{r/8})}}{\bar{\varPsi}^{-1}_r\big(E(u, r)\big)}+1\bigg)^{q+1},
\end{equation*}
which together with \eqref{est_al_min_ft}  inserts into \eqref{est2_al_min}, to conclude that $u$ is an almost minimizer of the local frozen functional $\int_{B_{r/8}} \bar{\varPsi}_r(|Dw|_{\bar{\mathbb{A}}_r^u})\,dx$ in the sense of \eqref{u_al_min_2}. This completes the proof.
\end{proof}

The following  indicates the comparison estimate in $L^1$-sense while $a(x)$ is H\"older continuous with the gap condition \eqref{gap_2}. This follows from the same reasoning as in the proof of Corollary \ref{cor_comp_1}, with ${\delta}_2(r)$ only replaced by ${\delta}_4(r)=\Big(\omega_{\mathbb{A}}^{\,\mu}(r+\mathcal{E}(r))+r^{\lambda}\big(\mathcal{G}(r^\frac{1}{1+\varepsilon}) +\mathcal{G}(r)+1\big)
\Big)^{\kappa_3/2}$ in Lemma \ref{lem_comp_2}, and we omit its proof.

\begin{corollary}\label{cor_comp_2}
Under the same assumptions as Lemma \ref{lem_comp_2}, then for any $B_r \Subset \Omega$ we have
\begin{equation}\label{est_comp_2}
\int_{B_{r/8}} \big|Du - Dv\big|\, dx \leq c_4\,\bigg(\omega_{\mathbb{A}}^{\,\mu}\big(r + \mathcal{E}(r)\big) + r^\lambda\Big(\mathcal{G}(r^\frac{1}{1+\varepsilon}) + \mathcal{G}(r) + 1\Big)\bigg)^{\kappa_4}\int_{B_{r/4}} |Du|\, dx  ,
\end{equation}
where $c_4 = c_4\,\big(\texttt{data}_0\big) \geq 1$, $\mu = \min\{\frac{\varepsilon}{1+\varepsilon},\frac{1}{2}\}$, $\lambda = \min\big\{\frac{\beta\varepsilon}{1+\varepsilon}, \frac{n\varepsilon(\varepsilon_0-\varepsilon)}{(1+\varepsilon)(1+\varepsilon_0)}\big\} > 0$, $\kappa_4 = \kappa_4(p, q,\varepsilon) \in (0,1)$ for $\varepsilon < \varepsilon_0$ with $\varepsilon_0$ in Lemma \ref{lem_high_int}, and  $\mathcal{G}(r)$ as shown in \eqref{Gr}.
\end{corollary}

\section{Proof of main results}
\setcounter{equation}{0}
\setcounter{theorem}{0}

In this section, we prove the main results of this paper concerning the partial H{\"o}lder continuity  as well as the partial H{\"o}lder continuity of the gradients to the minimizer functional \eqref{main_funl}.

\subsection{Partial H\"older regularity}

This subsection is devoted to proving that $u \in C^{0,\alpha}(\Omega_0, \mathbb{R}^N)$ for any $\alpha \in (0,1)$ to the local minimizers of \eqref{main_funl}, where $\Omega_0 \subset \Omega$ is an open subset of $\Omega$ with a lower dimensional measure. First of all, we show the Morrey-type estimate of $Du$ to the local minimizers by assuming sufficiently small excess energy.

\begin{lemma}\label{lem_morrey}
For ${\varphi_1}, {\varphi_2} \in \mathbf{\Phi}_{p,q}$ with ${\varphi_1} \prec  {\varphi_2}$, let $u \in W^{1,1}\big(\Omega, \mathbb{R}^N\big)$ be a local minimizer of the functional \eqref{main_funl} with assumption that the modulus of continuity of $\mathbb{A}(\cdot,\cdot)$ with \eqref{cond_A} and $a(\cdot)$ shown as in \eqref{cond_continuous} satisfies the gap condition \eqref{gap_1}. For given $\tau \in (0, 1)$, there exist an index $\delta_0 = \delta_0(\texttt{data}_0, \tau) > 0$ and a radius $R_1 = R_1(\texttt{data}_0, \tau) \in (0,1)$, such that for $B_{r} \ \Subset \Omega$ with $r \in (0, R_1)$ there holds
\begin{equation*}
\mathcal{E}(r) = \fint_{B_{r}} |u - \langle u \rangle_{B_{r}}| \, dx < \delta_0;
\end{equation*}
then we have the  Morrey-type decay estimate
\begin{equation}\label{morrey}
\int_{B_{\rho}} |Du| \, dx \leq c \,\bigg(\frac{\rho}{r}\bigg)^{n - \tau}\int_{B_{r}} |Du| \, dx
\end{equation}
for any $\rho \in (0, r]$, where $c = c(\texttt{data}_0, \tau) \geq 1$.
\end{lemma}

\begin{proof} Denote
\begin{equation}\label{def_theta}
\theta := \min\Bigg\{\frac{1}{32},\bigg(\frac{1}{8^n c_0c_7}\bigg)^\frac{1}{\tau}, \bigg(\frac{1}{4^n c_5 c_6}\bigg)^\frac{1}{1-\tau} \Bigg\},
\end{equation}
where $c_0, c_5, c_7$ are as in \eqref{lip_v} \eqref{cacc_L_1} \eqref{energy_v_L_1}, and $c_6$ is the constant from Poincar\'e inequality in \eqref{poincare}.

Note that $\omega_\mathbb{A}(\cdot)$ is the modulus of continuity of $\mathbb{A}(x,u)$ with $\omega_\mathbb{A}(0)=0$, and $\limsup_{r \to 0^+}\mathcal{G}(r) = 0$ by the gap condition \eqref{gap_1}. Then we can pick up $R_1 \in (0,R_0)$ and $\delta_0 > 0$  small enough that for any $r \in (0, R_1)$ it holds
\begin{equation}\label{R_2}
\varUpsilon(r, \delta_0) := c_2\,\Big(\omega_{\mathbb{A}}^{\,\mu}\big(r + \delta_0\big) + \mathcal{G}(r) + r^\frac{n\varepsilon^2}{1+\varepsilon}\Big)^{\kappa_2} \leq \frac{\theta^{n-\tau}}{10^{\,n}}.
\end{equation}
We are next to prove a decay estimate under assumption of the small excess energy $\mathcal{E}(r) < \delta_0$ for $r \in (0, R_1)$. It is readily checked that
\begin{eqnarray}\label{thrta_r}
\int_{B_{\theta r}} |Du| \, dx
&\leq&  \int_{B_{\theta r}} |Du - Dv| \, dx + \int_{B_{\theta r}} |Dv| \, dx\notag\\
&\leq& \int_{B_{r/32}} |Du - Dv| \, dx + |B_{\theta r}|\,\sup_{x \in B_{r/32}} |Dv|.
\end{eqnarray}
By the comparison estimate \eqref{est_comp_1} and \eqref{R_2} we get that
\begin{eqnarray}\label{u-v}
\int_{B_{r/32}} |Du - Dv| \, dx
&\le& \varUpsilon\big(r, \mathcal{E}(r)\big)\int_{B_{r/4}} |Du| \, dx \notag\\
&\leq&  \varUpsilon\big(r, \delta_0 \big)\int_{B_{r/4}} |Du| \, dx\notag\\
&\leq& \frac{\theta^{\,n-\tau}}{2}\int_{B_{r/4}} |Du| \, dx.
\end{eqnarray}
On the other hand, by the Lipschitz estimate  in \eqref{lip_v}, $L^1$-estimate  in \eqref{energy_v_L_1} and $\theta \leq \Big(\frac{1}{8^n c_7c_0}\Big)^\frac{1}{\tau}$ from \eqref{def_theta} it follows that
\begin{eqnarray}\label{sup_v}
|B_{\theta r}|\,\sup_{x \in B_{r/32}} |Dv|
&\leq& |B_{\theta r}|\,c_7\,\fint_{B_{r/8}} |Dv| \, dx \notag\\
&\leq& (\theta\,r)^{\,n}\,c_7\,c_0\,\fint_{B_{r/4}} |Du| \, dx \notag\\
&\leq& \frac{\theta^{\,n-\tau}}{2}\int_{B_{r/4}} |Du| \, dx.
\end{eqnarray}
Putting \eqref{u-v} and \eqref{sup_v} into \eqref{thrta_r} we get the following implication relation:
\begin{equation}\label{result_r}
\mathcal{E}(r) < \delta_0 \implies
\int_{B_{\theta r}} |Du| \, dx
\leq \theta^{\,n-\tau}\int_{B_{r/4}} |Du| \, dx
\quad \text{for any } r \in (0, R_1).
\end{equation}

Next, we claim that if $\mathcal{E}(r) < \delta_0$ for a fixed $r \in (0, R_1)$, then we get
\begin{equation}\label{est_excess}
\mathcal{E}(\theta^{k-1}r) = \fint_{B_{\theta^{k-1}r}} |u - \langle u \rangle_{B_{\theta^{k-1}r}}| \, dx < \delta_0
\quad
\text{for every } k \in \mathbb{N}.
\end{equation}
It is immediately observed that $\mathcal{E}(r) < \delta_0$ implies \eqref{est_excess} for $k = 1$. Assume by induction that $\mathcal{E}(\theta^{\,\bar{k}-1}r) < \delta_0$ for some $\bar{k} \in \mathbb{N}$. In the following, we show the validity of $\mathcal{E}(\theta^{\,\bar{k}}r) < \delta_0$. In accordance with \eqref{result_r} we see that $\int_{B_{\theta^{\,\bar{k}}r}} |Du| \, dx \leq \theta^{\,n-\tau}\int_{B_{\theta^{\bar{k}-1}r/4}} |Du| \, dx$ on the basis of induction hypothesis  $\mathcal{E}(\theta^{\,\bar{k}-1}r) < \delta_0$. This combines with the Poincar{\'e} inequality to deduce that
\begin{eqnarray*}
\mathcal{E}(\theta^{\,\bar{k}}r)
&\equiv & \fint_{B_{\theta^{\,\bar{k}}r}} |u - \langle u \rangle_{B_{\theta^{\,\bar{k}}r}}| \, dx\notag\\
&\leq&  c_6\,\theta^{\,\bar{k}}r \fint_{B_{\theta^{\,\bar{k}} r}} |Du| \, dx\notag\\
&\leq& c_6\,\Big(\theta^{\,\bar{k}}r\Big)^{1-n}\theta^{\,(n-\tau)}\int_{B_{\theta^{\bar{k}-1}r/4}} |Du| \, dx.
\end{eqnarray*}
We use the Caccioppoli inequality \eqref{cacc_L_1}, $\theta \leq \Big(\frac{1}{4^n c_6 c_5}\Big)^\frac{1}{1-\tau}$ from \eqref{def_theta} and $\mathcal{E}(\theta^{\,\bar{k}-1}r) < \delta_0$, to  get that
\begin{eqnarray*}
\mathcal{E}(\theta^{\,\bar{k}}r)
&\leq&  4^n\,c_6\,c_5\,\theta^{\,(1-\tau)}\fint_{B_{\theta^{\bar{k}-1}r}} |u - \langle u \rangle_{B_{\theta^{\bar{k}-1}r}}| \, dx \notag\\
&\leq&  4^n\,c_6\,c_5\,\bigg(\frac{1}{4^n c_6 c_5}\bigg)\,\mathcal{E}(\theta^{\,\bar{k}-1}r) < \delta_0,
\end{eqnarray*}
which completes the proof of Claim \eqref{est_excess}.
In the sequel, we make use of \eqref{result_r} to get
\begin{equation*}
\int_{B_{\theta^k r}} |Du| \, dx
\leq \theta^{n-\tau}\int_{B_{\theta^{k-1}r}} |Du| \, dx
\quad \text{for any } k \in \mathbb{N}.
\end{equation*}
We use the iteration approach again to deduce
\begin{eqnarray}\label{theta-k}
\int_{B_{\theta^k r}} |Du| \, dx
&\leq& \theta^{\,k(n-\tau)}\int_{B_{r}} |Du| \, dx
\quad \text{for any } k \in \mathbb{N}.
\end{eqnarray}

Finally, we prove the Morrey-type decay estimate of $Du$. For arbitrary $\rho \in (0,r)$, if $\rho \in (0, \theta r]$ then there is a unique $\bar{k} \in \mathbb{N}$ such that $\theta^{\,\bar{k}+1} r < \rho \leq \theta^{\,\bar{k}} r$. By \eqref{theta-k} we have that
\begin{eqnarray*}
\int_{B_{\rho}} |Du| \, dx
&\leq& \int_{B_{\theta^{\bar{k}} r}} |Du| \, dx
\leq \theta^{\,\bar{k}(n-\tau)}\int_{B_{r}} |Du| \, dx  \notag\\
&\leq& \theta^{-(n - \tau)}\theta^{\,(\bar{k} + 1)(n-\tau)}\int_{B_{r}} |Du| \, dx \notag\\
&\leq& c_\theta\, \bigg(\frac{\rho}{r}\bigg)^{\,n-\tau}\int_{B_{r}} |Du| \, dx
\end{eqnarray*}
for $c_\theta = \theta^{-(n - \tau)}> 0$. If $\rho \in (\theta r, r]$, by the standard calculation it yields that
\begin{eqnarray*}
\int_{B_{\rho}} |Du| \, dx
&=&
\bigg(\frac{r}{\rho}\bigg)^{\,n-\tau}\bigg(\frac{\rho}{r}\bigg)^{\,n-\tau}\int_{B_{\rho}} |Du| \, dx \notag\\
&\leq& \bigg(\frac{1}{\theta}\bigg)^{\,n-\tau}\bigg(\frac{\rho}{r}\bigg)^{\,n-\tau}\int_{B_{r}} |Du| \, dx \\
&=& c_\theta\, \bigg(\frac{\rho}{r}\bigg)^{\,n-\tau}\int_{B_{r}} |Du| \, dx.
\end{eqnarray*}
In summary, we combine $\rho \in (0, \theta r]$ and $\rho \in (\theta r, r]$, to get the Morrey-type decay estimate \eqref{morrey}.
\end{proof}

\begin{proof}[Proof of Theorem \ref{thm_1}]
With $\delta_0$ as in Lemma \ref{lem_morrey}, let us define an open subset $\Omega_0\subset \Omega$ as follows:
\begin{equation}\label{Omega_0}
\Omega_0:=\left\{
y \in \Omega:
\liminf _{r \rightarrow 0^{+}} \fint_{B_r(y)} |u - \langle u \rangle_{B_r(y)}| \,dx
<  \delta_0
\right\}.
\end{equation}
For any $y \in \Omega' \Subset \Omega_0$, we can take
$0 < r < \min\Big\{R_1, \frac{\mathrm{dist}(\Omega', \Omega)}{4}\Big\}$
such that
\begin{equation*}
\mathcal{E}(y, r) = \fint_{B_{r}(y)} |u - \langle u \rangle_{B_{r}(y)}| \, dx < \delta_0.
\end{equation*}
By the absolute continuity of integral with respect to the center $y$, then there exists a small  $0<\sigma\ll r$ such that for any $z \in B_{\sigma}(y)$, we have
\begin{equation*}
\mathcal{E}(z, r) = \fint_{B_{r}(z)} |u - \langle u \rangle_{B_{r}(z)}| \, dx < \delta_0.
\end{equation*}
According to Lemma \ref{lem_morrey}, we get the following Morrey-type decay estimate
\begin{equation}\label{morrey_z}
\int_{B_{\rho}(z)} |Du| \, dx \leq c \,\bigg(\frac{\rho}{r}\bigg)^{n - \tau}\int_{B_{r}(z)} |Du| \, dx
\end{equation}
for any $\rho \in (0, r]$ and $\tau \in (0, 1)$. This yields
$$
\frac{1}{\rho^{\,n - \tau}}\int_{B_{\rho}(z)} |Du| \, dx \leq  \frac{c}{r^{n - \tau}}\int_{B_{r}(z)} |Du| \, dx < \infty
$$
which indicates that $Du \in L^{1,\,n - \tau}(B_{\sigma}(y), \mathbb{R}^{N \times n})$ for any $\tau \in (0, 1)$.
Thanks to the Morrey embedding  (iii) of Lemma \ref{lem_isomorphism} we conclude  $u \in C^{0, 1-\tau}(B_{\sigma}(y), \mathbb{R}^N)$, that is to say, $u \in C^{0, 1-\tau}_{\mathrm{loc}}(\Omega_0, \mathbb{R}^N)$ for any $\tau \in (0, 1)$ by the arbitrariness of $y \in \Omega' \Subset \Omega_0$.

In the sequel, we are to estimate the Hausdorff dimension of the singular set $\Omega\setminus\Omega_0$. To this end, by the Poincar{\'e} inequality, the properties  \eqref{est_<1}\eqref{est_>1}, and the Jensen inequality of convex function $\varphi_1(\cdot)$, we deduce that for $r \in (0,1)$ there holds
\begin{eqnarray*}
\varphi_1\bigg(\fint_{B_r(y)} |u - \langle u \rangle_{B_r(y)}| \,dx\bigg)
&\leq& \varphi_1\bigg(c\,r \fint_{B_r(y)} \,|Du| \, dx\bigg)\notag\\
&\leq&  c\,r^{\,p} {\varphi_1}\bigg(\fint_{B_r(y)} |Du| \, dx\bigg) \notag\\
&\leq&  c\,r^{\,p} \fint_{B_r(y)} {\varphi_1}\,\big(|Du|\big) \, dx.
\end{eqnarray*}
Employing $\varphi_1(\cdot) \leq \varPsi(x, \cdot)$ and the H\"older inequality, one gets that
\begin{eqnarray*}
\varphi_1\bigg(\fint_{B_r(y)} |u - \langle u \rangle_{B_r(y)}| \,dx\bigg)
&\leq&  c\, r^{\,p}\fint_{B_r(y)} \varPsi\big(x, |Du|\big) \, dx\notag\\
&\leq&  c\, r^{\,p}\bigg(\fint_{B_r(y)} \Big(\varPsi\big(x, |Du|\big)\Big)^{1+\varepsilon_0}  \, dx\bigg)^{\frac{1}{1+\varepsilon_0}}\notag\\
&=& c\, \omega_n^{-\frac 1 {1+\varepsilon_0}}\bigg(r^{-n+p(1+\varepsilon_0)}\int_{B_r(y)} \Big(\varPsi\big(x, |Du|\big)\Big)^{1+\varepsilon_0} \, dx\bigg)^{\frac{1}{1+\varepsilon_0}}.
\end{eqnarray*}
We see that $ \fint_{B_r(y)} |u - \langle u \rangle_{B_r(y)}| \,dx \ge \delta_0 $ yields $\ \varphi_1\Big(\fint_{B_r(y)} |u - \langle u \rangle_{B_r(y)}| \,dx\Big) \ge \varphi_1(\delta_0)$ due to the non-decreasing nature of $\varphi_1(t)$. Therefore, we deduce
$$
c \,\omega_n^{-\frac 1 {1+\varepsilon_0}}\bigg(r^{-n+p(1+\varepsilon_0)}\int_{B_r(y)} \Big(\varPsi\big(x, |Du|\big)\Big)^{1+\varepsilon_0} \, dx\bigg)^{\frac{1}{1+\varepsilon_0}}\ge \varphi_1(\delta_0)
$$
or
$$
r^{-n+p(1+\varepsilon_0)}\int_{B_r(y)} \Big(\varPsi\big(x, |Du|\big)\Big)^{1+\varepsilon_0} \, dx \ge \delta^*.
$$
with $\delta^*=\frac{\omega_n}{c} \big(\varphi_1(\delta_0)\big)^{1+\varepsilon_0}$.
Hence, for the singular set $\Omega\setminus\Omega_0$ we have the following conclusion:
\begin{eqnarray*}
\Omega \setminus \Omega_0 &=& \Big\{y \in \Omega: \liminf _{r \rightarrow 0^{+}} \fint_{B_r(y)} |u - \langle u \rangle_{B_r(y)}| \,dx \geq \delta_0\Big\} \notag\\
&\subset& \Big\{y \in \Omega: \liminf _{r \rightarrow 0^{+}} r^{-n+p(1+\varepsilon_0)}\int_{B_r(y)} \Big(\varPsi\big(x, |Du|\big)\Big)^{1+\varepsilon_0} \, dx \geq \delta^*\Big\} \notag\\
&\subset& \Big\{y \in \Omega: \liminf _{r \rightarrow 0^{+}} \frac{1}{r^{n-p-p\varepsilon_0}}\int_{B_r(y)} \Big(\varPsi\big(x, |Du|\big)\Big)^{1+\varepsilon_0} \, dx > 0\Big\}.
\end{eqnarray*}
Since $\big(\varPsi(x, |Du|)\big)^{1+\varepsilon_0} \in L_\mathrm{loc}^1(\Omega)$ by \eqref{high_norm}, we use  Lemma \ref{lem_Hausdorff} to obtain
\begin{equation*}
\mathcal{H}^{n - p-\epsilon}\big(\Omega \setminus \Omega_0\big) = 0
\end{equation*}
with some $\epsilon: = p\varepsilon_0 \in (0, n-p]$.  This completes the proof of Theorem \ref{thm_1}.
\end{proof}

\begin{proof}[Proof of Theorem \ref{thm_2}]
In this way, a key ingredient is to make a modification  to the regular conditions of $\mathcal{G}(r)$ and $a(x)$; more precisely, we  use $\limsup_{r \to 0^+}\mathcal{G}(r) < \infty$ instead of $\limsup_{r \to 0^+}\mathcal{G}(r) = 0$,  and $a(x) \in C^{0,\, \beta}(\Omega)$ instead of $a(x) \in C(\Omega)$. Corresponding to the result of Corollary \ref{cor_comp_2}, we define
\begin{equation*}
\varUpsilon\big(r, \mathcal{E}(r)\big) := c_4\,\bigg(\omega_{\mathbb{A}}^{\,\mu}\big(r + \mathcal{E}(r)\big) + r^\lambda\Big(\mathcal{G}(r^\frac{1}{1+\varepsilon}) + \mathcal{G}(r) + 1\Big) \bigg)^{\kappa_4}
\end{equation*}
with $\lambda = \min\big\{\frac{\beta\varepsilon}{1+\varepsilon}, \frac{n\varepsilon(\varepsilon_0-\varepsilon)}{(1+\varepsilon)(1+\varepsilon_0)}\big\} > 0$.
Let us take $\theta$ with the same as in \eqref{def_theta}. Furthermore, by considering $\omega_\mathbb{A}(\cdot)$ as the modulus of continuity of $\mathbb{A}(x,u)$ with $\omega_\mathbb{A}(0)=0$ and $\limsup_{r \to 0^+}\mathcal{G}(r) < \infty$, then we can pick up $R_2 \in (0,R_0)$ and $\delta_0 > 0$ such that
\begin{equation*}
\varUpsilon(r, \delta_0)\leq \frac{\theta^{n-\tau}}{8^n} \quad \text{for any } r \in (0, R_2).
\end{equation*}
Therefore, we employ the same way as in the proof of Lemma \ref{lem_morrey} only with \eqref{est_comp_1} replaced by \eqref{est_comp_2} to conclude the Morrey-type estimate holds for sufficiently small excess energy. The rest of the proof follows the same lines as the proof of Theorem \ref{thm_1} to get $Du \in L^{1,\,n - \tau}(B_{\sigma}(y), \mathbb{R}^{N \times n})$ for any $\tau \in (0, 1)$, which implies $u \in C^{0, 1-\tau}_{\mathrm{loc}}(\Omega_0, \mathbb{R}^N)$ for any $\tau \in (0, 1)$ by the Morrey embedding (iii) of Lemma \ref{lem_isomorphism}.
\end{proof}

\subsection{Partial H\"older regularity of gradients}

Note that the conditions of Theorem \ref{thm_3} are stronger than those of Theorem \ref{thm_2}, then for $y \in \Omega' \Subset \Omega_0$ defined as \eqref{Omega_0}, we have $Du \in L^{1,\,n - \tau}(B_{\sigma}(y), \mathbb{R}^{N \times n})$ with  $0 < \sigma < \frac{\mathrm{dist}(\Omega', \Omega)}{4}$ for any $\tau \in (0, 1)$. In other words, it holds that
\begin{equation}\label{morrey_2}
\int_{B_{r}(z)} |Du| \, dx \leq c \,r^{\,n - \tau}
\end{equation}
for any $z \in B_{\sigma}(y)$ and $r \in (0, \frac{\mathrm{dist}(\Omega', \Omega)}{8})$.
With \eqref{morrey_2} in hand, we are in a position to prove Theorem \ref{thm_3}.

\begin{proof}[Proof of Theorem \ref{thm_3}]
For fixed $0 < r < \min\big\{R_0, \frac{\mathrm{dist}(\Omega', \Omega)}{8}\big\}$, let $v \in W^{1,1}\big(B_{r/8}, \mathbb{R}^N\big)$ be the minimizer of the localized functional \eqref{fun_v}.
By the Poincar{\'e} inequality and \eqref{morrey_2} we see that
\begin{equation*}
\mathcal{E}(z, r) \equiv  \fint_{B_{r}(z)} |u - \langle u \rangle_{B_{r}}| \, dx \leq c\,r\,\fint_{B_{r}(z)} |Du| \, dx \leq c \,r^{\,1 - \tau}.
\end{equation*}
By the gap condition \eqref{R1} we see  $\mathcal{G}(r^\frac{1}{1+\varepsilon}) + \mathcal{G}(r)\le 2K$. Let us use Corollary \ref{cor_comp_2} with $\omega_\mathbb{A}(r) = r^{\,\beta} $, $\mathcal{E}(z, r) \leq  c \,r^{\,1 - \tau}$ to obtain that
\begin{eqnarray}\label{Du-Dv}
\fint_{B_{r/8}(z)} \big|Du - Dv\big|\, dx
&\leq& c\,\Big(\omega_{\mathbb{A}}^{\,\mu}\big(r + \mathcal{E}(z, r)\big) + r^\lambda \big(\mathcal{G}(r^\frac{1}{1+\varepsilon}) + \mathcal{G}(r) + 1\big) \Big)^{\kappa_4}\fint_{B_r(z)} |Du|\, dx \notag\\
&\leq& c\,\Big(r^{\beta\mu (1-\tau)} + r^\lambda \Big)^{\kappa_4}\fint_{B_r(z)} |Du|\, dx \notag\\
&\leq& c\,r^{\,\gamma'}\fint_{B_r(z)} |Du| \, dx
\end{eqnarray}
with an index $\gamma' = \kappa_4\min\big\{\beta\mu(1-\tau),\lambda\big\}  \in (0,1)$, which combines \eqref{Du-Dv} with \eqref{morrey_2} to show
\begin{equation}\label{est_comp_grad}
\fint_{B_{r/8}(z)} \big|Du - Dv\big|\, dx \leq c \,r^{\,\gamma' - \tau}.
\end{equation}

Let us now pick up an $R_3 > 0$ such that
\begin{equation*}
R_3 < \min \bigg\{R_0,\, \bigg(\frac{1}{8}\bigg)^{\frac{2n}{\gamma'}+1}, \frac{\mathrm{dist}(\Omega', \Omega)}{8}\bigg\}.
\end{equation*}
For any $\rho \in \big(0, R_3^\frac{2n + \gamma'}{2n}\big)$, we write $\rho^* = \rho^\frac{2n}{2n + \gamma'} \in \big(0, \min\big\{R_0, \frac{\mathrm{dist}(\Omega', \Omega)}{8}\big\}\big)$. Let $v \in u + W_0^{1,1}\big(B_{\rho^*/8}, \mathbb{R}^N\big)$ be a minimizer of the local frozen functional $ \int_{B_{\rho^*/8}} \bar{\varPsi}_{\rho^*}(|Dv|_{\bar{\mathbb{A}}_{\rho^*}^u}) \, dx$. It is checked that for $0<\rho <\rho^* < R_3$,
\begin{equation*}
\rho = \rho^\frac{2n}{2n + \gamma'} \rho^{1 - \frac{2n}{2n + \gamma'}} < \rho^*R_3^{1 - \frac{2n}{2n + \gamma'}} \leq \rho^*\bigg(\frac{1}{8}\bigg)^{\big(1 - \frac{2n}{2n + \gamma'}\big)\,\big(\frac{2n}{\gamma'}+1\big)} =  \frac{\rho^*}{8},
\end{equation*}
which means $B_\rho(z)\subset B_{\frac{\rho^*}{8}}(z)$. By the Campanato estimate of $Dv$  shown in \eqref{est_camp_v} and the comparison estimate \eqref{est_comp_grad}, we obtain that
\begin{eqnarray*}
\fint_{B_{\rho}(z)} |Du - \langle Du \rangle_{B_{\rho}}| \, dx
&\leq&  2\fint_{B_{\rho}(z)} |Du - \langle Dv \rangle_{B_{\rho}}| \, dx\notag\\
&\leq& 2\fint_{B_{\rho}(z)} |Du - Dv| \, dx + 2\fint_{B_{\rho}(z)} |Dv - \langle Dv \rangle_{B_{\rho}}| \, dx\notag\\
&\leq& c\,\bigg(\frac{\rho^*}{\rho}\bigg)^n\fint_{B_{\rho^*/8}(z)} |Du - Dv| \, dx + c\,\bigg(\frac{\rho}{\rho^*}\bigg)^{\,\sigma_1}\fint_{B_{\rho^*/8}(z)} |Dv| \, dx\notag\\
&\leq& c\,\bigg(\frac{\rho^*}{\rho}\bigg)^n(\rho^*)^{\,\gamma' - \tau}\, + c\,\bigg(\frac{\rho}{\rho^*}\bigg)^{\,\sigma_1}\fint_{B_{\rho^*/8}(z)} |Dv| \, dx.
\end{eqnarray*}
Since \eqref{energy_v_L_1} and \eqref{morrey_2} are in effect, one has
\begin{equation*}
\fint_{B_{\rho^*/8}(z)} |Dv| \, dx
\leq  c \fint_{B_{\rho^*/4}(z)} |Du| \, dx
\leq c \,(\rho^*)^{\,-\tau}.
\end{equation*}
Then it follows from $\rho^* = \rho^{\frac{2n}{2n + \gamma'}}$ that
\begin{eqnarray*}
\fint_{B_{\rho}(z)} |Du - \langle Du \rangle_{B_{\rho}}| \, dx
&\leq & c\,\Bigg((\rho^*)^{\,\gamma' - \tau}\,\bigg(\frac{\rho^*}{\rho}\bigg)^n + (\rho^*)^{\,-\tau}\bigg(\frac{\rho}{\rho^*}\bigg)^{\,\sigma_1}\Bigg)\\
&=&  c\,\bigg(\rho^{\frac{\gamma'(n-\sigma_1)}{2n+\gamma'}} + 1\bigg)\,\rho^{\frac{\gamma'\sigma_1-2n\tau}{2n+\gamma'}}\\
&\le & c\, \rho^{\frac{\gamma'\sigma_1-2n\tau}{2n+\gamma'}},
\end{eqnarray*}
where we used $\rho^{\frac{\gamma'(n-\sigma_1)}{2n+\gamma'}}\le 1$ due to $0<\rho<R_3<1$ and $n-\sigma_1>0$ in the last step.
By considering any $\tau \in (0,1)$, we only confine  $\tau \in (0, \frac {\gamma'\sigma_1}{2n})$ to ensure $\gamma:=\frac{\gamma'\sigma_1-2n\tau}{2n+\gamma'}>0$. Then we have
\begin{equation*}
\fint_{B_{\rho}(z)} |Du - \langle Du \rangle_{B_{\rho}}| \, dx
\leq \,c\,\rho^{\,\gamma} \quad \forall z \in B_{\sigma}(y).
\end{equation*}
This means that $Du \in \mathcal{L}^{1,\, n + \gamma}(B_{\sigma}(y), \mathbb{R}^{N \times n})$. On the basis of an equivalence of Campanato space and H\"{o}lder space shown in  Lemma \ref{lem_isomorphism} (ii), we see that $Du \in C^{0, \gamma}(B_{\sigma}(y), \mathbb{R}^{N \times n})$, then $Du \in C^{0, \gamma}_{\mathrm{loc}}(\Omega_0, \mathbb{R}^{N \times n})$ by the arbitrariness of $y \in \Omega' \Subset \Omega_0$. This proof is complete.
\end{proof}

\noindent\textbf{Acknowledgement}\ \
This paper was supported by National Natural Science Foundation of China grant No.\ 12071021 and No. 12271021.

\vspace{7pt}
\noindent\textbf{Data Availability}\ \
No data was used for the research described in the article.

\section*{Declarations}
\noindent\textbf{Conflicts of Interest}\ \
We declare that we have no conflict of interest.

\end{document}